\documentclass{amsart}

 \usepackage{amsmath}
\usepackage{amsfonts}
\usepackage{bbm}
\usepackage[dvips]{graphicx}
\usepackage{color}
\usepackage{float} 
\usepackage{fullpage}
\usepackage{epsfig} 

\usepackage{url}

\newtheorem{thm}{\bf{Theorem}}[section]
\newtheorem{lemma}[thm]{\bf{Lemma}}

\newtheorem{cor}[thm]{\bf{Corollary}}

\newtheorem{prop}[thm]{\bf{Proposition}}

\newtheorem{rem}[thm]{\bf{Remark}}

\begin{document}

\title[]{Inexact cuts for Deterministic and Stochastic Dual Dynamic Programming applied to convex nonlinear optimization problems}

%\date{Version from \today }%Received: date / Accepted: date}

\maketitle

\begin{center}
Vincent Guigues\\
School of Applied Mathematics, FGV\\
Praia de Botafogo, Rio de Janeiro, Brazil\\ 
{\tt vguigues@fgv.br}
\end{center}

\date{}

\begin{abstract} We introduce an extension of Dual Dynamic
Programming (DDP) to solve convex nonlinear dynamic programming equations.
We call this extension Inexact DDP (IDDP) which 
applies to situations where some or all primal and dual subproblems to be solved along the iterations
of the method
are solved with a bounded error (inexactly). We show that any accumulation point of the sequence
of decisions is an approximate solution to the dynamic programming equations.
When these errors tend to zero as the number of iterations goes to infinity,
we show that IDDP solves the dynamic programming equations.
We extend the analysis to stochastic convex nonlinear dynamic programming equations,
introducing Inexact Stochastic Dual Dynamic Programming (ISDDP), an inexact variant of SDDP 
corresponding to the situation where some or all problems to be solved in the forward and backward passes of SDDP 
are solved approximately. We also show the almost sure convergence of ISDDP for vanishing errors.\\
\end{abstract}

\par {\textbf{Keywords:} Stochastic programming; Inexact cuts for value functions; Bounding $\varepsilon$-optimal dual solutions; SDDP; Inexact SDDP.}\\

\par AMS subject classifications: 90C15, 90C90.

\section{Introduction}

Stochastic Dual Dynamic Programming (SDDP) is a sampling-based 
extension of the nested decomposition method \cite{birgemulti} to solve 
some $T$-stage stochastic programs, pioneered by \cite{pereira}.
Originally, in \cite{pereira}, it was presented to solve Multistage Stochastic
Linear Programs (MSLPs). Since many real-life applications in, e.g., finance and engineering,
can be modelled by such problems, until recently most papers on SDDP and related
decomposition methods, especially theory papers, focused on enhancements of the method for MSLPs.
These enhancements include risk-averse SDDP \cite{shapsddp}, \cite{guiguesrom12} \cite{guiguesrom10}, \cite{philpmatos},
\cite{kozmikmorton}, \cite{shaptekaya} and a convergence proof in \cite{philpot}.

However, SDDP can be applied to solve nonlinear stochastic convex dynamic programming equations. For such problems,
the convergence of the method was proved recently in \cite{lecphilgirar12} for risk-neutral problems, in \cite{guiguessiopt2016}
for risk-averse problems, and in \cite{guilejtekregsddp} for a regularized variant implemented on a nonlinear dynamic portfolio model with market impact costs.

To the best of our knowledge, all studies on SDDP
rely on the assumption that all primal and
dual subproblems solved in the forward and backward passes of the method are solved exactly. However, when these methods are applied to nonlinear problems,
only approximate solutions are available for the subproblems solved in the forward and backward passes of the algorithm.
In this context, the objective of this paper is to design variants of DDP
(the deterministic counterpart of SDDP) and SDDP 
to solve nonlinear convex dynamic programming equations that take this fact into account. 
We call the corresponding variants of DDP and SDDP Inexact DDP (IDDP) and Inexact SDDP (ISDDP). It should be mentioned, however, that
there is another motivation for considering inexact variants of DDP and SDDP. Indeed, it is known (see for instance the numerical experiments
in \cite{guiguesbandarra17}, \cite{guiguesejor2017}) that for the first iterations of the method and especially for the first stages, the cuts computed can be quite
distant from the corresponding recourse function in the neighborhood of the trial point at which the cut was computed, so this cut is quickly
dominated by other "more relevant" cuts in this neighborhood. Therefore, it makes sense to try and solve more quickly and less accurately (inexactly)
all subproblems of the forward and backward passes corresponding to the first iterations, especially for the first stages, and to increase the precision of the computed solutions as the algorithm progresses.

While the idea behind IDDP and ISDDP is simple and the motivations clear, the description and convergence analysis of IDDP and ISDDP
require solving the following problems of convex analysis, interesting per se, and which, to the best of our knowledge, had not been discussed so far
in the literature:
\begin{itemize}
\item SDDP for nonlinear programs relies on a formula for the subdifferential of the value function $\mathcal{Q}(x)$ of a convex optimization problem
of form: 
\begin{equation}\label{qdefintro}
\mathcal{Q}(x)=\left\{
\begin{array}{l}
\inf_{y \in \mathbb{R}^{n}} \;f(y,x)\\
y\in Y \;:\;Ay+Bx=b,\;g(y,x)\leq 0,
\end{array}
\right.
\end{equation}
where $Y \subseteq \mathbb{R}^n$ is nonempty and convex, $f:\mathbb{R}^n \small{\times} \mathbb{R}^m \rightarrow \mathbb{R} \cup \{+\infty\}$ is
convex, lower semicontinuous, and proper, and the components of 
$g$ are convex lower semicontinuous functions.
Formulas for the subdifferential $\partial \mathcal{Q}(x)$ are given in \cite{guiguessiopt2016}.
These formulas are based on the assumption that primal and dual solutions to \eqref{qdefintro} are available.
When only approximate $\varepsilon$-optimal primal and dual solutions are available for \eqref{qdefintro}
written with $x=\bar x$, we derive formulas for 
affine lower bounding functions $\mathcal{C}$  for $\mathcal{Q}$, that we call inexact cuts,
such that the distance $\mathcal{Q}( \bar x ) - \mathcal{C}( \bar x)$ between the 
values of $\mathcal{Q}$ and of the cut at $\bar x$ is bounded from above by a known function $\varepsilon_0$ of the problem
parameters.
Of course, we would like $\epsilon_0$ to be as small as possible
and $\varepsilon_0=0$ when $\varepsilon=0$.
Two cases are considered: 
\begin{itemize}
\item[(i)] the case when the feasible set
of \eqref{qdefintro} is $Y$, i.e., when the argument $x$ of $\mathcal{Q}$ appears only in the objective function of
\eqref{qdefintro}. In this situation, formulas for inexact cuts are given in Proposition \ref{fixedprop1}, with a refined
bound on $\varepsilon_0$ given in Propositions \ref{fixedprop1b} and \ref{fixedprop2} under an additional assumption.
\item[(ii)] the general case of a value function of form \eqref{qdefintro}.
The corresponding inexact cuts are given in Propositions \ref{varprop1} and \ref{varprop2}.
\end{itemize}
\item We provide conditions ensuring that $\varepsilon$-optimal dual solutions to a convex nonlinear optimization problem
are bounded. Proposition \ref{propboundnormepsdualsol} gives an analytic formula for an upper bound on the norm of these $\varepsilon$-optimal dual solutions.
\item We show in Propositions \ref{propvanish1} and \ref{propvanish1dual}
that if we compute inexact cuts for a sequence $({\underline{\mathcal{Q}}}^k)$  of value functions of the form \eqref{qdefintro}
(with objective functions $f^k$ of special structure)
at a sequence of points $(x^k)$ on the basis of $\varepsilon^k$-optimal primal and dual solutions
with $\lim_{k \rightarrow +\infty} \varepsilon^k = 0$, then
the distance between the inexact cuts and the value functions at these points $x^k$  converges to 0 too.
This result is very natural (see Propositions \ref{propvanish1} and \ref{propvanish1dual}) but some constraint qualification conditions
are needed.
\end{itemize}

When optimization problem \eqref{qdefintro} is linear, i.e., when $\mathcal{Q}$ is the value function 
of a linear program, inexact cuts can easily be obtained from approximate dual solutions
since the dual objective is linear in this case.
This observation was used in \cite{philpzakinex} where inexact cuts are combined with
Benders Decomposition \cite{benderscut} to solve two-stage stochastic linear programs.
In this sense, our work can be seen as an extension of \cite{philpzakinex} where 
two-stage stochastic linear problems are considered whereas ISDDP applies to 
multistage stochastic nonlinear problems. 
In integer programming, inexact master solutions are also commonly used in Benders-like methods \cite{danieldevine},
including in SDDiP, a variant of SDDP to solve multistage stochastic linear programs with integer variables
introduced in \cite{shabbirzou}.

The outline of the study is as follows. 
Section \ref{sec:computeinexactcuts} provides analytic formulas for computing inexact cuts for a value function of an optimization problem
of the form \eqref{qdefintro}. In Section \ref{sec:boundingmulti}, we provide an explicit bound for the norm of $\varepsilon$-optimal dual solutions.
Section \ref{iddp} introduces and studies the IDDP method. The class of problems to which this method applies
is described in Subsection \ref{iddp1}. The detailed IDDP algorithm is given in Subsections \ref{iddp2}-\ref{iddp4} while Subsection \ref{iddp5} studies the convergence of IDDP.
For a problem with  $T$ periods, when noises (error terms quantifying the inexactness) are bounded, by, say, ${\bar{\varepsilon}}$,  
we show in Theorem \ref{convaniddp} and Corollary \ref{corconviddp} that
any accumulation point of the sequence
of decisions  is a $\frac{T(T+1)}{2}({\bar \delta} + {\bar {\varepsilon}})$-optimal solution
 to the problem  where 
 ${\bar \delta}$ is an upper bound on the distance between the value of 
 (theoretical) exact cuts and the value of our inexact cuts at the trial points computed by the algorithm.
It is interesting to see the quadratic dependence of the global error with respect to the number of periods and the linear dependence with respect to noises.
When noises are vanishing we prove that IDDP solves the nonlinear dynamic programming equations (see Theorem \ref{convaniddp}).
Section \ref{sec:isddp} introduces and studies ISDDP. The class of problems to which ISDDP applies is given in Subsection \ref{sec:sddp1}.
A detailed description of ISDDP is given in Subsection \ref{sec:isddpalgo} and its convergence is studied
in Subsection \ref{sec:convsddp}. More precisely, Theorem \ref{convanisddp} shows the convergence of the method when the noises vanish.

We use the following notation and terminology:
\par - The usual scalar product in $\mathbb{R}^n$ is denoted by $\langle x, y\rangle = x^T y$ for $x, y \in \mathbb{R}^n$.
The corresponding norm is $\|x\|=\|x\|_2=\sqrt{\langle x, x \rangle}$.
\par - $\mbox{ri}(A)$ is the relative interior of set $A$.
\par - $\mathbb{B}_n(x_0, r)=\{x \in \mathbb{R}^n : \|x - x_0\| \leq r\}$ for $x_0 \in \mathbb{R}^n, r \geq 0$.
\par - dom($f$) is the domain of function $f$.
\par - $\mbox{Diam}(X)=\max_{x, y \in X} \|x-y\|$ is the diameter of $X$.
\par - $\mathcal{N}_A(x)$ is the normal cone to $A$ at $x$.
\par - $X^\varepsilon := X + \varepsilon  \mathbb{B}_n(0, 1)$ is  the $\varepsilon$-fattening of the set $X \subset \mathbb{R}^n$.
\par - $\mathcal{C}(\mathcal{X})$ is the set of continuous real-valued functions on $\mathcal{X}$, equipped with the
norm $\|f\|_{\mathcal{X}} = \sup_{x \in \mathcal{X}}|f(x)|$. 
\par - $\mathcal{C}^1(\mathcal{X})$ is the set of real-valued continuously differentiable functions on $\mathcal{X}$. 
\par - span($X$) is the linear span of set of vectors $X$ and Aff($X$) is the affine span of $X$.

\section{Computing inexact cuts for the value function of a convex optimization problem}\label{sec:computeinexactcuts}

Let $\mathcal{Q}: X\rightarrow {\overline{\mathbb{R}}}$ be the value function given by
\begin{equation} \label{vfunctionq}
\mathcal{Q}(x)=\left\{
\begin{array}{l}
\inf_{y \in \mathbb{R}^{n}} \;f(y,x)\\
y \in S(x):=\{y\in Y \;:\;Ay+Bx=b,\;g(y,x)\leq 0\}.
\end{array}
\right.
\end{equation}
Here, $X \subseteq \mathbb{R}^m$ and $Y \subseteq \mathbb{R}^n$ are nonempty, compact,  and convex
sets, and $A$ and $B$ are respectively $q \small{\times} n$ and $q \small{\times} m$ real matrices.
We will make the following assumptions which imply, in particular, the convexity of $\mathcal{Q}$ given
by \eqref{vfunctionq}:\\

\par (H1) $f:\mathbb{R}^n \small{\times} \mathbb{R}^m \rightarrow \mathbb{R} \cup \{+\infty\}$ is 
lower semicontinuous, proper, and convex.
\par (H2) For $i=1,\ldots,p$, the $i$-th component of function
$g(y, x)$ is a convex lower semicontinuous function
$g_i:\mathbb{R}^n \small{\times} \mathbb{R}^m \rightarrow \mathbb{R} \cup \{+\infty\}$.\\

In what follows, we say that $\mathcal{C}$ is a cut for $\mathcal{Q}$ if $\mathcal{C}$ is an affine
function of $x$ such that $\mathcal{Q}(x) \geq \mathcal{C}(x)$ for all $x \in X$. We say that the cut
is exact at $\bar x \in X$ if $\mathcal{Q}(\bar x) = \mathcal{C}(\bar x)$.
Otherwise, the cut is said to be inexact.

In this section,  our basic goal is, given ${\bar x} \in X$ and
$\varepsilon$-optimal primal and dual solutions of \eqref{vfunctionq} written for $x=\bar x$, to derive
an inexact cut $\mathcal{C}(x)$ for $\mathcal{Q}$ at $\bar x$, i.e., an affine lower bounding function
for $\mathcal{Q}$ such that the distance
$\mathcal{Q}( \bar x ) - \mathcal{C}( \bar x)$ between the 
values of $\mathcal{Q}$ and of the cut at $\bar x$ is bounded from above by a known function of the problem
parameters. Of course, when $\varepsilon=0$, we will check that 
$\mathcal{Q}( \bar x )= \mathcal{C}( \bar x)$.

We first recall from \cite{guiguessiopt2016} how to compute exact cuts for $\mathcal{Q}$ when optimal primal
and dual solutions of \eqref{vfunctionq} are available.

\subsection{Formula for the subdifferential of the value function of a convex optimization problem} \label{propvaluefunction}

Consider for \eqref{vfunctionq} the dual problem 
\begin{equation}\label{dualpb}
\displaystyle \sup_{(\lambda, \mu) \in \mathbb{R}^q \small{\times} \mathbb{R}_{+}^{p} }\; \theta_{x}(\lambda, \mu)
\end{equation}
for the dual function
\begin{equation}\label{dualfunctionx}
\theta_{x}(\lambda, \mu)=\displaystyle \inf_{y \in Y} \;f(y, x) + \lambda^T (Ay+Bx-b) + \mu^T g(y,x).
\end{equation}
We denote by $\Lambda(x)$ the set of optimal solutions of the  dual problem \eqref{dualpb}
and we use the notation
$$
\mbox{Sol}(x):=\{y \in S(x) : f(y, x)=\mathcal{Q}(x)\}
$$
to indicate the solution set to \eqref{vfunctionq}.

The description of the subdifferential of $\mathcal{Q}$ is given in
the following lemma:
\begin{lemma}\label{dervaluefunction} 
Consider the value function $\mathcal{Q}$ given by \eqref{vfunctionq} and take $x_0 \in X$
such that $S(x_0)\neq \emptyset$.
Let Assumptions (H1) and (H2) hold and
assume the Slater-type constraint qualification condition:
$$
there \;exists\; (\bar x, \bar y) \in X{\small{\times}}\emph{ri}(Y) \mbox{ such that }A \bar y + B \bar x = b \mbox{ and } (\bar y, \bar x) \in \emph{ri}(\{g \leq 0\}).
$$
We also assume that there exists $\varepsilon>0$ such that $Y{\small{\times}}  X^{\varepsilon} \subset \mbox{dom}(f)$.
Then $s \in \partial \mathcal{Q}(x_0)$ if and only if
\begin{equation}\label{caractsubQ}
\begin{array}{l}
(0, s) \in  \partial f(y_0, x_0)+\Big\{[A^T; B^T ] \lambda \;:\;\lambda \in \mathbb{R}^q\Big\}\\
\hspace*{1.2cm}+ \Big\{\displaystyle \sum_{i \in I(y_0, x_0)}\; \mu_i \partial g_i(y_0, x_0)\;:\;\mu_i \geq 0 \Big\}+\mathcal{N}_{Y}(y_0)\small{\times}\{0\},
\end{array}
\end{equation}
where $y_0$ is any element in the solution set \mbox{Sol}($x_0$) and with
$$I(y_0, x_0)=\Big\{i \in \{1,\ldots,p\} \;:\;g_i(y_0, x_0) =0\Big\}.$$ Moreover, the set $\cup_{x \in X} \partial \mathcal{Q}(x)$ is bounded. In particular, if $f$  and $g$ are differentiable, then 
$$
\partial \mathcal{Q}(x_0)=\Big\{  \nabla_x f(y_0, x_0)+ B^T \lambda + \sum_{i \in I(y_0, x_0 )}\; \mu_i \nabla_x g_i(y_0, x_0)\;:\; (\lambda, \mu) \in \Lambda(x_0) \Big\}.
$$
\end{lemma}
\begin{proof}
See the proofs of Lemma 2.1 and Proposition 2.1 in \cite{guiguessiopt2016}. \hfill
\end{proof}

Let us now discuss the computation of inexact cuts for $\mathcal{Q}$ given by \eqref{vfunctionq}.
We start with the case where the argument $x$ of the value function appears only in the objective function of \eqref{vfunctionq}.

\subsection{Fixed feasible set}\label{fixedcsetcut}

As a special case of problem \eqref{vfunctionq},  let $\mathcal{Q}: X\rightarrow {\overline{\mathbb{R}}}$ be the value function given by
\begin{equation} \label{vfunction1}
\mathcal{Q}(x)=\left\{
\begin{array}{l}
\inf_{y \in \mathbb{R}^n} \;f(y, x)\\
y \in Y
\end{array}
\right.
\end{equation}
where $X, Y$ are convex, compact, and nonempty sets. We pick $\bar x \in X$ and
denote by $\bar y \in Y$ an optimal solution of \eqref{vfunction1} written for $x =\bar x$:
\begin{equation}\label{optfixedset}
\mathcal{Q}( \bar x ) = f(\bar y , \bar x).
\end{equation}
Using Lemma \ref{dervaluefunction}, if $f$ is differentiable, we have that $\nabla_x f(\bar y, \bar x) \in \partial \mathcal{Q}( \bar x )$. If 
instead of an optimal solution $\bar y$ of \eqref{vfunction1} we only have 
at hand an approximate $\varepsilon$-optimal solution $\hat y( \varepsilon )$ it is natural to replace $\nabla_x f(\bar y, \bar x)$
by $\nabla_x f(\hat  y( \varepsilon ) , \bar x)$. 
The inexact cut from Proposition \ref{fixedprop1} below will be expressed in terms of the function 
$\ell_1 :Y \small{\times} X \rightarrow \mathbb{R}_{+}$ given  by
\begin{equation}\label{defrxy}
\ell_1 (\hat y , \bar x  ) = - \min_{y \in Y} \langle \nabla_y f( \hat y ,  \bar x  ) , y - \hat y  \rangle = \max_{y \in Y} \langle \nabla_y f(    \hat y , \bar x) , \hat y - y  \rangle.
\end{equation}
\begin{prop} \label{fixedprop1} Let $\bar x \in X$ and
let $\hat y(\varepsilon) \in Y$ be an $\epsilon$-optimal solution for problem \eqref{vfunction1}
written for $x= \bar x$ with optimal value $\mathcal{Q}( \bar x )$, i.e., $\mathcal{Q}( \bar x ) \geq f( \hat y(\varepsilon) , \bar x ) - \varepsilon$.
Assume that $f$ is differentiable and convex on $Y \small{\times} X$.
Then setting  $\eta( \varepsilon )=\ell_1( \hat y(\varepsilon) , \bar x )$, the affine function
\begin{equation}\label{cutfixed1}
\mathcal{C}(x):= f ( \hat y(\varepsilon) , \bar x ) - \eta( \varepsilon )  + \langle \nabla_x f(  \hat y(\varepsilon) , \bar x ) , x - \bar x  \rangle
\end{equation}
is a cut for $\mathcal{Q}$ at $\bar x$, i.e., for every $x \in X$ we have
$\mathcal{Q}(x) \geq \mathcal{C}(x)$ and the quantity $\eta( \varepsilon )$
is an upper bound for
the 
distance 
$\mathcal{Q}( \bar x ) - \mathcal{C}( \bar x)$ between the 
values of $\mathcal{Q}$ and of the cut at $\bar x$. 
\end{prop}
\begin{proof}
For every $(x,y) \in X \small{\times} Y$ using the convexity of $f$ we have
$$
\begin{array}{lll}
f(y, x )  & \geq & f ( \hat y(\varepsilon) , \bar x ) +   \langle \nabla_x f(  \hat y(\varepsilon) , \bar x ) , x - \bar x  \rangle + \langle \nabla_y f(  \hat y(\varepsilon) , \bar x ) , y - \hat y(\varepsilon)  \rangle.   
\end{array}
$$
Minimizing over $y$ in $Y$ on each side of the above inequality we get for every $x \in X$
\begin{equation}\label{subgradfisrtcase0}
\mathcal{Q}( x )  \geq  \mathcal{C}(x) = f ( \hat y(\varepsilon) , \bar x  ) - \ell_1( \hat y(\varepsilon) , \bar x ) + \langle \nabla_x f(  \hat y(\varepsilon) , \bar x  ) , x - \bar x  \rangle
\end{equation}
which shows that $\mathcal{C}$ is a valid cut for $\mathcal{Q}$. 
Finally, since $\hat y(\varepsilon) \in Y$, we have $f ( \hat y(\varepsilon) , \bar x )  \geq \mathcal{Q}( \bar x )$ and
\begin{equation}\label{inexcut1secstep0}
\mathcal{C}(\bar x) - \mathcal{Q}( \bar x ) =f ( \hat y(\varepsilon) , \bar x ) - \ell_1( \hat y(\varepsilon) , \bar x )  - \mathcal{Q}( \bar x ) \geq -\ell_1( \hat y(\varepsilon) , \bar x ).  
\end{equation}\hfill
\end{proof}
We now refine the bound $\ell_1( \hat y(\varepsilon) , \bar x )$ on $\mathcal{Q}( \bar x ) - \mathcal{C}( \bar x)$ given by Proposition \ref{fixedprop1}
making the following assumption:  
\begin{itemize}
\item[(H3)] $f$ is differentiable on $Y \small{\times} X$ and there exists $M_1>0$ such that for every $x \in X, y_1, y_2 \in Y$, we have
$$
\|\nabla_y f(y_2,x) -  \nabla_y f(y_1, x)  \|  \leq  M_1 \|y_2   - y_1\|. 
$$
\end{itemize}
\begin{prop} \label{fixedprop1b} Let $\bar x \in X$ and
let $\hat y(\varepsilon) \in Y$ be an $\epsilon$-optimal solution for problem \eqref{vfunction1}
written for $x= \bar x$ with optimal value $\mathcal{Q}( \bar x )$, i.e., $\mathcal{Q}( \bar x ) \geq f( \hat y(\varepsilon) , \bar x ) - \varepsilon$.
Then setting  $\eta( \varepsilon )=\ell_1( \hat y(\varepsilon) , \bar x )$, if $f$ is differentiable and convex on $Y \small{\times} X$
the affine function
$\mathcal{C}(x)$ given by \eqref{cutfixed1}
is a cut for $\mathcal{Q}$ at $\bar x$.
Moreover, if Assumption (H3) holds, then
setting
\begin{equation}\label{defepsilon00}
\varepsilon_0 =
\left\{ 
\begin{array}{ll}
\frac{ \ell_1( \hat y(\varepsilon) , \bar x )  }{2 M_1 \emph{Diam}(Y)^2 }(  2 M_1 \emph{Diam}(Y)^2  - \ell_1( \hat y(\varepsilon) , \bar x ) ) & \mbox{if }\ell_1( \hat y(\varepsilon) , \bar x )  \leq M_1 \emph{Diam}(Y)^2,\\
\frac{1}{2} \ell_1( \hat y(\varepsilon) , \bar x ) & \mbox{otherwise,}
\end{array}
\right.
\end{equation}
the distance 
$\mathcal{Q}( \bar x ) - \mathcal{C}( \bar x)$ between the 
values of $\mathcal{Q}$ and of the cut at $\bar x$ is at most
$\varepsilon_0$.
\end{prop}
\begin{proof} We already know from Proposition \ref{fixedprop1} that $\mathcal{C}$ is an inexact cut for $\mathcal{Q}$.
It remains to show that if Assumption (H3) holds then
\begin{equation}\label{inexcut1secstep0}
\mathcal{C}(\bar x) - \mathcal{Q}( \bar x ) =f ( \hat y(\varepsilon) , \bar x ) - \ell_1( \hat y(\varepsilon) , \bar x )  - \mathcal{Q}( \bar x ) \geq - \varepsilon_0.  
\end{equation}
Let $y_* \in Y$ be such that
$$
\ell_1( \hat y(\varepsilon) , \bar x ) = \langle \nabla_y f(  \hat y(\varepsilon) , \bar x ) , \hat y(\varepsilon) - y_*  \rangle.
$$
Using (H3), for every $0 \leq t \leq 1$, we have
$$
\begin{array}{lll}
f( \hat y(\varepsilon)+ t (  y_*  - \hat y(\varepsilon)) , \bar x      )
& \leq & f ( \hat y(\varepsilon) , \bar x ) +  t \langle y_* - \hat y(\varepsilon) , \nabla_y f( \hat y(\varepsilon) , \bar x )\rangle +  \frac{1}{2} M_1 t^2 \| \hat y(\varepsilon) - y_*  \|^2 \\
& \leq & f ( \hat y(\varepsilon) , \bar x ) -  t \ell_1( \hat y(\varepsilon) , \bar x )  +  \frac{1}{2} M_1 t^2 \| \hat y(\varepsilon) - y_*  \|^2.
\end{array}
$$
By convexity of $Y$, since $\hat y(\varepsilon), y_* \in Y$, for every $0 \leq t \leq 1$ we have that $\hat y(\varepsilon) + t (  y_*  - \hat y(\varepsilon)    ) \in Y$ 
and the above relation yields
$$  
\mathcal{Q}( \bar x ) \leq  f ( \hat y(\varepsilon) , \bar x ) - \max_{0 \leq t \leq 1} \Big[ t \ell_1( \hat y(\varepsilon) , \bar x  ) - \frac{1}{2} M_1 \mbox{Diam}(Y)^2 t^2 \Big].
$$
If $\ell_1( \hat y(\varepsilon) , \bar x ) \leq M_1 \mbox{Diam}(Y)^2$ then 
$\max_{0 \leq t \leq 1} \Big[ t \ell_1( \hat y(\varepsilon) , \bar x ) - \frac{1}{2} M_1 \mbox{Diam}(Y)^2 t^2 \Big]=\frac{1}{2} \frac{ \ell_1( \hat y(\varepsilon) , \bar x )^2  }{M_1 \mbox{Diam}(Y)^2}$
and
\begin{equation}\label{inexcut1secstep20}
\mathcal{Q}( \bar x ) \leq f ( \hat y(\varepsilon) , \bar x ) - \frac{1}{2} \frac{ \ell_1( \hat y(\varepsilon) , \bar x )^2  }{M_1 \mbox{Diam}(Y)^2}.
\end{equation}
If $\ell_1( \hat y(\varepsilon) , \bar x ) \geq M_1 \mbox{Diam}(Y)^2$ then 
$\max_{0 \leq t \leq 1} \Big[ t \ell_1( \hat y(\varepsilon) , \bar x ) - \frac{1}{2} M_1 \mbox{Diam}(Y)^2 t^2 \Big]=\ell_1( \hat y(\varepsilon) , \bar x ) - \frac{1}{2} M_1 \mbox{Diam}(Y)^2  $
and
\begin{equation}\label{inexcut1secstep30}
\mathcal{Q}( \bar x ) \leq f ( \hat y(\varepsilon) , \bar x ) - \frac{1}{2} \ell_1( \hat y(\varepsilon) , \bar x ).
\end{equation}
Combining \eqref{inexcut1secstep20} and \eqref{inexcut1secstep30} with \eqref{defepsilon00} gives \eqref{inexcut1secstep0} and completes the proof.\hfill
\end{proof}
\begin{rem} As expected, if $\varepsilon=0$ then $\hat y( \varepsilon )$ is an optimal solution of problem \eqref{vfunction1}
written for $x= \bar x$ and the first order optimality conditions ensure
that $\ell_1(\hat y( \varepsilon )  ,  \bar x  )=0$, meaning that 
the cut given by Proposition \ref{fixedprop1} is exact. Otherwise it is inexact.
Since $\ell_1(\hat y( \varepsilon )  ,  \bar x  ) \geq 0$ we also observe that $ \varepsilon_0$ given 
in Proposition \ref{fixedprop1b} is nonnegative and smaller than $\ell_1(\hat y( \varepsilon )  ,  \bar x  )$, which shows
that Proposition \ref{fixedprop1b} improves the bound from Proposition \ref{fixedprop1} for $\mathcal{Q}( \bar x ) - \mathcal{C}( \bar x)$.
\end{rem} 
In Propositions \ref{fixedprop1} and \ref{fixedprop1b}, if the optimization problem 
$\max_{y \in Y} \langle \nabla_y f(\hat y( \varepsilon ) ,  \bar x  ) , \hat y( \varepsilon ) - y  \rangle$
with optimal value
$\ell_1 ( \hat y( \varepsilon )  ,  \bar x  )$ 
is solved approximately, we obtain the cuts given by Proposition \ref{fixedprop2}.
\begin{prop}\label{fixedprop2} Let $\bar x \in X$ and
let $\hat y(\varepsilon_1) \in Y$ be an $\epsilon_1$-optimal solution for problem \eqref{vfunction1}
written for $x= \bar x$ with optimal value $\mathcal{Q}( \bar x )$, i.e., $\mathcal{Q}( \bar x ) \geq f(\hat y(\varepsilon_1) , \bar x  ) - \varepsilon_1$.
Let also $\tilde y(\hat y(\varepsilon_1) , \bar x  ) \in Y$ be an approximate $\epsilon_2$-optimal solution for the problem 
$\max_{y \in Y} \langle \nabla_y f( \hat y(\varepsilon_1)  , \bar x  ) , \hat y(\varepsilon_1) - y  \rangle$
with optimal value $\ell_1(\hat y(\varepsilon_1)  , \bar x  )$, i.e.,
$\ell_1(\hat y(\varepsilon_1) , \bar x  ) - \varepsilon_2 \leq \langle \nabla_y f( \hat y(\varepsilon_1) , \bar x  ) , \hat y(\varepsilon_1) - \tilde y(\hat y(\varepsilon_1) , \bar x  )  \rangle.$
Assume that $f$ is convex and differentiable on $Y \small{\times} X$.
Then setting $\eta( \varepsilon_1, \varepsilon_2 ) = \varepsilon_2 - \langle \tilde y( \hat y(\varepsilon_1) ,  \bar x  ) - \hat y(\varepsilon_1) , \nabla_y f( \hat y(\varepsilon_1) , \bar x  )\rangle$
and $\hat{\ell_1}( \hat y(\varepsilon_1) , \bar x  )=\langle  \hat y(\varepsilon_1)  - \tilde y( \hat y(\varepsilon_1) , \bar x  )  , \nabla_y f( \hat y(\varepsilon_1) , \bar x  )\rangle$,
the affine function
$$
\mathcal{C}(x):= f (\hat y(\varepsilon_1) ,  \bar x   ) - \eta(\epsilon_1 , \epsilon_2) + \langle \nabla_x f(\hat y(\varepsilon_1)  ,   \bar x  ) , x - \bar x  \rangle
$$
is a cut for $\mathcal{Q}$ at $\bar x$, i.e., for every $x \in X$ we have
$\mathcal{Q}(x) \geq \mathcal{C}(x)$ and the distance 
$\mathcal{Q}( \bar x ) - \mathcal{C}( \bar x)$ between the 
values of $\mathcal{Q}$ and of the cut at $\bar x$ is at most $\varepsilon_2 + \hat{\ell}_1 (\hat y(\varepsilon_1) , \bar x)$.
Moreover, if Assumption (H3) holds, setting
\begin{equation}\label{defepsilon0}
\varepsilon_0 =
\left\{
\begin{array}{ll}
\varepsilon_2 + {\hat{\ell_1}}(\hat y(\varepsilon_1) , \bar x) & \mbox{if }\hat{\ell_1}(\hat y(\varepsilon_1) , \bar x) \leq 0,\\
\varepsilon_2 + \frac{  {\hat{\ell_1}}(\hat y(\varepsilon_1) , \bar x) }{ 2 M_1 \emph{Diam}(Y)^2 }(  2 M_1 \emph{Diam}(Y)^2  -  {\hat{\ell_1}}(\hat y(\varepsilon_1) , \bar x) ) & \mbox{if }0<\hat{\ell_1}(\hat y(\varepsilon_1) , \bar x) \leq M_1 \emph{Diam}(Y)^2,\\
\varepsilon_2+\frac{1}{2} {\hat{\ell_1}}(\hat y(\varepsilon_1) , \bar x) & \mbox{otherwise,}
\end{array}
\right.
\end{equation}
the distance 
$\mathcal{Q}( \bar x ) - \mathcal{C}( \bar x)$ between the 
values of $\mathcal{Q}$ and of the cut at $\bar x$ is at most $\varepsilon_0$.
\end{prop}
\begin{proof} We will use the short notation $\hat y$ for  $\hat y(\varepsilon_1)$, $\tilde y$ for $\tilde y(\hat y(\varepsilon_1) , \bar x  )$, and
${\hat{\ell_1}}$ for ${\hat{\ell_1}}(\hat y(\varepsilon_1) , \bar x)$.
Proceeding as in the proof of Proposition \ref{fixedprop1}, we get for every $x \in X$
\begin{equation}\label{subgradfisrtcase}
\begin{array}{lll}
\mathcal{Q}( x ) & \geq & f ( \hat y , \bar x  ) - \ell_1( \hat y , \bar x ) + \langle \nabla_x f(  \hat y , \bar x  ) , x - \bar x  \rangle \\
&   \geq &   \mathcal{C}(x) = f ( \hat y , \bar x  ) + \langle \tilde y - \hat y , \nabla_y f( \hat y , \bar x )\rangle - \varepsilon_2 + \langle \nabla_x f(  \hat y , \bar x  ) , x - \bar x  \rangle
\end{array}
\end{equation}
which shows that $\mathcal{C}$ is a valid cut for $\mathcal{Q}$.
Now observe that 
$$
\mathcal{C}( \bar x )- \mathcal{Q}( \bar x ) = f ( \hat y , \bar x  ) + \langle \tilde y - \hat y , \nabla_y f( \hat y , \bar x )\rangle - \varepsilon_2 - \mathcal{Q}( \bar x ) \geq -\varepsilon_2 - \hat{\ell_1}.  
$$
It remains to show that if Assumption (H3) holds then
\begin{equation}\label{inexcut1secstep}
f ( \hat y , \bar x  ) + \langle \tilde y - \hat y , \nabla_y f( \hat y , \bar x )\rangle - \varepsilon_2 - \mathcal{Q}( \bar x ) \geq - \varepsilon_0.  
\end{equation}
Using assumption (H3) we have for every $0 \leq t \leq 1$,
$$
\begin{array}{lll}
f( \hat y + t (  \tilde y  - \hat y ) , \bar x    )
& \leq & f (\hat y ,  \bar x  ) +  t \langle \tilde y - \hat y , \nabla_y f( \hat y , \bar x )\rangle +  \frac{1}{2} M_1 t^2 \| \tilde y - \hat y  \|^2.
\end{array}
$$
This yields
$$
\mathcal{Q}( \bar x ) \leq  f ( \hat y , \bar x ) + \min_{0 \leq t \leq 1} \Big[ -t \hat{\ell_1} + \frac{1}{2} M_1 \mbox{Diam}(Y)^2 t^2 \Big].
$$
Three cases are possible: $\hat{\ell_1} \leq 0$ (Case A), $0 <\hat{\ell_1} \leq M_1 \mbox{Diam}(Y)^2$  (Case B),
$\hat{\ell_1} > M_1 \mbox{Diam}(Y)^2$  (Case C).
\par Case A. We have   
$$
f ( \hat y , \bar x  ) + \langle \tilde y - \hat y , \nabla_y f( \hat y , \bar x )\rangle - \varepsilon_2 - \mathcal{Q}( \bar x )  \geq -\hat{\ell_1} - \varepsilon_2  = - \varepsilon_0
$$
and \eqref{inexcut1secstep} holds.
\par Case B. We have 
$\min_{0 \leq t \leq 1} \Big[ -t \hat{\ell_1} + \frac{1}{2} M_1 \mbox{Diam}(Y)^2 t^2 \Big]=-\frac{1}{2} \frac{\hat{\ell_1}^2}{M_1 \mbox{Diam}(Y)^2}$
and
\begin{equation}\label{inexcut1secstep2}
\mathcal{Q}( \bar x ) \leq f ( \hat y , \bar x ) - \frac{1}{2} \frac{ \hat{\ell_1}^2  }{M_1 \mbox{Diam}(Y)^2}.
\end{equation}
\par Case C. We have 
$\min_{0 \leq t \leq 1} \Big[ -t \hat{\ell_1} + \frac{1}{2} M_1 \mbox{Diam}(Y)^2 t^2 \Big]=
-{\hat{\ell_1}} + \frac{1}{2} M_1 \mbox{Diam}(Y)^2 \leq -\frac{1}{2}{\hat{\ell_1}}$
which gives
\begin{equation}\label{inexcut1secstep3}
\mathcal{Q}( \bar x ) \leq f ( \hat y , \bar x ) - \frac{1}{2} {\hat \ell}_1.
\end{equation}
Combining \eqref{inexcut1secstep2} and \eqref{inexcut1secstep3} with \eqref{defepsilon0} gives \eqref{inexcut1secstep} for Cases B-C and completes the proof.\hfill
\end{proof}
\begin{rem} If $\varepsilon_1 = \varepsilon_2=0$
then $\hat y$ is an optimal solution of problem \eqref{vfunction1}
written for $x= \bar x$ and 
$\varepsilon_0=\varepsilon_1=\varepsilon_2= \ell_1( \hat y , \bar x )= \hat{\ell_1}(\hat y(\varepsilon_1) , \bar x )  = 0$, meaning that 
the cut given by Proposition \ref{fixedprop2} is exact. 
Also if $\varepsilon_2=0$ then $\hat \ell_1(\hat y(\varepsilon_1) , \bar x ) = \ell_1(\hat y(\varepsilon_1) , \bar x  ) \geq 0$.
Therefore when $\varepsilon_2=0$ and
$0<\hat{\ell_1}(\hat y(\varepsilon_1) , \bar x ) \leq M_1 \emph{Diam}(Y)^2$ or $\hat{\ell_1}(\hat y(\varepsilon_1) , \bar x ) > M_1 \emph{Diam}(Y)^2$
the inexact cuts from Proposition \ref{fixedprop2} correspond to the inexact cuts given in Proposition \ref{fixedprop1b}.
For the case $\hat \ell_1(\hat y(\varepsilon_1) , \bar x )  \leq 0$ in Proposition \ref{fixedprop2},
if $\varepsilon_2=0$ we get $\hat \ell_1 (\hat y(\varepsilon_1) , \bar x ) =0$ which implies $\eta(\varepsilon_1, \varepsilon_2)=0$ and the cut is exact, which is
in accordance with $\varepsilon_0=\varepsilon_2=0$. 
\end{rem}

\subsection{Variable feasible set} \label{compinexcutvarset}

Let us now discuss the computation of inexact cuts for $\mathcal{Q}$ given by \eqref{vfunctionq}.
For $x \in X$, let us introduce for problem \eqref{vfunctionq} the Lagrangian
function 
$$
L_{x}(y, \lambda, \mu)=f(y,x) + \lambda^T (Bx+Ay-b) + \mu^T g(y,x)
$$ and the function
$\ell_2 :Y \small{\times} X  \small{\times}  \mathbb{R}^q      \small{\times} \mathbb{R}_{+}^p \rightarrow \mathbb{R}_{+}$ given by
\begin{equation}\label{defrxy}
\ell_2 (\hat y , \bar x , \hat \lambda , \hat \mu ) = - \min_{y \in Y} \langle \nabla_y L_{\bar x} ( \hat y , \hat \lambda , \hat \mu ) , y - \hat y  \rangle = \max_{y \in Y} 
\langle  \nabla_y L_{\bar x} ( \hat y , \hat \lambda , \hat \mu ) , \hat y - y  \rangle.
\end{equation}
With this notation the dual function \eqref{dualfunctionx} for problem \eqref{vfunctionq} can be written
$$
\theta_{x}(\lambda, \mu)=\displaystyle \inf_{y \in Y} \;L_{x}(y, \lambda, \mu).
$$
We make the following assumption which ensures no duality gap for \eqref{vfunctionq} for any $x \in X$:
\begin{itemize}
\item[(H4)] for every $x \in X$ there exists $y_x \in \mbox{ri}(Y)$ such that
$Bx+Ay_x=b$ and $g(y_x , x)< 0$.
\end{itemize}
The following proposition provides an inexact cut for $\mathcal{Q}$ given by \eqref{vfunctionq}:
\begin{prop} \label{varprop1} Let $\bar x \in X$,
let $\hat y(\epsilon)$ be an $\epsilon$-optimal feasible primal solution for problem \eqref{vfunctionq}
written for $x= \bar x$ 
and let $(\hat \lambda(\epsilon), \hat \mu(\epsilon))$ be an $\epsilon$-optimal feasible solution of the
corresponding dual problem, i.e., of problem
\eqref{dualpb} written for $x=\bar x$.
Let Assumptions (H1), (H2), and (H4) hold. If additionally $f$ and $g$ are differentiable on $Y \small{\times} X$ then
setting $\eta(\varepsilon)=\ell_2 (\hat y(\epsilon) , \bar x  , \hat \lambda(\epsilon) , \hat \mu(\epsilon) )$, the affine function
\begin{equation}\label{cutvarprop1}
\mathcal{C}(x):= L_{\bar x} ( \hat y(\epsilon), {\hat \lambda}(\epsilon), \hat \mu(\epsilon) )- \eta(\varepsilon)  + 
\langle \nabla_x L_{\bar x} ( \hat y(\epsilon), {\hat \lambda}(\epsilon), \hat \mu(\epsilon) ) , x - \bar x \rangle
\end{equation}
is a cut for $\mathcal{Q}$ at $\bar x$ and the distance 
$\mathcal{Q}( \bar x ) - \mathcal{C}( \bar x)$ between the 
values of $\mathcal{Q}$ and of the cut at $\bar x$ is at most
$\varepsilon + \ell_2 (\hat y(\epsilon),  \bar x  , \hat \lambda(\epsilon) , \hat \mu(\epsilon) )$.
\end{prop}
\begin{proof}
To simplify notation, we use 
$\hat y, \hat \lambda, \hat \mu$, for respectively $\hat y(\epsilon), \hat \lambda(\epsilon), \hat \mu(\epsilon)$.
Consider primal problem \eqref{vfunctionq} written for $x=\bar x$.
Due to Assumption (H4) the optimal value $\mathcal{Q}(\bar x)$ of this problem
is the optimal value of the corresponding dual problem, i.e., of problem
\eqref{dualpb} written for $x=\bar x$. 
Using the fact that $\hat y$ and $(\hat \lambda, \hat \mu)$ are respectively $\varepsilon$-optimal primal
and dual solutions it follows that
\begin{equation}\label{optprimaldual}
f( \hat y , \bar x) \leq \mathcal{Q}(  \bar x  ) + \varepsilon \mbox{ and }\theta_{\bar x}( \hat \lambda , \hat \mu  ) \geq \mathcal{Q}(  \bar x  ) - \varepsilon. 
\end{equation}
Moreover, since the approximate primal and dual solutions are feasible, we have that
\begin{equation}\label{feasibility}
 \hat y \in Y,\, B {\bar x} + A {\hat y} = b,\,g(\hat y, \bar x) \leq 0, \,\hat \mu \geq 0.
\end{equation}
Using Relation \eqref{optprimaldual}, the definition of dual function $\theta_{{\bar x}}$, and the fact that $\hat y \in Y$, we get
\begin{equation}\label{optdual2}
L_{\bar x} (\hat y, {\hat \lambda}, \hat \mu ) \geq  \theta_{\bar x}( \hat \lambda , \hat \mu  ) \geq \mathcal{Q}( \bar x ) - \varepsilon.  
\end{equation}
Due to Assumptions (H1) and (H2), for any $\lambda$ and $\mu \geq 0$ the function $L_{\cdot}(\cdot,\lambda,\mu)$ which associates 
the value $L_{x}(y, \lambda, \mu)$ to $(x,y)$ is convex. 
It follows that for every $x \in X, y \in Y$, we have that 
$$
L_x (y, {\hat \lambda}, \hat \mu )   \geq  
L_{\bar x} ( \hat y, {\hat \lambda}, \hat \mu ) +\langle \nabla_x L_{\bar x} ( \hat y, {\hat \lambda}, \hat \mu ) , x - \bar x \rangle + \langle \nabla_y L_{\bar x} ( \hat y, {\hat \lambda}, \hat \mu ) , y - \hat y \rangle.
$$
Since $(\hat \lambda, \hat \mu)$ is dual feasible for dual problem \eqref{dualpb}, the Weak Duality Theorem gives
$\mathcal{Q}( x  ) \geq  \theta_{x}( \hat \lambda , \hat \mu  )  = \inf_{y \in Y} L_{x}(y, \hat \lambda , \hat \mu )$ 
for every $x \in X$
and minimizing over $y \in Y$ on each side of the above inequality  we obtain
$$
\mathcal{Q}( x  )   \geq  
\mathcal{C}(x)=
L_{\bar x} ( \hat y, {\hat \lambda}, \hat \mu )-  \ell_2 (\hat y,  \bar x  , \hat \lambda , \hat \mu ) + 
\langle \nabla_x L_{\bar x} ( \hat y, {\hat \lambda}, \hat \mu ) , x - \bar x \rangle.
$$
Finally, using relation \eqref{optdual2}, we get
$$
\mathcal{Q}( \bar x ) - \mathcal{C}( \bar x ) =
\mathcal{Q}( \bar x ) - L_{\bar x} ( \hat y, {\hat \lambda}, \hat \mu )+  \ell_2 (\hat y,  \bar x  , \hat \lambda , \hat \mu ) \leq \varepsilon + \ell_2 (\hat y , \bar x  , \hat \lambda , \hat \mu ).
$$\hfill
\end{proof}
We now refine the bound $\varepsilon + \ell_2 (\hat y(\epsilon),  \bar x  , \hat \lambda(\epsilon) , \hat \mu(\epsilon) )$ on $\mathcal{Q}( \bar x ) - \mathcal{C}( \bar x)$ given by Proposition \ref{varprop1}
making the following assumption:  
\begin{itemize}
\item[(H5)] $g$ is differentiable on $Y \small{\times} X$ and
there exists $M_2>0$ such that for every $i=1,\ldots,p, x \in X, y_1, y_2 \in Y$, we have
$$
\|\nabla_y g_i(y_2,x) -  \nabla_y g_i(y_1, x)  \|  \leq  M_2 \|y_2   - y_1\|. 
$$
\end{itemize}
\begin{prop} \label{varprop2}
Let $\bar x \in X$,
let $\hat y( \epsilon)$ be an $\epsilon$-optimal feasible primal solution for problem \eqref{vfunctionq}
written for $x= \bar x$ 
and let $(\hat \lambda( \epsilon), \hat \mu( \epsilon))$ be an $\epsilon$-optimal feasible solution of the
corresponding dual problem, i.e., of problem
\eqref{dualpb} written for $x=\bar x$.
Let also $\mathcal{L}_{\bar x}$ be any lower bound on $\mathcal{Q}( \bar x )$.
Let Assumptions (H1), (H2), (H3), (H4), and (H5) hold. 
Then $\mathcal{C}(x)$ given by \eqref{cutvarprop1} is a cut for $\mathcal{Q}$
at $\bar x$  and setting
$
M_3=M_1 + \mathcal{U}_{\bar x} M_2
$
with
$$
\mathcal{U}_{\bar x}= 
\frac{f(y_{\bar x} , \bar x  )  - \mathcal{L}_{\bar x} + \varepsilon }{\min (-g_{i}(y_{\bar x} ,  \bar x), i=1,\ldots,p )},
$$
the distance 
$\mathcal{Q}( \bar x ) - \mathcal{C}( \bar x)$ between the 
values of $\mathcal{Q}$ and of the cut at $\bar x$ is at most
$$
\varepsilon_0 = 
\left\{
\begin{array}{ll}
\varepsilon + \ell_2 (\hat y( \epsilon),  \bar x  , \hat \lambda( \epsilon) , \hat \mu( \epsilon) ) - \frac{\ell_2 (\hat y( \epsilon),  \bar x  , \hat \lambda( \epsilon) , \hat \mu( \epsilon) )^2}{2 M_3 \emph{Diam}(Y)^2} & \mbox{if }\ell_2 (\hat y( \epsilon),  \bar x  , \hat \lambda( \epsilon) , \hat \mu( \epsilon) ) \leq M_3 \emph{Diam}(Y)^2,\\
\varepsilon + \frac{1}{2} \ell_2 (\hat y( \epsilon),  \bar x  , \hat \lambda( \epsilon) , \hat \mu( \epsilon) ) & \mbox{otherwise.}
\end{array}
\right.
$$
\end{prop}
\begin{proof}
As before we use the short notation
$\hat y, \hat \lambda, \hat \mu$, for respectively $\hat y(\epsilon), \hat \lambda(\epsilon), \hat \mu(\epsilon)$.
We already know from Proposition \ref{varprop1} that $\mathcal{C}$ is a cut for $\mathcal{Q}$. Let us show that 
$\varepsilon_0$ is an upper bound for $\mathcal{Q}( \bar x ) - \mathcal{C}( \bar x)$.
We compute 
$$
\nabla_y L_{\bar x}(y, \lambda, \mu) =\nabla_y f(y, \bar x) + A^T \lambda + \sum_{i=1}^p \mu_i \nabla_y g_i(y, \bar x).
$$
Therefore for every $y_1, y_2 \in Y$, using Assumptions (H3) and (H5), we have
\begin{equation} \label{boundgradLag}
\|\nabla_y L_{\bar x}(y_2, {\hat \lambda}, {\hat \mu})  -   \nabla_y L_{\bar x}(y_1, {\hat \lambda}, {\hat \mu}) \| \leq (M_1 +  \|{\hat \mu}\|_1 M_2) \|y_2 - y_1\|.  
\end{equation}
Next observe that 
$$
\begin{array}{lll}
\mathcal{L}_{\bar x} - \varepsilon \leq 
\mathcal{Q}( \bar x )-\varepsilon  \leq \theta_{\bar x}(\hat \lambda, \hat \mu)  & \leq & f(y_{\bar x}, {\bar x}) + {\hat \lambda}^T ( A y_{\bar x} + B {\bar x} - b ) +  {\hat \mu}^T g( y_{\bar x}, \bar x ) \\
& \leq & f(y_{\bar x}, {\bar x}) + \|{\hat \mu}\|_1  \max_{i=1,\ldots,p} g_{i}(y_{\bar x} ,  \bar x).
\end{array}
$$
From the above relation, we get $\| {\hat \mu}\|_1 \leq \mathcal{U}_{\bar x}$, which, plugged into \eqref{boundgradLag} gives
\begin{equation}\label{boundlag}
\|\nabla_y L_{\bar x}(y_2, {\hat \lambda}, {\hat \mu})  -   \nabla_y L_{\bar x}(y_1, {\hat \lambda}, {\hat \mu}) \| \leq M_3 \|y_2 - y_1\|.
\end{equation}
The computations are now similar to the proof of Proposition \ref{fixedprop1b}.  
More precisely, let $y_* \in Y$ such that 
$$
\ell_2 (\hat y , \bar x , \hat \lambda , \hat \mu ) = 
\langle  \nabla_y L_{\bar x} ( \hat y , \hat \lambda , \hat \mu ) , \hat y - y_*  \rangle.
$$
Using relation \eqref{boundlag}, for every $0 \leq t \leq 1$, we get
$$
L_{\bar x}({\hat y} + t (y_* - {\hat y}), {\hat \lambda}, {\hat \mu}) \leq L_{\bar x}({\hat y}, {\hat \lambda}, {\hat \mu})
+ t \langle \nabla_y L_{\bar x}({\hat y}, {\hat \lambda}, {\hat \mu}), y_* - {\hat y} \rangle + \frac{1}{2} M_3 t^2 \|y_* - {\hat y}\|^2.
$$
Since ${\hat y} + t (y_* - {\hat y}) \in Y$, using the above relation and the definition of $\theta_{\bar x}$, we obtain
$$
\mathcal{Q}( \bar x ) - \varepsilon \leq \theta_{\bar x}(\hat \lambda , \hat \mu ) \leq L_{\bar x}({\hat y}, {\hat \lambda}, {\hat \mu}) -t \ell_2 (\hat y,  \bar x  , \hat \lambda , \hat \mu ) + \frac{1}{2} M_3 t^2 \|y_* - {\hat y}\|^2.
$$
Therefore
$$
\mathcal{Q}( \bar x ) - \mathcal{C}( \bar x ) = \mathcal{Q}( \bar x ) - L_{\bar x}({\hat y}, {\hat \lambda}, {\hat \mu}) + \ell_2 (\hat y,  \bar x  , \hat \lambda , \hat \mu )
\leq \varepsilon +  \ell_2 (\hat y,  \bar x  , \hat \lambda , \hat \mu ) + \min_{0 \leq t \leq 1} \Big(-t \ell_2 (\hat y,  \bar x  , \hat \lambda , \hat \mu ) + \frac{1}{2} M_3 t^2 \mbox{Diam}(Y)^2 \Big)
$$
and we easily conclude computing $\min_{0 \leq t \leq 1} \Big( -t \ell_2 (\hat y,  \bar x  , \hat \lambda , \hat \mu ) + \frac{1}{2} M_3 t^2 \mbox{Diam}(Y)^2 \Big)$.\hfill
\end{proof}
\begin{rem}\label{remextincutvar} As was done for the extension of Proposition \ref{fixedprop1} corresponding to Proposition \ref{fixedprop2},
we can extend Proposition \ref{varprop2} to the case when the optimization problem 
$\max_{y \in Y}  \langle  \nabla_y L_{\bar x} ( \hat y , \hat \lambda , \hat \mu ) , \hat y - y  \rangle$
with optimal value
$\ell_2 (\hat y , \bar x , \hat \lambda , \hat \mu )$
is  solved approximately.
\end{rem}

\section{Bounding the norm of $\varepsilon$-optimal solutions to the dual of a convex optimization problem} \label{sec:boundingmulti} 

Consider the following convex optimization problem:
\begin{equation} \label{defpbbounddual}
f_* =
\left\{
\begin{array}{l}
\min f(y)\\
Ay= b, \,g(y) \leq 0,\;y \in Y 
\end{array}
\right.
\end{equation}
where 
\begin{itemize}
\item[(i)] $Y \subset \mathbb{R}^n$ is a closed convex set and $A$ is a $q \small{\times}  n$ matrix;
\item[(ii)] $f: Y \rightarrow \mathbb{R}$ is convex Lipschitz continuous with Lipschitz constant $L( f )$;
\item[(iii)] $g: Y \rightarrow \mathbb{R}^p$ where all components of $g$ are convex Lipschitz continuous functions
with Lipschitz constant $L( g )$;
\item[(iv)] $f$ is bounded from below on the feasible set.
\end{itemize}
We also assume the following Slater type constraint qualification condition:
\begin{equation}\label{slaterboundmulti}
\mbox{SL: There exist }\kappa>0 \mbox{ and }y_0 \in \mbox{ri}( Y ) \mbox{ such that }g( y_0 ) \leq -\kappa {\textbf{e}} \mbox{ and }A y_0 = b   
\end{equation}
where {\textbf{e}} is a vector of ones in $\mathbb{R}^p$. 

Since SL holds, the optimal value $f_*$ of \eqref{defpbbounddual} can be written as the optimal value of the dual problem:
\begin{equation}\label{dualfirst}
f_* = \displaystyle \max_{\lambda , \mu \geq 0} \left\{ \theta(\lambda, \mu ) := \displaystyle \min_{y \in Y} \{ f(y) + \langle \lambda , Ay -b \rangle + \langle \mu , g(y) \rangle  \} \right\}.
\end{equation}

Consider the vector space $F=A\mbox{Aff}(Y) - b$ (recall that $0 \in F$).
Clearly for any $y\in Y$ and every $\lambda \in F^{\perp}$ we have $\lambda^T ( A y - b) = 0$ and therefore
for every $\lambda \in \mathbb{R}^q$, $\theta(\lambda , \mu )=\theta(\Pi_{F}( \lambda ) , \mu )$
where $\Pi_{F}( \lambda )$ is the orthogonal projection of $\lambda$ onto $F$.

It follows that if $F^{\perp} \neq \{ 0 \}$, the set of $\epsilon$-optimal dual solutions of
dual problem \eqref{dualfirst} is not bounded because from any $\epsilon$-optimal dual solution
$(\lambda(\varepsilon), \mu(\varepsilon))$ we can build an $\epsilon$-optimal dual solution
$(\lambda(\varepsilon)+\lambda, \mu(\varepsilon))$ with the same value of the dual function of norm
arbitrarily large taking $\lambda$ in $F^{\perp}$ with norm sufficiently large. 

However, the optimal value of the dual (and primal) problem can be written equivalently as
\begin{equation}\label{dualreform}
f_* = \displaystyle \max_{\lambda, \mu} \left\{  \theta(\lambda , \mu )  : \mu \geq 0, \lambda = A y-b, y \in \mbox{Aff}(Y) \right\}. 
\end{equation}

In this section, our goal is to derive bounds on the norm of $\epsilon$-optimal solutions to the dual of \eqref{defpbbounddual}
written in the form \eqref{dualreform}.

From Assumption SL, we deduce that there exists $r>0$ such that $\mathbb{B}_n( y_0 , r)  \cap   \mbox{Aff}(Y) \subseteq Y$
and that there is some ball $\mathbb{B}_q( 0, \rho_*)$ of positive radius $\rho_*$ such that
the intersection of this ball and of the set $A \mbox{Aff}(Y) - b$ is contained in the set 
$A\Big( \mathbb{B}_n( y_0 , r)  \cap   \mbox{Aff}(Y)  \Big)-b$. To define such $\rho_*$, let
$\rho: A \mbox{Aff}(Y)-b \rightarrow \mathbb{R}_{+}$ given by
$$
\rho(z)=\max \left\{t \|z\| \;:\;t\geq 0, tz \in A( \mathbb{B}_n( y_0 , r)  \cap   \mbox{Aff}(Y) )-b\right\}.
$$
Since $y_0 \in Y$, we can write $\mbox{Aff}(Y)=y_0 + V_Y$ where $V_Y$ is the vector space
$V_Y = \{x-y,\;x,y \in \mbox{Aff}(Y)\}$. Therefore
$$
A( \mathbb{B}_n( y_0 , r)  \cap   \mbox{Aff}(Y) )-b= A ( \mathbb{B}_n (0, r ) \cap V_Y )
$$
and $\rho$ can be reformulated as
\begin{equation}\label{definitionrho}
\rho(z)=\max \left\{t \|z\| \;:\;t\geq 0, tz \in A ( \mathbb{B}_n (0, r ) \cap V_Y )\right\}.
\end{equation}
Note that $\rho$ is well defined and finite valued (we have $ 0\leq \rho(z) \leq \|A\|r$).
Also, clearly $\rho( 0 ) = 0$ and
$\rho(z)=\rho(\lambda z)$ for every $\lambda > 0$ and $z \neq 0$. 
Therefore if $A=0$ then  $\rho_*$ can be any positive real, for instance $\rho_*=1$,
and if $A \neq 0$ 
we define 
\begin{equation}\label{defrhostar}
\begin{array}{lll}
\rho_* = \min \{\rho(z)\;:\;z\neq 0, z \in A \mbox{Aff}(Y)-b\}&=&\min \{\rho(z)\;:\;\|z\|=1, z \in A \mbox{Aff}(Y)-b\},\\
& = & \min \{\rho(z)\;:\;\|z\|=1, z \in A V_Y \},
\end{array}
\end{equation}
which is well defined and positive since $\rho(z)>0$ for every $z$ such that  $\|z\|=1, z \in A \mbox{Aff}(Y)-b$
(indeed if $z \in A \mbox{Aff}(Y)-b$ with $\|z\|=1$ then $z = A y-b$ for some $y \in\mbox{Aff}(Y), y \neq y_0$,
and since 
$$
\frac{r}{\|y-y_0\|} z = A\Big( y_0  + r \frac{y- y_0}{\|y-y_0\|} \Big)- b \in A\Big( \mathbb{B}_n( y_0 , r)  \cap   \mbox{Aff}(Y)  \Big)-b,
$$
we have $\rho(z) \geq \frac{r}{\|y-y_0\|}\|z\|=\frac{r}{\|y-y_0\|}>0$). We now claim that parameter $\rho_*$ we have just defined satisfies
our requirement namely
\begin{equation}\label{deductionfromSL}
\mathbb{B}_q( 0, \rho_*) \cap \Big(  A  \mbox{Aff}(Y) - b \Big)  \subseteq A\Big( \mathbb{B}_n( y_0 , r)  \cap   \mbox{Aff}(Y)  \Big)-b.
\end{equation}
This can be rewritten as 
\begin{equation}\label{deductionfromSLrewr}
\mathbb{B}_q( 0, \rho_*) \cap A  V_Y  \subseteq A\Big( \mathbb{B}_n( 0 , r)  \cap  V_Y  \Big).
\end{equation}
Indeed, let $z \in \mathbb{B}_q( 0, \rho_*) \cap \Big(  A \mbox{Aff}(Y) - b \Big)$.
If $A=0$ or $z=0$ then $ z \in A\Big( \mathbb{B}_n( y_0 , r)  \cap   \mbox{Aff}(Y)  \Big)-b$.
Otherwise, by definition of $\rho$, we have $\rho(z ) \geq \rho_* \geq \|z\|$.
Let ${\bar t} \geq 0$ be such that ${\bar t} z \in A( \mathbb{B}_n( y_0 , r)  \cap   \mbox{Aff}(Y) )-b$ and
$\rho(z)={\bar t}\|z\|$. The relations $({\bar t}-1)\|z\| \geq 0$ and $z \neq 0$ imply $\bar t \geq 1$.
By definition of $\bar t$, we can write ${\bar t} z = Ay -b$ where $y \in \mathbb{B}_n( y_0 , r)  \cap   \mbox{Aff}(Y)$.
It follows that $z$ can be written
$$
z=A\Big(y_0 + \frac{y-y_0}{{\bar t}}  \Big) -b = A {\bar y} - b
$$
where $\displaystyle  \bar y = y_0 + \frac{y-y_0}{{\bar t}}  \in \mbox{Aff}(Y)$ and 
$\|\bar y - y_0\|= \displaystyle \frac{\|y-y_0\|}{\bar t} \leq \|y - y_0\| \leq r$ (because $\bar t \geq 1$ and $y \in \mathbb{B}_n( y_0 , r)$).
This  means that $z \in A\Big( \mathbb{B}_n( y_0 , r)  \cap   \mbox{Aff}(Y)  \Big)-b$, which proves inclusion \eqref{deductionfromSL}.

We are now in a position to state the main result of this section:
\begin{prop}\label{propboundnormepsdualsol} Consider the optimization problem \eqref{defpbbounddual} with optimal value $f_*$.
Let Assumptions (i)-(iv) and SL hold and let $(\lambda(\varepsilon), \mu(\varepsilon))$ be an $\varepsilon$-optimal solution
to the dual problem \eqref{dualreform} with optimal value $f_*$.
Let 
\begin{equation}\label{defrkappaLg}
0<r \leq \frac{\kappa}{2 L( g )}, 
\end{equation}
be such that 
the intersection of the ball $\mathbb{B}_n(y_0, r)$ 
and of Aff($Y$) is contained in $Y$ (this $r$ exists because $y_0 \in \mbox{ri}(Y)$).
If $A=0$ let $\rho_*=1$. Otherwise, let $\rho_*$ given by
\eqref{defrhostar} with  $\rho$ as in  \eqref{definitionrho}.
Let $\mathcal{L}$ be any lower bound on the optimal value
$f_*$ of \eqref{defpbbounddual}. Then we have
$$
\|(\lambda(\varepsilon), \mu(\varepsilon)) \| \leq \frac{f( y_0 ) - \mathcal{L} + \varepsilon + L( f ) r }{\min( \rho_* , \kappa/2)}.
$$
\end{prop}
\begin{proof} By definition of $(\lambda(\varepsilon), \mu(\varepsilon))$ and of $\mathcal{L}$, we have
\begin{equation}\label{firstrelationboundepssol}
\mathcal{L} - \varepsilon \leq f_* - \varepsilon \leq \theta( \lambda(\varepsilon), \mu(\varepsilon)) . 
\end{equation}
Now define $z(\varepsilon)=0$ if $\lambda(\varepsilon)=0$ and
$z( \varepsilon )=-\frac{\rho_*}{\| \lambda( \varepsilon ) \| }  \lambda( \varepsilon )$ otherwise.
Observing that $z( \varepsilon ) \in \mathbb{B}_q( 0, \rho_*) \cap \Big(  A  \mbox{Aff}(Y) - b \Big)$
and using relation \eqref{deductionfromSL} we deduce that
$z( \varepsilon ) \in A\Big( \mathbb{B}_n( y_0 , r)  \cap   \mbox{Aff}(Y)  \Big)-b \subseteq AY -b$.
Therefore, we can write $z( \varepsilon ) = A {\bar y}-b$ for some
${\bar y} \in \mathbb{B}_n(y_0 , r) \cap \mbox{Aff}(Y) \subseteq Y$. Next, using the definition of $\theta$, we get 
$$
\begin{array}{lll}
\theta( \lambda(  \varepsilon ) , \mu ( \varepsilon  )  ) &   \leq & 
f( \bar y ) + \lambda( \varepsilon )^T ( A {\bar y} - b ) + \mu( \varepsilon )^T g( \bar y ) \mbox{ since }{\bar y} \in Y,\\
& \leq & f(y_0 ) + L( f ) r + z( \varepsilon )^T \lambda( \varepsilon ) + \mu( \varepsilon )^T g( y_0 ) + L(g) r  \| \mu( \varepsilon ) \|_1  \mbox{ using }(ii), (iii), \bar y \in \mathbb{B}_n(y_0, r),\\
& \leq & f(y_ 0) + L( f ) r - \rho_*  \| \lambda( \varepsilon ) \| - \frac{\kappa}{2} \| \mu( \varepsilon ) \|_1  \mbox{ using SL and }\eqref{defrkappaLg},
\end{array}
$$
which can be rewritten as
\begin{equation}\label{rewritingfinaldualnorm}
\begin{array}{lll}
\|(\lambda(  \varepsilon ) , \mu ( \varepsilon  ))\| &=& \sqrt{ \|\lambda(  \varepsilon )\|^2  + \|\mu(  \varepsilon )\|^2  }  \leq  \|\lambda(  \varepsilon )\| + \|\mu(  \varepsilon )\|
\leq  \|\lambda(  \varepsilon )\| + \|\mu(  \varepsilon )\|_1 \leq \frac{f(y_0 ) + L( f ) r - \theta(\lambda(  \varepsilon ) , \mu ( \varepsilon  ))}{\min( \rho_* , \kappa/2)}. 
\end{array}
\end{equation}
Combining \eqref{firstrelationboundepssol} with \eqref{rewritingfinaldualnorm}, we obtain the desired bound. \hfill
\end{proof}
Recalling that $\mbox{Aff}(Y)={\tilde y} + \mbox{span}(Y- {\tilde y})$ for any $\tilde y \in Y$, the constraints 
$y \in \mbox{Aff}(Y)$ in \eqref{dualreform}
can be written 
$y={\tilde y} +  \sum_{i=1}^k \alpha_i e_i$ in variables $(\alpha_i)_{i=1}^k$
where $(e_1,\ldots,e_k)$ is a basis
of $\mbox{span}(Y - {\tilde y})$ and $\tilde y$ is an arbitrary point chosen in $Y$. 
For instance, if $Y- {\tilde y}$ is a box, i.e., $Y- {\tilde y}=\{y \in \mathbb{R}^n : \ell \leq y \leq u \}$ with $\ell <0< u$ then $\mbox{span}(Y- {\tilde y})=\mathbb{R}^n$
and if $Y- {\tilde y}=\{y \in \mathbb{R}^n : \ell_i \leq y_i \leq u_i, i=1,\ldots,n_0, \;y_i = 0,\, i > n_0 \}$ with $\ell_i < 0 < u_i$
then the first $n_0$ vectors of the canonical basis of $\mathbb{R}^n$ form a basis of span($Y- {\tilde y}$)=$\mathbb{R}^{n_0}  \small{\times} \underbrace{\{0 \} \small{\times} \ldots \small{\times} \{0\}}_{n - n_0 \mbox{ times}} $. 

We also have the following immediate corollary of Proposition \ref{propboundnormepsdualsol}:

\begin{cor}\label{corpropboundnormepsdualsol} Under the assumptions of Proposition \ref{propboundnormepsdualsol}, let $\bar f$
be an upper bound on $f$ on the feasibility set of \eqref{defpbbounddual} and assume that
$\bar f$ is convex and Lipschitz continuous on $\mathbb{R}^n$ with Lipschitz constant $L( \bar f )$. Then we have for 
$\|(\lambda(\varepsilon), \mu(\varepsilon)) \|$ the bound 
$
\|(\lambda(\varepsilon), \mu(\varepsilon)) \| \leq \frac{ {\bar f} ( y_0 ) - \mathcal{L} + \varepsilon + L( \bar f ) r }{\min( \rho_* , \kappa/2)}.
$
\end{cor}

\section{Inexact Dual Dynamic Programming (IDDP)} \label{iddp}

\subsection{Problem formulation and assumptions}\label{iddp1}

Consider the optimization problem
\begin{equation}\label{defpb}
\left\{
\begin{array}{l}
\displaystyle \inf_{x_1,\ldots,x_T}\;\sum_{t=1}^{T} f_t(x_{t}, x_{t-1}) \\
x_t \in X_t(x_{t-1}), t=1,\ldots,T,
\end{array}
\right.
\end{equation}
for $x_0$ given with the corresponding dynamic programming equations
$$
\mathcal{Q}_{t}(x_{t-1})=
\left\{
\begin{array}{l}
\displaystyle \inf_{x_t}\;F_t(x_{t}, x_{t-1}):=f_t(x_{t}, x_{t-1}) + \mathcal{Q}_{t+1}(x_{t})\\
x_t \in X_t(x_{t-1}),
\end{array}
\right.
$$
for $t=1,\ldots,T$, with $\mathcal{Q}_{T+1} \equiv 0$. Observe that $\mathcal{Q}_1 ( x_0 )$ is the optimal value of \eqref{defpb}.

We will consider two structures for sets $X_t(x_{t-1} ),t=1,\ldots,T$:
\begin{itemize}
\item[(S1)] $X_t( x_{t-1} ) = \mathcal{X}_t \subset \mathbb{R}^n$ (in this case, for short, we say that $X_t$ is of type S1); 
\item[(S2)] $X_t( x_{t-1} )= \{x_t \in \mathbb{R}^n : x_t \in \mathcal{X}_t,\;g_t(x_t, x_{t-1}) \leq 0,\;\;\displaystyle A_{t} x_{t} + B_{t} x_{t-1} = b_t \}$
(in this case, for short, we say that $X_t$ is of type S2).
\end{itemize}
Note that a mix of these types of constraints is allowed: for instance we can have $X_1$ of type S1 and $X_2$ of type $S2$.

Setting $\mathcal{X}_0=\{x_0\}$, we make the following assumptions (H1): for $t=1,\ldots,T$,\\

\par (H1)-(a) $\mathcal{X}_t$ is nonempty, convex, and compact.
\par (H1)-(b) The function
$f_t(\cdot, \cdot)$ is convex on  $\mathcal{X}_t \small{\times} \mathcal{X}_{t-1}$
and belongs to  $\mathcal{C}^{1}(\mathcal{X}_t \small{\times} \mathcal{X}_{t-1})$.\\

For $t=1,\ldots,T$, if $X_t$ is of type $S2$ we additionally assume that: there exists $\varepsilon_t>0$ such that
(without loss of generality, we will assume in the sequel that $\varepsilon_t=\varepsilon$)\\

\par (H1)-(c) each component $g_{t i}(\cdot, \cdot), i=1,\ldots,p$,  of the function $g_t(\cdot, \cdot)$ is
convex on $\mathcal{X}_t \small{\times} \mathcal{X}_{t-1}^{\varepsilon_t}$
and belongs to $\mathcal{C}^{1}( \mathcal{X}_t \small{\times} \mathcal{X}_{t-1} )$.
\par (H1)-(d) For every $x_{t-1} \in \mathcal{X}_{t-1}^{\varepsilon_t}$,
the set $X_t(x_{t-1}) \cap \mbox{ri}( \mathcal{X}_t )$ is nonempty.
\par (H1)-(e) If $t \geq 2$, there exists
${\bar x}_{t}=({\bar x}_{t t}, {\bar x}_{t  t-1}  ) \in  \mbox{ri}(\mathcal{X}_t) \small{\times} \mathcal{X}_{t-1}$
such that $A_t {\bar x}_{t t} + B_t {\bar x}_{t  t-1} = b_t$, and $g_t( {\bar x}_{t t} , {\bar x}_{t t-1} ) < 0 $.\\

Assumptions (H1)-(a), (b), (c) ensure that functions $\mathcal{Q}_t$ are convex.
Assumption (H1)-(d) is used to bound the cut coefficients (see Proposition \ref{propboundeddual})
and show that functions $\mathcal{Q}_t$ are Lipschitz continuous on $\mathcal{X}_{t-1}$.
Differentiability and Assumption (H1)-(e) are  useful to derive inexact cuts, see 
Sections \ref{iddp2}-\ref{iddp4}, in particular Lemma \ref{boundcutcoeff}. 

The Inexact Dual Dynamic Programming (IDDP) algorithm
to be presented in the next section is a solution method for problem \eqref{defpb} that exploits the convexity of $\mathcal{Q}_t, t=2,\ldots,T$.

\subsection{Inexact Dual Dynamic Programming: overview}\label{iddp2}

Similarly to DDP, to solve problem \eqref{defpb}, the Inexact Dual Dynamic Programming algorithm 
approximates for each $t=2,\ldots,T+1$, the function $\mathcal{Q}_{t}$ by
a polyhedral lower approximation $\mathcal{Q}_t^k$ at iteration $k$.

We start at the first iteration with 
the lower approximation $\mathcal{Q}_{t}^0=-\infty$ for
$\mathcal{Q}_{t}, t=2,\ldots,T$. At the beginning of iteration $k$, we have the lower polyhedral approximations (computed at previous iterations)
$\mathcal{Q}_t^{k-1}$ for $\mathcal{Q}_t$, whose computations are detailed below.

For convenience, for $t=1,\ldots,T,$ and $k \geq 0$,
let $F_t^{k}(y, x)=f_t(y, x  ) + \mathcal{Q}_{t+1}^{k}(y)$ and 
let ${\underline{\mathcal{Q}}}_t^{k}:\mathcal{X}_{t-1} \rightarrow \mathbb{R}$ given by
\begin{equation} \label{parametric}
{\underline{\mathcal{Q}}}_t^{k}(  x  ) = 
\left\{
\begin{array}{l}
\displaystyle \inf_{y \in \mathbb{R}^n}\;F_t^{k}(y, x)\\
y \in X_t( x ). 
\end{array}
\right.
\end{equation}

Iteration $k$ starts with a forward pass: for $t=1,\ldots,T$, we compute an $\varepsilon_t^k$-optimal solution $x_t^k$ of
\begin{equation} \label{forward1}
{\underline{\mathcal{Q}}}_{t}^{k-1}(x_{t-1}^k )=
\left\{
\begin{array}{l}
\displaystyle \inf_{y}\;F_t^{k-1}(y, x_{t-1}^k)\\
y \in X_t( x_{t-1}^k),
\end{array}
\right.
\end{equation}
starting from $x_0^{k}=x_0$
where $F_t^{k-1}(y, x_{t-1}^k)=f_t(y, x_{t-1}^k  ) + \mathcal{Q}_{t+1}^{k-1}(y)$ and knowing that $\mathcal{Q}_{T+1}^{k-1}=\mathcal{Q}_{T+1}\equiv 0$.
Therefore, we have 
\begin{equation} \label{epssolforward}
{\underline{\mathcal{Q}}}_{t}^{k-1}(x_{t-1}^k ) \leq F_t^{k-1}( x_t^k , x_{t-1}^k) \leq {\underline{\mathcal{Q}}}_{t}^{k-1}(x_{t-1}^k ) + \varepsilon_t^k.
\end{equation}

At iteration $k$, a backward pass then computes a cut $\mathcal{C}_t^k$ for $\mathcal{Q}_t$ at $x_{t-1}^k$ for $t=T+1$
down to $t=2$. For $t=T+1$, the cut is exact: $\mathcal{C}_{T+1}^k \equiv 0$.
For step $t<T+1$, we compute an $\varepsilon_t^k$-optimal solution $x_t^{B k} \in X_t ( x_{t-1}^k)$ of
\begin{equation} \label{backward1}
{\underline{\mathcal{Q}}}_{t}^{k}(x_{t-1}^k )=
\left\{
\begin{array}{l}
\displaystyle \inf_{y}\;F_t^{k}(y, x_{t-1}^k)\\
y \in X_t ( x_{t-1}^k),
\end{array}
\right.
\end{equation}
knowing $\mathcal{Q}_{t+1}^{k}$. It follows that
\begin{equation}\label{epssolprimal}
x_{t}^{B k} \in X_t( x_{t-1}^k ) \mbox{ and }{\underline{\mathcal{Q}}}_{t}^{k}(x_{t-1}^k ) \leq F_t^{k}( x_t^{B k}  , x_{t-1}^k ) \leq {\underline{\mathcal{Q}}}_{t}^{k}(x_{t-1}^k )  + \varepsilon_t^k.
\end{equation}
If $X_t$ is of type $S2$ we also compute an $\varepsilon_t^k$-optimal solution  $(\lambda_t^k, \mu_t^k)$  of the  dual problem
\begin{equation}\label{dualproblemS2}
\left\{
\begin{array}{l}
\sup \; h_{t, x_{t-1}^k}^k (\lambda, \mu) \\
\lambda = A_t y +B_t x_{t-1}^k -b_t, y \in \mbox{Aff}( \mathcal{X}_t ),\; \mu \in \mathbb{R}_{+}^p
\end{array}
\right.
\end{equation}
for the dual function
\begin{equation}\label{dualiddp}
h_{t, x_{t-1}^k }^k (\lambda, \mu) = 
\left\{
\begin{array}{l}
\inf \;F_t^k(y, x_{t-1}^k ) + \lambda^T(  A_t y + B_t x_{t-1}^k  - b_t ) + \mu^T g_t(y, x_{t-1}^k )\\
y \in \mathcal{X}_t.
\end{array}
\right.
\end{equation}
We now check that Assumption (H1) implies that the following Slater type constraint qualification condition holds for problem \eqref{backward1}
(i.e. for all problems solved in the backward passes):
\begin{equation}\label{slater}
\mbox{there exists }{\tilde x}_t^k \in \mbox{ri}(\mathcal{X}_{t}) \mbox { such that }A_t {\tilde x}_t^k + B_t x_{t-1}^k = b_t \mbox{ and }g_t({\tilde x}_t^k,  x_{t-1}^k)<0.
\end{equation}
The above constraint qualification condition is the analogue of \eqref{slaterboundmulti} for problem \eqref{backward1}.
\begin{lemma}\label{boundcutcoeff} 
Let Assumption (H1) hold. Then for every $k \in \mathbb{N}^*$, \eqref{slater} holds. 
\end{lemma}
\begin{proof}
If $x_{t-1}^k = {\bar x}_{t t-1}$ then recalling (H1)-(e), \eqref{slater} holds with ${\tilde x}_t^k  = {\bar x}_{t t}$.
Otherwise, we define
$$
x_{t-1}^{k \varepsilon}  = x_{t-1}^k + \varepsilon \frac{x_{t-1}^k - {\bar x}_{t t-1}}{\|x_{t-1}^k - {\bar x}_{t t-1}\|}.
$$
Observe that since $x_{t-1}^k \in \mathcal{X}_{t-1}$, we have $x_{t-1}^{k \varepsilon} \in \mathcal{X}_{t-1}^{\varepsilon}$. Setting
$$
X_t=\{(x_t, x_{t-1}) \in \mbox{ri}(\mathcal{X}_{t}) \small{\times} \mathcal{X}_{t-1}^{\varepsilon}: A_t x_t + B_t x_{t-1} = b_t,\;g_t(x_t, x_{t-1}) \leq 0   \},
$$
since $x_{t-1}^{k \varepsilon} \in \mathcal{X}_{t-1}^{\varepsilon}$, using (H1)-(d), there exists
$x_{t}^{k \varepsilon} \in \mbox{ri}(\mathcal{X}_t )$ such that
$(x_{t}^{k \varepsilon}, x_{t-1}^{k \varepsilon}) \in X_t$.
Now clearly, since $\mathcal{X}_t$ and $\mathcal{X}_{t-1}$ are convex, the set $\mbox{ri}(\mathcal{X}_{t}) \small{\times} \mathcal{X}_{t-1}^{\varepsilon}$
is convex too and using (H1)-(c), we obtain that $X_t$ is convex.
Since $({\bar x}_{t t}, {\bar x}_{t t-1}) \in X_t$ (due to Assumption (H1)-(e)) 
and recalling that $(x_{t}^{k \varepsilon}, x_{t-1}^{k \varepsilon}) \in X_t$, we obtain
that for every $0<\theta<1$, the point 
\begin{equation}\label{defxttheta}
(x_t(\theta), x_{t-1}(\theta))=(1-\theta) ({\bar x}_{t t}, {\bar x}_{t t-1}) + \theta (x_{t}^{k \varepsilon}, x_{t-1}^{k \varepsilon}) \in X_t. 
\end{equation}
For 
\begin{equation}\label{deftheta0}
0<\theta=\theta_0=\frac{1}{1+\frac{\varepsilon_0}{2 \|x_{t-1}^k - {\bar x}_{t t-1}\|}}<1,
\end{equation}
we get $x_{t-1}(\theta_0)=x_{t-1}^k$, $x_t(\theta_0 ) \in \mbox{ri}(\mathcal{X}_{t}), 
A_t x_t(\theta_0 ) + B_t x_{t-1}(\theta_0 ) = A_t x_t(\theta_0 ) + B_t x_{t-1}^k = b_t$, 
and since $g_{t i}, i=1,\ldots,p$, are convex on $\mathcal{X}_t \small{\times} \mathcal{X}_{t-1}^{\varepsilon}$  (see Assumption (H1)-(c)) and therefore on $X_t$, we get 
$$
\begin{array}{lll}
g_t(x_t(\theta_0), x_{t-1}(\theta_0)) & = & g_t(x_t(\theta_0), x_{t-1}^k) \\  
&\leq & \underbrace{(1-\theta_0)}_{>0} \underbrace{g_t( {\bar x}_{t t}, {\bar x}_{t t-1} )}_{<0} + 
\underbrace{\theta_0}_{>0} \underbrace{g_t( x_{t}^{k \varepsilon}, x_{t-1}^{k, \varepsilon } )}_{\leq 0} <0.
\end{array}
$$
We have justified that \eqref{slater} holds with ${\tilde x}_t^k = x_t(\theta_0 )$.\hfill
\end{proof}
From \eqref{slater}, we deduce that the optimal value ${\underline{\mathcal{Q}}}_{t}^{k}(x_{t-1}^k )$ of primal problem \eqref{backward1}
is the optimal value of dual problem \eqref{dualproblemS2} and therefore $\varepsilon_t^k$-optimal dual solution
$(\lambda_t^k , \mu_t^k)$ satisfies:
\begin{equation}\label{defepssoldual}
{\underline{\mathcal{Q}}}_{t}^{k}(x_{t-1}^k ) - \varepsilon_t^k  \leq   h_{t, x_{t-1}^k}^k (\lambda_t^k , \mu_t^k ) \leq {\underline{\mathcal{Q}}}_{t}^{k}(x_{t-1}^k ). 
\end{equation}
We now intend to use the results of Section \ref{sec:computeinexactcuts} to derive an inexact cut $\mathcal{C}_t^k$
for $\mathcal{Q}_t$ at $x_{t-1}^k$. Since for all iteration $k$ the relation $\mathcal{Q}_t \geq {\underline{\mathcal{Q}}}_t^k$
is preserved, $\mathcal{C}_t^k$ will in fact be an inexact cut for ${\underline{\mathcal{Q}}}_t^k$ and therefore for 
$\mathcal{Q}_t$. To proceed, let us write function $\mathcal{Q}_{t+1}^k$, which is a maximum of $k$
affine functions, in the form
$$
\mathcal{Q}_{t+1}^{k}(x_{t}) =\displaystyle \max_{1 \leq j \leq k} \Big(\mathcal{C}_{t+1}^j(x_{t}) := \theta_{t+1}^j -  \eta_{t+1}^j (  \varepsilon_{t+1}^j ) + \langle \beta_{t+1}^{j}, x_{t} -  x_{t}^j \rangle\Big)
$$
for some coefficients $\theta_{t+1}^j, \eta_{t+1}^j (  \varepsilon_{t+1}^j )$, and  $\beta_{t+1}^{j}$ whose iterative computation is detailed below
with the convention that for $t=T$ coefficients $\theta_{t+1}^j, \eta_{t+1}^j (  \varepsilon_{t+1}^j ), \beta_{t+1}^{j}$ 
are all zero.
Plugging this representation into \eqref{backward1}, we get
\begin{equation} \label{backward1bis}
{\underline{\mathcal{Q}}}_{t}^{k}(x_{t-1}^k )=
\left\{
\begin{array}{l}
\displaystyle \inf_{x_t, y_t}\;f_t(x_t, x_{t-1}^k) + y_t\\
x_t \in X_t( x_{t-1}^k ),\\
y_t \geq \theta_{t+1}^j -  \eta_{t+1}^j (  \varepsilon_{t+1}^j ) + \langle \beta_{t+1}^{j}, x_t -  x_{t}^j \rangle,j=1,\ldots,k,
\end{array}
\right.
\end{equation}
which is of form \eqref{vfunctionq} with 
$$
\begin{array}{l}
y=(x_t,y_t), x=x_{t-1}^k, f(y,x)=f_t(x_t,x)+y_t, Y=\{y=[x_t ; y_t ] :x_t \in \mathcal{X}_t, B_{t+1}^k y \leq b_{t+1}^k \},
\end{array}
$$
and for constraints of type $S2$
$$
\begin{array}{l}
A=[A_t \;0_{q \small{\times} 1} ], B=B_t, b=b_t, g(y,x)=g_t(x_t,x),
\end{array}
$$
where the $j$-th line of matrix $B_{t+1}^k$ is $[(\beta_{t+1}^j)^T, -1]$
and the $j$-th component of $b_{t+1}^k$ is $-\theta_{t+1}^j +  \eta_{t+1}^j (  \varepsilon_{t+1}^j ) - \langle \beta_{t+1}^{j} x_{t}^j \rangle$.
We can now  use the results of Section \ref{sec:computeinexactcuts} and consider several cases depending on the problem structure.

\subsection{Computation of inexact cuts in the backward pass for constraints of type $S1$}\label{iddp3}

Let us first consider the case where $X_t$ is of type $S1$.
Let $(x_t^{B k}, y_t^{B k})$ be an $\varepsilon_t^k$-optimal solution of 
\begin{equation} \label{backward1ter}
{\underline{\mathcal{Q}}}_{t}^{k}(x_{t-1}^k )=
\left\{
\begin{array}{l}
\displaystyle \inf_{x_t, y_t}\;f_t(x_t, x_{t-1}^k) + y_t\\
x_t \in \mathcal{X}_t, B_{t+1}^k \left[ \begin{array}{c}x_t\\y_t\end{array} \right]\leq b_{t+1}^k.\\
\end{array}
\right.
\end{equation}
We compute
$$
\theta_t^k=f_t( x_t^{B k}, x_{t-1}^k)+ y_t^{B  k},\; \eta_t^k( \varepsilon_t^k ) = \ell_{1 t}^{k}( x_t^{B  k} , y_t^{B  k},  x_{t-1}^k ), \beta_t^k=\nabla_{x_{t-1}} f_t( x_t^{B  k}, x_{t-1}^k ), 
$$
where 
\begin{equation}\label{defetaiddp0}
\ell_{1 t}^{k}( x_t^{B  k} , y_t^{B  k},  x_{t-1}^k )= 
\left\{
\begin{array}{l}
\max_{x_t, y_t}  \langle \nabla_{x_t} f_t( x_t^{B  k}, x_{t-1}^k ), x_t^{B  k} - x_t \rangle + y_t^{B  k} -y_t\\
x_t \in \mathcal{X}_t, B_{t+1}^k \left[ \begin{array}{c}x_t\\y_t\end{array} \right]  \leq b_{t+1}^k.
\end{array}
\right.
\end{equation}
Using Proposition \ref{fixedprop1} we have that
$\mathcal{C}_{t}^k(x_{t-1}) = \theta_{t}^k -  \eta_{t}^k (  \varepsilon_{t}^k ) + \langle \beta_{t}^{k}, x_{t-1} -  x_{t-1}^k \rangle$
is an inexact cut for ${\underline{\mathcal{Q}}}_{t}^{k}$ and therefore for $\mathcal{Q}_t$.
Moreover, the distance between ${\underline{\mathcal{Q}}}_{t}^{k}( x_{t-1}^k ) $  and $\mathcal{C}_{t}^k(x_{t-1}^k)$ is at most 
$\eta_{t}^k (  \varepsilon_{t}^k ) = \ell_{1 t}^{k}( x_t^{B  k} , y_t^{B  k},  x_{t-1}^k )$.

\subsection{Computation of inexact cuts in the backward pass for constraints of type $S2$}\label{iddp4}

We now consider the case where $X_t$ is of type $S2$.
Let $(x_t^{B  k}, y_t^{B  k})$ be an $\varepsilon_t^k$-optimal solution of 
\begin{equation} \label{backward1terS2}
{\underline{\mathcal{Q}}}_{t}^{k}(x_{t-1}^k )=
\left\{
\begin{array}{l}
\displaystyle \inf_{x_t, y_t}\;f_t(x_t, x_{t-1}^k) + y_t\\
x_t \in X_t( x_{t-1}^k ), B_{t+1}^k \left[ \begin{array}{c}x_t\\y_t\end{array} \right]\leq b_{t+1}^k.
\end{array}
\right.
\end{equation}
Define for problem \eqref{backward1terS2} the Lagrangian 
$$
L_{x_{t-1}^k}(x_t, y_t, \lambda, \mu)=f_t(x_t, x_{t-1}^k) + y_t
+\lambda^T ( A_t x_t + B_t x_{t-1}^k - b_t) +\mu^T  g_t(x_t, x_{t-1}^k )
$$
and 
\begin{equation}\label{defetaiddp}
\ell_{2 t}^{k}({x}_t^{B  k}, y_t^{B  k}, x_{t-1}^k, \lambda, \mu)=
\left\{
\begin{array}{l}
\max_{x_t, y_t} \;\langle \nabla_{x_t} L_{x_{t-1}^k}({x}_t^{B  k} , {y}_t^{B  k}, \lambda, \mu),x_t^{B  k} - x_t   \rangle + y_t^{B  k} - y_t   \\
x_t \in \mathcal{X}_t, \; B_{t+1}^k \left[ \begin{array}{c}x_t\\y_t\end{array} \right]\leq b_{t+1}^k.
\end{array}
\right.
\end{equation}
With this notation and recalling that $(\lambda_t^k, \mu_t^k)$ is an $\varepsilon_t^k$-optimal solution of \eqref{dualproblemS2} we put
\begin{equation} \label{defbetatk}
\begin{array}{l}
\theta_t^k = L_{x_{t-1}^k}(x_t^{B  k}, y_t^{B  k}, \lambda_t^k, \mu_t^k ),\;\;\;\eta_{t}^k (  \varepsilon_{t}^k )  = \ell_{2 t}^{k}(x_t^{B  k}, y_t^{B  k}, x_{t-1}^k, \lambda_t^k , \mu_t^k ),\\
\beta_t^k = \nabla_{x_{t-1}} f_t(x_t^{B  k}, x_{t-1}^{k})+ B_t^T \lambda_{t}^{k} +  \sum_{i=1}^{p}\; \mu_t^{k}(i) \nabla_{x_{t-1}} g_{t i}(x_t^{B  k}, x_{t-1}^{k}).
\end{array}
\end{equation}
Using Proposition \ref{varprop1}, the affine function
$$
\mathcal{C}_t^k( x_{t-1}) = \theta_{t}^k -  \eta_{t}^k (  \varepsilon_{t}^k ) + \langle \beta_t^k ,  x_{t-1} - x_{t-1}^k  \rangle
$$
defines an inexact cut for $\mathcal{Q}_t$.
Moreover, the distance between ${\underline{\mathcal{Q}}}_{t}^{k}( x_{t-1}^k ) $  and $\mathcal{C}_{t}^k(x_{t-1}^k)$ is at most 
$\varepsilon_t^k + \ell_{2 t}^{k}(x_t^{B  k}, y_t^{B  k}, x_{t-1}^k, \lambda_t^k , \mu_t^k )=\varepsilon_t^k + \eta_{t}^k (  \varepsilon_{t}^k )$.\\

\par For IDDP, we assume that nonlinear optimization problems (such as primal problems \eqref{backward1ter}, \eqref{backward1terS2}
or dual problem \eqref{dualproblemS2}) are solved approximately
whereas linear optimization problems are solved exactly. 
Notice that we assumed that we can compute the optimal value $\ell_{1 t}^{k}( x_t^{B  k} , y_t^{B  k},  x_{t-1}^k )$
of optimization problem \eqref{defetaiddp0} and the optimal value
$\ell_{2 t}^{k}({x}_t^{B  k}, y_t^{B  k}, x_{t-1}^k, \lambda_t^k , \mu_t^k)$
of optimization problem \eqref{defetaiddp} written for $(\lambda, \mu) = ( \lambda_t^k , \mu_t^k )$. Since these optimization problems have a linear objective function, they are linear programs if and only if 
$\mathcal{X}_t$ is polyhedral. If this is not the case then 
\begin{itemize}
\item[a)] either we add components  to $g$ 
pushing
the nonlinear constraints in the representation of $\mathcal{X}_t$ in $g$ or
\item[b)] we also solve approximately \eqref{defetaiddp0} and \eqref{defetaiddp}.  
\end{itemize}
In Case b), we can still build an inexact cut $\mathcal{C}_t^k$ and study the convergence of the corresponding variant of IDDP along the lines of Section \ref{iddp5}.
More precisely, in this situation, we obtain cut $\mathcal{C}_t^k$ using 
Proposition \ref{fixedprop2} instead of
Proposition \ref{fixedprop1} if $X_t$ is of type $S1$. If $X_t$ is of type $S2$
we can use the extension of Proposition \ref{varprop1} obtained 
when \eqref{defrxy} is solved approximately, exactly as was done for the extension 
of Proposition \ref{fixedprop1} corresponding to Proposition \ref{fixedprop2}.

\subsection{Convergence analysis}\label{iddp5}

The main result of this section is  Theorem \ref{convaniddp}, a convergence analysis of IDDP.

We will use the following immediate observation:
\begin{lemma} \label{convrecfuncQt} For $t=2,\ldots,T+1$, function $\mathcal{Q}_t$ is convex and Lipschitz continuous on 
$\mathcal{X}_{t-1}$.
\end{lemma}
\begin{proof} The proof is by backward induction on $t$. The result holds for $t=T+1$ by definition of $\mathcal{Q}_{T+1}$.
Let us now assume that $\mathcal{Q}_{t+1}$ is convex and Lipschitz continuous on 
$\mathcal{X}_{t}$ for some $t \in \{2,\ldots,T\}$. We consider two cases: $X_t$ is of type $S1$ (Case A) and
$X_t$ is of type $S2$ (Case B).

\par {\textbf{Case A.}} Convexity of $\mathcal{Q}_t$ immediately follows from (H1)-(a),(b).
(H1)-(b) implies that $f_t$ is continuous on the compact set $\mathcal{X}_t \small{\times} \mathcal{X}_{t-1}$
and therefore takes finite values on $\mathcal{X}_t \small{\times} \mathcal{X}_{t-1}$ but also on some
neighborhood $\mathcal{X}_t \small{\times} \mathcal{X}_{t-1}^{\varepsilon_0}$ of $\mathcal{X}_t \small{\times} \mathcal{X}_{t-1}$
with $\varepsilon_0>0$. Therefore, for every $x_{t-1} \in \mathcal{X}_{t-1}^{\varepsilon_0}$, we have that
$x_t \rightarrow f_t(x_t, x_{t-1} ) + \mathcal{Q}_{t+1}(x_t)$ is finite-valued on $\mathcal{X}_t$, and $\mathcal{Q}_t( x_{t-1} )$ is finite.
\par {\textbf{Case B.}} Convexity of $\mathcal{Q}_t$ immediately follows from (H1)-(a),(b), (c).
As in Case A, $f_t$ is finite valued on $\mathcal{X}_t \small{\times} \mathcal{X}_{t-1}^{\varepsilon_0}$ 
for some $\varepsilon_0>0$. Combining this observation with (H1)-(d), for every $x_{t-1} \in \mathcal{X}_t^{\min(\varepsilon_0, \varepsilon)}$
the function $x_t \rightarrow f_t(x_t, x_{t-1} ) + \mathcal{Q}_{t+1}(x_t)$ is finite-valued on the nonempty set $X_t( x_{t-1} )$
and therefore $\mathcal{Q}_t( x_{t-1}  )$ is finite.

In  both Cases (A) and (B) we checked that $\mathcal{X}_{t-1}$ is contained in the interior of the domain of $\mathcal{Q}_t$
which implies that convex function $\mathcal{Q}_t$ is Lipschitz continuous on $\mathcal{X}_{t-1}$.\hfill
\end{proof}
In view of Lemma \ref{convrecfuncQt}, we will denote by $L( \mathcal{Q}_t )$ a Lipschitz constant for $\mathcal{Q}_t$ for $t=2,\ldots,T+1$.\\

A useful ingredient for the convergence analysis of IDDP is the boundedness of the sequences of approximate
dual solutions $(\lambda_t^k, \mu_t^k)$. Recall that if $X_t$ is of type $S2$ then Slater constraint qualification
\eqref{slater} holds. 
From Theorem 2.3.2, p.312 in \cite{hhlem}, we deduce that if the rows of $A_t$ are independent then
the set of optimal dual solutions of problem \eqref{dualproblemS2} is bounded.
Therefore, the level set of $-h_{t, x_{t-1}^k }^k$ associated to its minimal value is bounded implying that the level set
associated to this minimal value plus $\varepsilon_t^k$
is bounded too (since for a convex function if a level set is bounded then all level
sets are bounded). It follows that 
if the rows of $A_t$ are independent, then
for every $k \in \mathbb{N}^*$ the norm $\|(\lambda_t^k, \mu_t^k)\|$ is finite.

To obtain an upper bound on the sequence $(\|(\lambda_t^k, \mu_t^k)\|)_{t k}$
we will use a slightly stronger assumption than (H1)-(e), namely we will assume:\\
\par (H2) For $t=2, \ldots, T$, there exists $\kappa_t>0, r_t>0$ such that for every $x_{t-1} \in \mathcal{X}_{t-1}$, 
there exists $x_t \in \mathcal{X}_t$ such that $\mathbb{B}(x_t, r_t) \cap \mbox{Aff}( \mathcal{X}_t ) \neq \emptyset$,
$A_t x_t + B_t x_{t-1}=b_t$, and for every $i=1,\ldots,p$, $g_{t i}( x_t, x_{t-1}) \leq -\kappa_t$. \\
\begin{rem} Of course, by definition of the relative interior, the condition 
$\mathbb{B}(x_t, r_t) \cap \mbox{Aff}( \mathcal{X}_t ) \neq \emptyset$ implies that
$x_t \in \mbox{ri}( \mathcal{X}_t )$.
\end{rem}
However, we do not assume that the rows of $A_t$ are independent. Using (H2) and
Section \ref{sec:boundingmulti} we can
now show that the sequences of cut coefficients and approximate dual solutions belong to a compact set:
\begin{prop} \label{propboundeddual} Assume that noises $(\varepsilon_t^k)_{k \geq 1}$ are bounded: for 
$t=2,\ldots,T$, we have $0 \leq \varepsilon_t^k \leq {\bar \varepsilon}_t< +\infty$.
If Assumptions (H1) and (H2) hold then the sequences 
$(\theta_{t}^k)_{t ,k}$, $(\eta_{t}^k (  \varepsilon_{t}^k ))_{t, k}$, $( \beta_t^k )_{t, k}$, 
$(\lambda_{t}^k )_{t, k}$, $( \mu_{t}^k )_{t, k}$
 generated by the IDDP algorithm are bounded: for $t=2,\ldots,T+1$, there exists
 a compact set $C_t$ such that the sequence  
 $(\theta_{t}^k , \eta_{t}^k (  \varepsilon_{t}^k ), \beta_t^k )_{k \geq 1}$ belongs to 
 $C_t$ and for $t=2,\ldots,T$, if $X_t$ is of type $S2$ then
 there exists
 a compact set $\mathcal{D}_t$ such that the sequence  
 $( \lambda_{t}^k , \mu_t^k )_{k \geq 1}$ belongs to 
 $\mathcal{D}_t$.
\end{prop}
\begin{proof} The proof is by backward induction on $t$. Our induction hypothesis 
$\mathcal{H}(t)$
for 
$t \in \{2,\ldots,T+1\}$ is that the sequence $(\theta_{t}^k , \eta_{t}^k (  \varepsilon_{t}^k ), \beta_t^k )_{k \geq 1}$ belongs to 
a compact set $C_t$. We have that $\mathcal{H}(T+1)$ holds because for $t=T+1$ the corresponding coefficients are all zero.
Now assume that $\mathcal{H}(T+1)$ holds for some $t \in \{2,\ldots,T+1\}$. We want to show that $\mathcal{H}(t)$
holds and if $X_t$ is of type $S2$ 
that the sequence  $( \lambda_{t}^k , \mu_t^k )_{k \geq 1}$ belongs to 
some compact set $\mathcal{D}_t$. Since $f_t$ and $g_t$ belong to  $\mathcal{C}^1( \mathcal{X}_t \small{\times} \mathcal{X}_{t-1} )$
we can find finite $m_t, M_{t 1}, M_{t 2}, M_{t 3}, M_{t 4}$ such that for every $x_t \in \mathcal{X}_t, x_{t-1} \in \mathcal{X}_{t-1}$, 
for every $i=1,\ldots,p$, we have
$$
m_t \leq f_t( x_t , x_{t-1} ) \leq M_{t 1},\; \| \nabla f_t ( x_t , x_{t-1})  \| \leq M_{t 2},\;\| \nabla g_{t i} ( x_t , x_{t-1} ) \| \leq M_{t 3},\;
\| g_t ( x_t , x_{t-1} ) \| \leq M_{t 4}.
$$
Also since $\mathcal{H}(t+1)$ holds, the sequence $(\|\beta_{t+1}^k\|)_{k \geq 1}$ is bounded
from above by, say, $L_{t+1}$, which is a Lipschitz constant for all
functions $(\mathcal{Q}_{t+1}^k )_{k \geq 1}$.

We now consider two cases: $X_t$ is of type $S1$ (Case A) and $X_t$ is of type $S2$ (Case B).

\par {\textbf{Case A.}} We have $\theta_t^k = f_t ( x_t^{B k} , x_{t-1}^k ) + \mathcal{Q}_{t+1}^k ( x_t^{B k} )$ which gives 
the bound
$$
m_t + \min_{x_t \in \mathcal{X}_t} \mathcal{Q}_{t+1}^1 ( x_t ) \leq \theta_t^k \leq M_{t 1} + \max_{x_t \in \mathcal{X}_t } \mathcal{Q}_{t+1}( x_t ),\;\forall  k \geq 1,
$$
(recall that due to $\mathcal{H}(t+1)$ and Lemma \ref{convrecfuncQt}, the minimum and maximum in the relation above
are well defined because functions $\mathcal{Q}_{t+1}^1$ and $\mathcal{Q}_{t+1}$ are continuous on the compact $\mathcal{X}_t$).

Now for $\eta_t^k( \varepsilon_t^k ) = \ell_{1 t}^{k}( x_t^{B  k} , y_t^{B  k},  x_{t-1}^k )$ and recalling definition
\eqref{defetaiddp0} of $\ell_{1 t}^{k}( x_t^{B  k} , y_t^{B  k},  x_{t-1}^k )$, we see that
\begin{equation}\label{boundetafirst}
0 \leq \eta_t^k( \varepsilon_t^k )  \leq {\bar \eta}_t := (M_{t 2} + L_{t+1}) D(\mathcal{X}_t),\;\forall \; k\geq 1,
\end{equation}
and of course the norm of $\beta_t^k=\nabla_{x_{t-1}} f_t( x_t^{B  k}, x_{t-1}^k )$ for all $k \geq 1$ is bounded from above by $M_{t 2}$.
This shows $\mathcal{H}(t)$ for Case A.

\par {\textbf{Case B.}}
We first obtain a bound on $\|(\lambda_t^k , \mu_t^k )\|$ using Proposition \ref{propboundnormepsdualsol}  
and Corollary \ref{corpropboundnormepsdualsol}. Let us check that the Assumptions of this corollary
are satisfied for problem \eqref{backward1terS2}: 
\begin{itemize}
\item[(i)] $\mathcal{X}_t$ is a closed convex set;
\item[(ii)] the objective function $F_t^k(\cdot, x_{t-1}^k )$ is bounded from above by
${\bar f}(\cdot)=f_t(\cdot, x_{t-1}^k ) + \mathcal{Q}_{t+1}( \cdot )$. Since 
$f_t$ is convex and finite in a neighborhood of $\mathcal{X}_t \small{\times} \mathcal{X}_{t-1}$,
it is Lipschitz continuous on $\mathcal{X}_t \small{\times} \mathcal{X}_{t-1}$ with Lipschitz constant,
say, $L(f_t)$. Therefore ${\bar f}$ is Lipschitz continuous with Lipschitz constant
$L(f_t)+L(\mathcal{Q}_{t+1})$ on $\mathcal{X}_t$.
\item[(iii)] Since all components of $g_t$ are convex and finite in a neighborhood of $\mathcal{X}_t \small{\times} \mathcal{X}_{t-1}$,
they are  Lipschitz continuous on $\mathcal{X}_t \small{\times} \mathcal{X}_{t-1}$.
\item[(iv)] The objective function is bounded on the feasible set 
by $\mathcal{L}= \displaystyle \min_{x_{t-1} \in \mathcal{X}_{t-1}} {\underline{\mathcal{Q}}}_t^1 ( x_{t-1} )$ (the minimum is well defined due to Assumption (H1)).
\end{itemize}
Due to Assumption (H2) we can find $\hat x_t^k \in \mbox{ri}( \mathcal{X}_t )$ such that $\hat x_t^k \in X_t( x_{t-1}^k )$ and
$\mathbb{B}_n( \hat x_t^k , r_t ) \cap \mbox{Aff}( \mathcal{X}_t ) \neq \emptyset$.
Therefore, reproducing the reasoning of Section \ref{sec:boundingmulti}, we can find
$\rho_t$ such that 
$$
\mathbb{B}_q( 0, \rho_t ) \cap A_t V_{\mathcal{X}_t}  \subseteq A_t \Big( \mathbb{B}_n( 0 , r_t )  \cap  V_{\mathcal{X}_t} \Big)
$$
where $V_{\mathcal{X}_t}$ is the vector space $V_{\mathcal{X}_t}=\{x-y,\;x,y \in \mbox{Aff}( \mathcal{X}_t )\}$
(this is relation \eqref{deductionfromSLrewr} for problem \eqref{backward1terS2}).
Applying Corollary \ref{corpropboundnormepsdualsol} to problem \eqref{backward1terS2} we deduce that
$\|(\lambda_t^k , \mu_t^k )\| \leq U_t $
where 
$$
U_t
= \frac{
(  L(f_t)   + L( \mathcal{Q}_{t+1} )  ) r_t + {\bar{\varepsilon}}_t +
\displaystyle  \max_{x_t \in \mathcal{X}_t, x_{t-1} \in \mathcal{X}_{t-1}}(f_t( x_t ,x_{t-1} ) + \mathcal{Q}_{t+1}(x_t))   -\displaystyle  \min_{x_{t-1} \in \mathcal{X}_{t-1} } {\underline{\mathcal{Q}}}_t^1 ( x_{t-1})        }  {\min(\rho_t, \frac{\kappa_t}{2})}.
$$
For 
$\theta_t^k = f_t (x_t^{B  k}, x_{t-1}^k) + \mathcal{Q}_{t+1}^k ( x_t^{B k}  )   + \langle \mu_t^k  ,  g_t( x_t^{B k}, x_{t-1}^k ) \rangle$
we get the bound 
$$
m_t -U_t M_{t 4} +  \min_{x_t \in \mathcal{X}_t} \mathcal{Q}_{t+1}^1 ( x_t ) \leq \theta_t^k \leq M_{t 1}+   \max_{x_t \in \mathcal{X}_t} \mathcal{Q}_{t+1} ( x_t ).
$$
Note that $\eta_{t}^k (  \varepsilon_{t}^k ) \geq 0$ and 
the objective function of problem \eqref{defetaiddp} written for $(\lambda, \mu)= (\lambda_t^k , \mu_t^k)$  with optimal value 
$\eta_{t}^k (  \varepsilon_{t}^k )$ is bounded from above on the feasible set by
\begin{equation}\label{boundetasecond} 
{\bar \eta}_t = \Big( M_{t 2} + \sqrt{2} \max( \|A_t^T\| , M_{t 3} \sqrt{p} ) U_t   + L_{t+1} \Big) D( \mathcal{X}_t  )
\end{equation}
and therefore the same upper bound holds for $\eta_{t}^k (  \varepsilon_{t}^k )$.
Finally, recalling definition \eqref{defbetatk} of $\beta_t^k$ we have:
\begin{equation}\label{formulaLt}
\| \beta_t^k \| \leq M_{t 2} + \Big[ \|B_t^T\| \|\lambda_t^k\|  +  M_{t 3} \sqrt{p} \| \mu_t^k \| \Big] \leq 
L_t := M_{t 2} + \sqrt{2} \max( \|B_t^T\| , M_{t 3} \sqrt{p} ) U_t, 
\end{equation}
which completes the proof and provides a Lipschitz constant
 $L_t$ valid for functions $(\mathcal{Q}_t^k)_k$.\hfill
\end{proof}

To show that the sequence of error terms $(\eta_t^k (  \varepsilon_t^k ))_k$
converges to 0 when $\lim_{k \rightarrow +\infty} \varepsilon_t^k = 0$,
we will make use of Propositions \ref{propvanish1} and \ref{propvanish1dual} which follow:
\begin{prop}\label{propvanish1} Let $X \subset \mathbb{R}^m, Y \subset \mathbb{R}^n$,  be two nonempty compact convex sets.
Let $f \in \mathcal{C}^1 (  Y\small{\times}X )$ be convex  on $Y\small{\times}X$.
Let $(\mathcal{Q}^k)_{k \geq 1}$ be a sequence of convex $L$-Lipschitz continuous functions
on $Y$ satisfying $\underline{\mathcal{Q}}  \leq \mathcal{Q}^k \leq {\bar{\mathcal{Q}}}$ on $Y$
where $\underline{\mathcal{Q}}, {\bar{\mathcal{Q}}}$ are continuous on $Y$.
Let $(x^k)_{k \geq 1}$ be a sequence in $X$, $(\varepsilon^k)_{k \geq 1}$ be a sequence of nonnegative real numbers, 
and let $y^k ( \varepsilon^k ) \in Y$ be an $\varepsilon^k$-optimal solution to 
\begin{equation}\label{defpblemmavanish1}
\inf \;\{f(y, x^k ) + \mathcal{Q}^k ( y) \;\;: \;\; y \in Y \} .
\end{equation}
Define
\begin{equation}\label{defpblemmavanish2}
\eta^k (  \varepsilon^k ) = \left\{ 
\begin{array}{l}
\max \; \langle \nabla_y f(y^k ( \varepsilon^k ) , x^k ) , y^k ( \varepsilon^k )  - y \rangle + \mathcal{Q}^k ( y^k ( \varepsilon^k  ) ) - \mathcal{Q}^k ( y ) \\
y \in Y.
\end{array}
\right.
\end{equation}
Then if $\lim_{k \rightarrow +\infty} \varepsilon^k = 0$ we have
\begin{equation}\label{vanishnoises1l}
\lim_{k \rightarrow +\infty} \eta^k (  \varepsilon^k ) = 0.
\end{equation}
\end{prop}
\begin{proof}
In what follows, to simplify notation, we write $y^k$ instead of $y^k ( \varepsilon^k )$.
We show \eqref{vanishnoises1l} by contradiction.
Denoting by $y_*^{k} \in  Y $ an optimal solution of \eqref{defpblemmavanish1}, we have for every $k \geq 1$ that
\begin{equation}\label{defepstkredefl1}
f( y_*^{k}, x^k ) + \mathcal{Q}^k ( y_*^{k} ) \leq 
f( y^{k}, x^k ) + \mathcal{Q}^k ( y^{k} ) \leq f( y_*^{k} , x^k ) + \mathcal{Q}^k ( y_*^{k} ) + \varepsilon^k.
\end{equation}
Denoting by ${\tilde y}^k  \in Y$ an optimal solution of optimization problem \eqref{defpblemmavanish2}  we get 
\begin{equation}\label{defetaiddp3}
\eta^k (  \varepsilon^k   )= \langle \nabla_y f(y^k  , x^k ) , y^k   - {\tilde y}^k \rangle + \mathcal{Q}^k ( y^k ) - \mathcal{Q}^k ( {\tilde y}^k ).
\end{equation}
Assume that \eqref{vanishnoises1l} does not hold. Then 
since $\eta^k (  \varepsilon^k   ) \geq 0$ there exists $\varepsilon_0 >0$ and $\sigma_1: \mathbb{N} \rightarrow \mathbb{N}$ increasing such that
for every $k \in \mathbb{N}$ we have
\begin{equation}\label{contradeta}
\eta^{\sigma_1(k)} (  \varepsilon^{\sigma_1(k)}   ) = \langle \nabla_{y} f(y^{\sigma_1(k)}, x^{ \sigma_1(k) } ),  -{\tilde y}^{ \sigma_1(k) }  + y^{\sigma_1(k)}  \rangle  + 
\mathcal{Q}^{\sigma_1(k)}( y^{\sigma_1(k)}  ) - \mathcal{Q}^{\sigma_1(k)}( {\tilde y}^{\sigma_1(k)}  ) \geq \varepsilon_0. 
\end{equation}
Now observe that the sequence $( \mathcal{Q}^{\sigma_1(k)} )_k$ in $\mathcal{C}( Y )$
\begin{itemize}
\item[(i)] is bounded: for every $k \geq 1$, for every $y \in Y$, we have
$$
-\infty < \min_{y \in Y} \underline{\mathcal{Q}}( y ) \leq \mathcal{Q}^{\sigma_1(k)}( y ) \leq \max_{y \in Y}   {\bar{\mathcal{Q}}}( y ) < +\infty;
$$
\item[(ii)] is equicontinuous since functions $(\mathcal{Q}^{\sigma_1(k)})_k$ are Lipschitz continous with Lipschitz constant $L$.
\end{itemize}
Therefore using the Arzel\`a-Ascoli theorem, this sequence has a uniformly convergent subsequence: there exists
$\mathcal{Q}^* \in \mathcal{C}( Y )$ and $\sigma_2: \mathbb{N} \rightarrow \mathbb{N}$ increasing 
such that setting $\sigma =\sigma_1 \circ \sigma_2$, we have
$\lim_{k \rightarrow +\infty} \|\mathcal{Q}^{\sigma(k)} - \mathcal{Q}^*  \|_{Y} = 0$.
Since $({y}^{\sigma(k)}, y_*^{\sigma(k)}, {\tilde y}^{ \sigma(k) }, x^{ \sigma(k) } )_{k \geq 1}$ is a sequence of the compact set 
$Y \small{\times} Y \small{\times} Y \small{\times} X$, taking further a subsequence if needed,
we can assume that $({y}^{\sigma(k)}, y_*^{\sigma(k)}, {\tilde y}^{ \sigma(k) }, x^{ \sigma(k) } )$ converges
to some $({\bar y} , y_* , {\tilde y} , x_* ) \in Y \small{\times} Y \small{\times} Y \small{\times} X$.
By continuity arguments, for $k$ sufficiently large, say $k \geq k_0$, we have that
\begin{equation}\label{gradlipetaproof}
\begin{array}{l}
|\langle \nabla_{y} f({y}^{\sigma(k)}, x^{ \sigma(k) } ),  -{\tilde y}^{ \sigma(k) }  + y^{\sigma(k)}  \rangle - \langle \nabla_{y} f({{\bar y}} , x_* ),-{\tilde y}^{\sigma(k)}   + {\bar y}  \rangle | \leq \varepsilon_0 /4,\\
\|y^{\sigma(k)} - {\bar y}\| \leq \frac{\varepsilon_0}{8 L},\;\|\mathcal{Q}^{\sigma(k)} - \mathcal{Q}^*  \|_{Y} \leq \varepsilon_0/16.
\end{array}
\end{equation}
It follows that
\begin{equation}\label{finaleq1}
\begin{array}{l}
\langle \nabla_{y} f({{\bar y}}, x_* ),  -{\tilde y}^{ \sigma(k_0 ) }  
+ {\bar y}  \rangle  +  \mathcal{Q}^{*}( {\bar y}  ) - \mathcal{Q}^{*}( \tilde y^{\sigma(k_0 )}  )\\
= \langle \nabla_{y} f({y}^{\sigma(k_0)}, x^{ \sigma(k_0) } ),  -{\tilde y}^{ \sigma(k_0) }  + y^{\sigma(k_0)}  \rangle  
+ \mathcal{Q}^{\sigma(k_0)}( y^{\sigma(k_0)}  ) - \mathcal{Q}^{\sigma(k_0)}( \tilde y^{\sigma(k_0)}  ) \\
\;\;\;+  \langle \nabla_{y} f( {\bar y}, x_* ),  -{\tilde y}^{ \sigma(k_0 ) }  + {\bar y}  \rangle - 
\langle \nabla_{y} f({y}^{\sigma(k_0)}, x^{ \sigma(k_0) } ),  -{\tilde y}^{ \sigma(k_0) }  + y^{\sigma(k_0)}  \rangle \\
\;\;\;+ [ \mathcal{Q}^{*}( {\bar y}  ) - \mathcal{Q}^{\sigma( k_0 )}( {\bar y}  )+ \mathcal{Q}^{\sigma( k_0 )}( {\bar y}  ) -  \mathcal{Q}^{\sigma(k_0)}( y^{\sigma(k_0)}  )  ] \\
\;\;\;- [ \mathcal{Q}^{*}( {\tilde y}^{\sigma(k_0 )})  - \mathcal{Q}^{\sigma(k_0)}( {\tilde y}^{\sigma(k_0)}  ) ],\\
\geq \varepsilon_0 - \frac{\varepsilon_0}{4} - 2 \|\mathcal{Q}^{*} - \mathcal{Q}^{\sigma(k_0)} \|_{Y} - L \| {\bar y} - y^{\sigma(k_0)}\|  \geq \frac{\varepsilon_0}{2}>0,
\end{array}
\end{equation}
where for the last two inequalities we have used \eqref{contradeta} and \eqref{gradlipetaproof}.

Recalling the definition of $y_*^{k}$, for every $k \geq 1$ we have that $y_*^{\sigma(k)}  \in Y$ and
$$
f( y_*^{\sigma(k)}, x^{\sigma(k)} ) + \mathcal{Q}^{\sigma(k)} ( y_*^{\sigma(k)} ) \leq f( y, x^{\sigma(k)} ) + \mathcal{Q}^{\sigma(k)} ( y ),  \;\;\forall y \in Y.
$$
Taking the limit as $k \rightarrow +\infty$ in the above inequality we get (using the continuity of $f$) 
$$
f_* :=f( y_{*}, x_* ) + \mathcal{Q}^{*} ( y_{*} ) \leq f( y , x_* ) + \mathcal{Q}^{*} ( y ),  \;\;\forall y \in Y.
$$
Since $y_* \in Y$,we have shown that  $y_*$ is an optimal solution for the optimization problem
\begin{equation}\label{optproblemxtsxtB}
f_* = \left\{
\begin{array}{l}
\min f( y, x_* ) + \mathcal{Q}^* ( y ) \\
y \in Y.
\end{array}
\right.
\end{equation}
Replacing $k$ by $\sigma(k)$ in \eqref{defepstkredefl1} and taking the limit as $k \rightarrow +\infty$, we obtain
$$
f_* = f( y_{*}, x_{*} ) + \mathcal{Q}^{*} ( y_{*} ) = f( {\bar y} , x_{*} ) + \mathcal{Q}^{*} ( {\bar y} ).
$$
Combining this observation with the fact that ${\bar y} \in Y$,
we deduce that ${\bar y}$ is also an optimal solution of \eqref{optproblemxtsxtB}. 
Next, since all functions $(\mathcal{Q}^{\sigma(k)})_k$ are convex on $Y$, the function 
$\mathcal{Q}^*$ is convex on $Y$ too.
Recalling Lemma \ref{optcondonediffonenotdiff},
the optimality conditions for ${\bar y}$ read
$$
\langle \nabla_{y} f({\bar y}, x_{*} ), y - {\bar y}   \rangle + \mathcal{Q}^* ( y  )   - \mathcal{Q}^* ( {\bar y} ) \geq 0,\;\forall \;y \in Y. 
$$
Since $\tilde y^{\sigma( k_0 )} \in Y$, we have in particular
$$
\langle \nabla_{y} f({\bar y} , x_{*} ), \tilde y^{\sigma( k_0 )} - {\bar y}   \rangle + \mathcal{Q}^* ( {\tilde y}^{\sigma( k_0 )} )   - \mathcal{Q}^* ( {\bar y} ) \geq 0. 
$$
However, from \eqref{finaleq1}, the left-hand side of the above inequality is $\leq -\frac{\varepsilon_0}{2}<0$ which yields the desired contradiction.\hfill 
\end{proof}

\begin{prop}\label{propvanish1dual} Let $Y \subset \mathbb{R}^n, X \subset \mathbb{R}^m$,  be two nonempty compact convex sets.
Let $f \in \mathcal{C}^1 (  Y\small{\times}X )$ be convex on $Y\small{\times}X$.
Let $(\mathcal{Q}^k)_{k \geq 1}$ be a sequence of convex $L$-Lipschitz continuous functions
on $Y$ satisfying $\underline{\mathcal{Q}}  \leq \mathcal{Q}^k \leq {\bar{\mathcal{Q}}}$ on $Y$
where $\underline{\mathcal{Q}}, {\bar{\mathcal{Q}}}$ are continuous on $Y$.
Let $g \in \mathcal{C}^1( Y \small{\times} X  )$ with components 
$g_i, i=1,\ldots,p$, 
convex on $Y\small{\times}X^{\varepsilon}$ for some $\varepsilon>0$. We also assume 
$$
(H):\;\exists \kappa >0,\;r>0,\;\mbox{ such that }\;\forall x \in X \;\exists y \in Y :\,\mathbb{B}_n( y, r ) \cap \emph{Aff}( Y) \neq \emptyset,\;A y + B x = b,\;g(y, x) < -\kappa {\textbf{e}},
$$
where ${\textbf{e}}$ is a vector of ones of size $p$.
Let $(x^k)_{k \geq 1}$ be a sequence in $X$, let $(\varepsilon^k)_{k \geq 1}$ be a sequence of nonnegative real numbers, 
and let $y^k ( \varepsilon^k )$ be an $\varepsilon^k$-optimal and feasible solution to 
\begin{equation}\label{defpblemmavanish4}
\inf \;\{f(y, x^k ) + \mathcal{Q}^k ( y) \;\;: \;\; y \in Y, \;Ay + B x^k = b, \;g(y, x^k) \leq 0 \} .
\end{equation}
Let $(\lambda^k(  \varepsilon^k ), \mu^k ( \varepsilon^k ) )$ be an $\varepsilon^k$-optimal solution to the dual
problem 
\begin{equation}\label{dualpropconv0etak}
\begin{array}{l}
\sup_{\lambda, \mu} \;h_{x^k}^k ( \lambda, \mu ) \\
\lambda = A y + B x^k -b ,\;y \in \emph{Aff}(Y),\;\mu \geq 0,
\end{array}
\end{equation}
where 
$$
h_{x^k}^k ( \lambda, \mu )=\displaystyle \inf_{y \in Y} \{f(y, x^k ) + \mathcal{Q}^k ( y) + \langle \lambda , A y + B x^k -b \rangle +  \langle \mu , g(y, x^k ) \rangle \}.
$$
Define $\eta^k (  \varepsilon^k )$ as the optimal value of the following optimization problem:
\begin{equation}\label{defpblemmavanish5}
\begin{array}{l}
\max \; \left \langle \nabla_y f(y^k ( \varepsilon^k ), x^k ) +A^T \lambda^k (  \varepsilon^k  )  +\displaystyle \sum_{i=1}^p \mu^k(\varepsilon^k )( i ) \nabla_{y} g_i( y^k ( \varepsilon^k )  , x^k ) , y^k ( \varepsilon^k )  - y \right \rangle + \mathcal{Q}^k ( y^k ( \varepsilon^k  ) ) - \mathcal{Q}^k ( y ) \\
y \in Y.
\end{array}
\end{equation}
Then if $\lim_{k \rightarrow +\infty} \varepsilon^k = 0$ we have
\begin{equation}\label{vanishnoises5}
\lim_{k \rightarrow +\infty} \eta^k (  \varepsilon^k ) = 0.
\end{equation}
\end{prop}
\begin{proof}
For simplicity, we write $\lambda^k, \mu^k, y^k $ instead of 
$\lambda^k(  \varepsilon^k ), \mu^k ( \varepsilon^k ) ), y^k ( \varepsilon^k )$, and 
put $\mathcal{Y}(x)=\{y \in Y: \;Ay + B x = b,\;g(y,x) \leq 0\}$.
Denoting by $y_*^{k} \in \mathcal{Y}( x^k ) $ an optimal solution of \eqref{defpblemmavanish4}, we get
\begin{equation}\label{defepstkredef}
f( y_*^{k}, x^k ) + \mathcal{Q}^k ( y_*^{k} ) \leq 
f( y^{k}, x^k ) + \mathcal{Q}^k ( y^{k} ) \leq f ( y_*^{k} , x^k ) + \mathcal{Q}^k ( y_*^{k} ) + \varepsilon^k.
\end{equation}
We prove \eqref{vanishnoises5} by contradiction. 
Let ${\tilde y}^k$ be an optimal solution of \eqref{defpblemmavanish5}:
$$
\eta^k ( \varepsilon^k )=\langle \nabla_{y} f( y^{k}, x^k ) + 
A^T \lambda^k + \sum_{i=1}^p \mu^k ( i) \nabla_{y} g_{i}( y^{k}, x^k ),y^{k} - {\tilde y}^k   \rangle -\mathcal{Q}^k( {\tilde y}^k ) + \mathcal{Q}^k( y^{k} ).
$$
Assume that \eqref{vanishnoises5} does not hold. Then 
there exists $\varepsilon_0 >0$ and $\sigma_1: \mathbb{N} \rightarrow \mathbb{N}$ increasing such that
for every $k \in \mathbb{N}$ we have
\begin{equation}\label{contradeta22}
\begin{array}{l}
\left \langle \nabla_{y} f({y}^{\sigma_1(k)}, x^{ \sigma_1(k) } )  +  A^T \lambda^{\sigma_1(k)} + 
\sum_{i=1}^p \mu^{\sigma_1(k)} ( i) \nabla_{y} g_{i}( y^{\sigma_1(k)}, x^{\sigma_1(k)} ),  -{\tilde y}^{ \sigma_1(k) }  + y^{\sigma_1(k)} \right  \rangle \\ 
+ \mathcal{Q}^{\sigma_1(k)}( y^{\sigma_1(k)}  ) - \mathcal{Q}^{\sigma_1(k)}( {\tilde y}^{\sigma_1(k)}  ) \geq \varepsilon_0. 
\end{array}
\end{equation}
Using Assumption (H) and Proposition \ref{propboundnormepsdualsol}, we obtain that
the sequence $(\lambda^{\sigma_1(k)}, \mu^{\sigma_1(k)})_k$ is a sequence of a compact set, say $\mathcal{D}$.
Therefore, same as in the proof of Proposition \ref{propvanish1}, we can find
$\mathcal{Q}^* \in \mathcal{C}( Y )$ and $\sigma_2: \mathbb{N} \rightarrow \mathbb{N}$ increasing 
such that setting $\sigma =\sigma_1 \circ \sigma_2$, we have
$\lim_{k \rightarrow +\infty} \|\mathcal{Q}^{\sigma(k)} - \mathcal{Q}^*  \|_{ Y } = 0$,
and $({y}^{\sigma(k)}, y_*^{\sigma(k)}, {\tilde y}^{ \sigma(k) }, x^{ \sigma(k) }, \lambda^{ \sigma(k) }, \mu^{ \sigma(k) } )$ converges
to some $({\bar y} , y_* , {\tilde y} , x_*, \lambda_*, \mu_* ) \in Y \small{\times} Y \small{\times} Y \small{\times} X \small{\times} \mathcal{D}$.
It follows that there is $k_0 \in \mathbb{N}$ such that for every $k \geq k_0$:
\begin{equation}\label{contradeta23}
\begin{array}{l}
\left| \left \langle \nabla_{y} f({y}^{\sigma(k)}, x^{ \sigma(k) } )  + A^T \lambda^{\sigma(k)} +  \sum_{i=1}^p \mu^{\sigma(k)} ( i) \nabla_{y} g_{i}( y^{\sigma(k)}, x^{\sigma(k)} ),  -{\tilde y}^{ \sigma(k) }  + y^{\sigma(k)} \right  \rangle \right. \\
\;\; \left. - \left \langle \nabla_{y} f({\bar y}, x_* )  + A^T \lambda_{*} + \sum_{i=1}^p \mu_{*} ( i) \nabla_{y} g_{i}( {\bar y},  x_{*} ), -{\tilde y}^{ \sigma(k) }  + {\bar y} \right  \rangle  \right| \leq \varepsilon_0/4, \\
\|{y}^{\sigma(k)} - {\bar y}\| \leq \frac{\varepsilon_0}{8 L},\;\|\mathcal{Q}^{\sigma(k)} - \mathcal{Q}^*  \|_{Y} \leq \varepsilon_0/16.
\end{array}
\end{equation}
Same as in the proof of Lemma \ref{propboundeddual}, we deduce from \eqref{contradeta22}, \eqref{contradeta23} that
\begin{equation}\label{caseBposcontrad}
\left \langle \nabla_{y} f({\bar y}, x_{* } )  + A^T \lambda^{*} + 
\sum_{i=1}^p \mu^{*} ( i) \nabla_{y} g_{i}( {\bar y}, x_{*} ), -{\tilde y}^{ \sigma(k_0 ) }  + {\bar y} \right  \rangle +  \mathcal{Q}^{*}( {\bar y}  ) - \mathcal{Q}^{*}( {\tilde y}^{\sigma(k_0 )}  ) \geq \varepsilon_0/2 >0.
\end{equation}
Due to Assumption (H), primal problem \eqref{defpblemmavanish4} and dual problem \eqref{dualpropconv0etak} have the same
optimal value and for every $y \in Y$ and $k \geq 1$ we have:
$$
\begin{array}{l}
 f({y}^{\sigma(k)}, x^{{\sigma(k)}} ) + \mathcal{Q}^{\sigma(k)}( y^{{\sigma(k)}}  ) + 
 \langle A y^{\sigma(k)} + B x^{\sigma(k)} - b ,   \lambda^{\sigma(k)} \rangle + \langle \mu^{\sigma(k)} ,  g( y^{\sigma(k)} , x^{\sigma(k)}) \rangle  \\
\leq  f ( y_*^{\sigma(k)} , x^{\sigma(k)} ) + \mathcal{Q}^{\sigma(k)} ( y_*^{\sigma(k)} )  + \varepsilon^{\sigma(k)} \;\;\mbox{by definition of }y_*^{\sigma(k)}, y^{\sigma(k)}\mbox{ and since }\mu^{\sigma(k)} \geq 0, y^{\sigma(k)} \in \mathcal{Y}( x^{\sigma(k)} ),\\
\leq h_{x^{\sigma(k)}}^{\sigma(k)} (\lambda^{\sigma(k)} , \mu^{\sigma(k)} ) + 2 \varepsilon^{\sigma(k)}, [(\lambda^{\sigma(k)} , \mu^{\sigma(k)})\mbox{ is an }\epsilon^{\sigma(k)}\mbox{-optimal dual solution and there is no duality gap}],\\
\leq f(y, x^{{\sigma(k)}} )  + \langle A y + B x^{\sigma(k)} - b , \lambda^{\sigma(k)} \rangle 
+ \langle \mu^{\sigma(k)} ,  g( y , x^{\sigma(k)} ) \rangle  +  \mathcal{Q}^{\sigma(k)}( y ) + 2 \varepsilon^{\sigma(k)} \mbox{ by definition of }h_{x^{\sigma(k)}}^{\sigma(k)}.
\end{array}
$$
Taking the limit in the above relation as $k \rightarrow +\infty$, we get  for every $y \in Y$:
$$
\begin{array}{l}
f({\bar y}, x_* )  + \langle A {\bar y} + B x_* - b , \lambda_{*} \rangle + \langle \mu_{*} , g( {\bar y} , x_* ) \rangle  + \mathcal{Q}^{*}( \bar y )\\
\leq  f(y, x_* )  + \langle A y + B x_* - b ,  \lambda_{*}  \rangle + \langle \mu_{*} ,  g( y , x_{*} ) \rangle + \mathcal{Q}^{*}( y ).
\end{array}
$$
Recalling that $\bar y \in Y$ this shows that $\bar y$ is an optimal solution of
\begin{equation}
\left\{
\begin{array}{l}
\min f( y, x_* ) + \mathcal{Q}^{*}( y ) + \langle A y + B x_*  - b ,  \lambda_* \rangle + \langle \mu_* , g( y, x_* ) \rangle \\
y \in Y.
\end{array}
\right.
\end{equation}
Now recall that all functions $(\mathcal{Q}^{\sigma(k)})_k$ are convex on $Y$ and therefore the function 
$\mathcal{Q}^*$ is convex on $Y$ too.
Using Lemma \ref{optcondonediffonenotdiff}, the first order optimality conditions for $\bar y$ can be written
\begin{equation}\label{caseBposcontradfinal}
\left \langle \nabla_{y} f(\bar y, x_{* } )  + A^T \lambda_{*} + \sum_{i=1}^p \mu_{*} ( i) \nabla_{y} g_{i}( \bar y , x_{*} ), y  - \bar y  \right  \rangle
+  \mathcal{Q}^{*}( y  ) - \mathcal{Q}^{*}( \bar y  ) \geq 0
\end{equation}
for all $y \in Y$. Specializing the above relation for $y = {\tilde y}^{ \sigma(k_0 ) } $, we get 
$$
\left \langle \nabla_{y} f(\bar y , x_{* } )  + A^T \lambda_{*} + \sum_{i=1}^p \mu_{*} ( i) \nabla_{y} g_{i}( \bar y , x_{*} ), {\tilde y}^{ \sigma(k_0 ) }  - \bar y  \right  \rangle
+  \mathcal{Q}^{*}( {\tilde y}^{ \sigma(k_0 ) }  ) - \mathcal{Q}^{*}( \bar y  ) \geq 0,
$$
but the left-hand side of the above inequality is $\leq \varepsilon_0/2<0$ due to \eqref{caseBposcontrad} which yields the desired contradiction.\hfill
\end{proof}

\begin{thm}[Convergence of IDDP for convex nonlinear programs]\label{convaniddp}
Consider the sequences of vectors $x_t^k$ and functions $\mathcal{Q}_ t^k$
generated by the IDDP algorithm. Let Assumptions (H1) and (H2) hold.
We have the following:
\begin{itemize}
\item[(i)] Assume that noises $(\varepsilon_t^k)_{k \geq 1}$ are bounded: for 
$t=1,\ldots,T$, we have $0 \leq \varepsilon_t^k \leq {\bar \varepsilon}_t< +\infty$.
Define ${\tt{Err}}_1= {\bar{\varepsilon}}_1$ and for $t=2,\ldots,T$, 
\begin{itemize}
\item ${\tt{Err}}_t ={\bar \eta}_{t} + {\bar \varepsilon}_{t}$ with ${\bar \eta}_{t}$ the upper bound on $\eta_{t}^k (  \varepsilon_{t}^k )$
given by \eqref{boundetafirst} if $X_t$ is of type $S1$;
\item ${\tt{Err}}_t ={\bar \eta}_{t} + 2{\bar \varepsilon}_{t}$ with ${\bar \eta}_{t}$ the upper bound on $\eta_{t}^k (  \varepsilon_{t}^k )$
given by \eqref{boundetasecond} if $X_t$ is of type $S2$.
\end{itemize}
Then there exists an infinite set of iterations $K$, such that for $t=1,\ldots,T$, the sequence $(x_t^k)_{k \in K_t}$ converges to some $x_t^* \in \mathcal{X}_t$ and
for $t=2,\ldots,T$, the sequence $(\mathcal{Q}_{t}^{k}(x_{t-1}^{k}))_{k \in K}$ converges with its limit satisfying 
$$
H_1(t): \;
\mathcal{Q}_{t}(x_{t-1}^{*})- \sum_{\tau=t}^T {\tt{Err}}_{\tau} \leq 
\displaystyle \lim_{k \rightarrow +\infty,\, k \in K} \mathcal{Q}_{t}^{k}(x_{t-1}^{k}) \leq \mathcal{Q}_{t}(x_{t-1}^{*}).
$$
Moreover, 
\begin{equation} \label{convapproxfirststep}
 \mathcal{Q}_1( x_0 ) - \displaystyle \sum_{\tau=2}^T  {\tt{Err}}_{\tau}  \leq  \lim_{k \rightarrow +\infty, k \in K} F_1^{k-1}( x_1^k ,  x_0 )
 \leq \mathcal{Q}_1( x_0 ) + {\tt{Err}}_{1},
\end{equation}
and $(x_1^*, \ldots, x_T^*)$ is an $(\sum_{\tau=1}^T \tau {\tt{Err}}_{\tau})$-optimal solution for problem \eqref{defpb}.
\item[(ii)] If for $t=1,\ldots,T$,
$\lim_{k \rightarrow +\infty} \varepsilon_t^k = 0$ then for $t=2,\ldots,T+1$, 
$$
H_2(t): \;\displaystyle \lim_{k \rightarrow +\infty} \mathcal{Q}_{t}(x_{t-1}^{k}) -  \mathcal{Q}_{t}^k (x_{t-1}^{k}) = 0,
$$
$\lim_{k \rightarrow +\infty} F_1^{k-1}( x_1^k ,  x_0 )= \mathcal{Q}_1 ( x_0)$, and any accumulation point
of the sequence $(x_1^k, \ldots,x_T^k)_{k \geq 1 }$ is an optimal solution of \eqref{defpb}.
\end{itemize}
\end{thm}
\begin{proof}
We show (i) by backward induction on $t$. 
Note that the sequence $(x_1^k, \ldots,x_T^k)_{k \geq 1}$ belongs to the compact set 
$\mathcal{X}_1 \small{\times} \ldots \small{\times} \mathcal{X}_T$ and for $t=2,\ldots,T$,  the sequence
$(\mathcal{Q}_{t}^{k}(x_{t-1}^{k}))_{k \geq 1}$ belongs to the compact intervals 
$$
[\min_{x_{t-1} \in \mathcal{X}_{t-1}} \mathcal{Q}_{t}^{1}(x_{t-1}) , \max_{x_{t-1} \in \mathcal{X}_{t-1}}, \mathcal{Q}_{t}(x_{t-1}) ].
$$
Therefore, these sequences have some convergent subsequences: there exists an infinite set of iterations $K$
such that $\lim_{k \in K} (x_1^k,\ldots,x_T^k)=(x_1^* ,\ldots,x_T^* ) \in \mathcal{X}_1 \small{\times} \ldots \small{\times} \mathcal{X}_T$ and the sequence
$(\mathcal{Q}_{t}^{k}(x_{t-1}^{k}))_{k \in K}$   has  a limit.
$H_1(T+1)$ holds by definition of  $\mathcal{Q}_{T+1}, \mathcal{Q}_{T+1}^{k}$.
Now assume that $H_1(t+1)$ holds for some $t \in \{2,\ldots,T\}$. We have for 
every $k \geq 1$:
\begin{equation} \label{eqconv1}
\mathcal{Q}_{t}(x_{t-1}^k)  \geq  \mathcal{Q}_{t}^k(x_{t-1}^k)  \geq \mathcal{C}_t^k (  x_{t-1}^k ) =  \theta_t^k - \eta_t^k (  \varepsilon_t^k   )  \mbox{ by definition of }\mathcal{Q}_t^k.
\end{equation}
Let us consider two cases to derive a lower bound on $\theta_t^k$: $X_t$ is of type $S1$ (Case A) and $X_t$ is of type $S2$ (Case B).

\par {\textbf{Case A.}} We have for all $k \geq 1$:
\begin{equation}\label{caseAeqconv1}
\theta_t^k = F_t^k ( x_t^{B k}, x_{t-1}^k ) \geq {\underline{\mathcal{Q}}}_t^k ( x_{t-1}^{k} ) \mbox{ using }\eqref{epssolprimal}.
\end{equation}
\par {\textbf{Case B.}} Using relations \eqref{dualiddp}, \eqref{defepssoldual} and the fact that
$x_t^{B k} \in \mathcal{X}_t$ we get for all $k \geq 1$:
\begin{equation}\label{caseBeqconv1}
\begin{array}{lll}
\theta_t^k & = &  L_{x_{t-1}^k}(x_t^{B  k}, y_t^{B  k}, \lambda_t^k, \mu_t^k ),\\
& \geq  &  h_{t, x_{t-1}^k}^k ( \lambda_t^k , \mu_t^k ) \geq   {\underline{\mathcal{Q}}}_t^k ( x_{t-1}^{k} ) - \varepsilon_t^k.
\end{array}
\end{equation}
We now need a  lower bound on ${\underline{\mathcal{Q}}}_t^k ( x_{t-1}^{k} )$ for all $k \geq 1$:
\begin{equation} \label{eqconv1bis0}
\begin{array}{lcl}
{\underline{\mathcal{Q}}}_t^k ( x_{t-1}^{k} ) & \geq & {\underline{\mathcal{Q}}}_t^{k-1} ( x_{t-1}^k )   \mbox{ by monotonicity,}\\
& \geq  &  F_t^{k-1}(x_t^{k}, x_{t-1}^k) -  \varepsilon_t^k    \mbox{ using } \eqref{epssolforward},\\
& = &  f_t(x_t^{k}, x_{t-1}^k) + \mathcal{Q}_{t+1}^{k-1}(x_t^{k}) -  \varepsilon_t^k  \mbox{ by definition of }F_t^{k-1},\\
& = &  F_t(x_t^{k}, x_{t-1}^k) + \mathcal{Q}_{t+1}^{k-1}(x_t^{k}) - \mathcal{Q}_{t+1}( x_t^{k})    -  \varepsilon_t^k   \mbox{ by definition of }F_t,\\
& \geq &  \mathcal{Q}_t(x_{t-1}^k)+ \mathcal{Q}_{t+1}^{k-1}(x_t^{k}) - \mathcal{Q}_{t+1}( x_t^{k})    -  \varepsilon_t^k  , \mbox{ by definition of }\mathcal{Q}_t.
\end{array}
\end{equation}
Combining \eqref{eqconv1}, \eqref{caseAeqconv1}, \eqref{caseBeqconv1}, and  \eqref{eqconv1bis0}, yields for all $k \geq 1$:
\begin{equation} \label{eqconv1bisfuture}
\left\{
\begin{array}{l}
\mathcal{Q}_{t}(x_{t-1}^k)  \geq  \mathcal{Q}_{t}^k(x_{t-1}^k)  \geq \mathcal{Q}_t(x_{t-1}^k)+ \mathcal{Q}_{t+1}^{k-1}(x_t^{k}) - \mathcal{Q}_{t+1}( x_t^{k}) - \varepsilon_t^k -  \eta_t^k (  \varepsilon_t^k   )\mbox{ if }X_t\mbox{ is of }\mbox{ type }S1\\
\mathcal{Q}_{t}(x_{t-1}^k)  \geq  \mathcal{Q}_{t}^k(x_{t-1}^k)  \geq \mathcal{Q}_t(x_{t-1}^k)+ \mathcal{Q}_{t+1}^{k-1}(x_t^{k}) - \mathcal{Q}_{t+1}( x_t^{k}) - 2\varepsilon_t^k -  \eta_t^k (  \varepsilon_t^k   )\mbox{ if }X_t\mbox{ is of }\mbox{ type } S2, 
\end{array}
\right.
\end{equation}
which implies, using the definition of ${\tt{Err}}_t$, that for all $k \geq 1$,
\begin{equation} \label{eqconv1bis}
\mathcal{Q}_{t}(x_{t-1}^k)  \geq  \mathcal{Q}_{t}^k(x_{t-1}^k)  \geq \mathcal{Q}_t(x_{t-1}^k)+ \mathcal{Q}_{t+1}^{k-1}(x_t^{k}) - \mathcal{Q}_{t+1}( x_t^{k}) - {\tt{Err}}_t.
\end{equation}
Take now an arbitrary $\delta>0$. Using the induction hypothesis, we can find $k_0 \in K$ such that
for all $k \in K$ with $k \geq k_0$ we have
\begin{equation}\label{convifirststep}
\mathcal{Q}_{t+1}^{k}(x_{t}^{k}) -  \mathcal{Q}_{t+1}(x_{t}^{k}) \geq -\frac{\delta}{3} -  \sum_{\tau=t+1}^T {\tt{Err}}_{\tau},
\;\|x_t^{k}  - x_t^{k_0}\| \leq \frac{\delta}{6 L_{t+1}}.
\end{equation}
Also, since the sequence $(\mathcal{Q}_{t+1}^k( x_t^{k_0}  ))_{k \geq k_0}$ is increasing and bounded from above
by $\mathcal{Q}_{t+1}( x_t^{k_0}  )<+\infty$, it converges. Therefore 
$\lim_{k \rightarrow + \infty, k \in K} \mathcal{Q}_{t+1}^{k-1}(x_t^{k_0}) - \mathcal{Q}_{t+1}^{k}(x_t^{k_0})=0 $ and $k_0$
can be chosen sufficiently large in such a way that for $k \in K$ with $k \geq k_0$ both \eqref{convifirststep} and
\begin{equation}\label{finalfirsticonviddp}
\mathcal{Q}_{t+1}^{k-1}(x_t^{k_0}) - \mathcal{Q}_{t+1}^{k}(x_t^{k_0}) \geq -\frac{\delta}{3}
\end{equation}
hold.
Then for all $k \in K$ with $k \geq k_0$ we get  
$$
\begin{array}{lcl}
\mathcal{Q}_{t}(x_{t-1}^k)  &\geq  &   \mathcal{Q}_{t}^k ( x_{t-1}^k ),\\ 
& \stackrel{\eqref{eqconv1bis}}{\geq}  & \mathcal{Q}_t(x_{t-1}^k)+ \mathcal{Q}_{t+1}^{k-1}(x_t^{k}) - \mathcal{Q}_{t+1}( x_t^{k})-{\tt{Err}}_t,\\
& = & \mathcal{Q}_t(x_{t-1}^k)+ \mathcal{Q}_{t+1}^{k-1}(x_t^{k}) - \mathcal{Q}_{t+1}^{k}(x_t^{k})  +   \mathcal{Q}_{t+1}^{k}(x_t^{k}) -   \mathcal{Q}_{t+1}( x_t^{k})-{\tt{Err}}_t, \\
&  \stackrel{\eqref{convifirststep}}{\geq} & \mathcal{Q}_t(x_{t-1}^k)  -\frac{\delta}{3} -  \displaystyle \sum_{\tau=t}^T {\tt{Err}}_{\tau} + \mathcal{Q}_{t+1}^{k-1}(x_t^{k_0}) - \mathcal{Q}_{t+1}^{k}(x_t^{k_0}) - 2 L_{t+1}\|x_t^k  - x_t^{k_0}\|,\\
& \stackrel{\eqref{convifirststep}, \eqref{finalfirsticonviddp}}{\geq } &  \mathcal{Q}_t(x_{t-1}^k) - \delta -  \displaystyle \sum_{\tau=t}^T {\tt{Err}}_{\tau}.
\end{array}
$$
Taking the limit in the above inequality as $k \in K, k \rightarrow +\infty$, using the continuity of
$\mathcal{Q}_t$,  and then taking the limit as $\delta \rightarrow 0$ we obtain $H_2(t)$. This achieves the induction step and therefore
$H_2(2),\ldots,H_2(T+1)$ hold.

Using \eqref{eqconv1bis0} for $t=1$, we get for all $k \geq 1$,
\begin{equation} \label{eqconv1bisfuturet1}
\mathcal{Q}_{1}(x_{0}) +  {\tt{Err}}_{1}  \geq  {\underline{\mathcal{Q}}}_{1}^k(x_{0}) +  {\tt{Err}}_{1}  \geq  F_1^{k-1}( x_1^k , x_0 )   \geq  \mathcal{Q}_1 (x_{0} ) + \mathcal{Q}_{2}^{k-1}(x_1^{k}) - \mathcal{Q}_{2}( x_1^{k}),
\end{equation}
both when $X_t$ is of type $S1$ and when $X_t$ is of type $S2$. Repeating the computations of the induction step which have shown
that for $t \in \{2,\ldots,T\}$ the sequence $(\mathcal{Q}_{t+1}^{k-1}(x_t^{k}) - \mathcal{Q}_{t+1}( x_t^{k}))_{k \in K}$ has 
a limit $\geq -\sum_{\tau=t+1}^T {\tt{Err}}_{\tau}$ when $k \rightarrow +\infty$, we obtain that the sequence $(\mathcal{Q}_{2}^{k-1}(x_1^{k}) - \mathcal{Q}_{2}( x_1^{k}))_{k \in K}$ has a limit when $k \rightarrow +\infty$
which is 
$\geq -\sum_{\tau=2}^T {\tt{Err}}_{\tau}$. 
Now observe that $F_1^{k-1} ( x_1^{k} , x_0 )=f_1( x_1^k , x_0 ) + \mathcal{Q}_2^{k-1} ( x_1^k )$.
Since the sequences $(\mathcal{Q}_{2}^{k-1}(x_1^{k}) - \mathcal{Q}_{2}( x_1^{k}))_{k \in K}$ and $\mathcal{Q}_2 ( x_1^k )$ converge when
$k \rightarrow +\infty$, the sequences $( \mathcal{Q}_2^{k-1} ( x_1^k ) )_{k \in K}$ and $( F_1^{k-1} ( x_1^{k} , x_0 ) )_{k \in K}$ also converge when $k \rightarrow +\infty$.
Therefore passing to the limit in  \eqref{eqconv1bisfuturet1} when $k\rightarrow +\infty, k\in K$, we get \eqref{convapproxfirststep}.

Relations \eqref{eqconv1}, \eqref{caseAeqconv1}, \eqref{caseBeqconv1}, \eqref{eqconv1bis0}, and \eqref{eqconv1bisfuturet1} also imply that for $t=1,\ldots,T$, and $k \geq 1$:
\begin{equation}\label{relationfinali}
\mathcal{Q}_t ( x_{t-1}^k ) + {\tt{Err}}_t \geq F_t^{k-1}( x_t^k, x_{t-1}^k ) = f_t( x_t^k , x_{t-1}^k ) + \mathcal{Q}_{t+1}^{k-1}(  x_t^k ).  
\end{equation}
For $t=1,\ldots,T$, we have that 
$\lim_{k \rightarrow +\infty, k \in K} \mathcal{Q}_{t+1}^{k-1}(  x_t^k ) - \mathcal{Q}_{t+1}^{k}(  x_t^k )=0$
and the sequence $(\mathcal{Q}_{t+1}^{k}(  x_t^k ))_{k \in K}$ has a limit when $k \rightarrow +\infty$ which is $\geq \mathcal{Q}_{t+1}( x_t^* ) -\sum_{\tau=t+1}^T {\tt{Err}}_{\tau}$.
It follows that the sequence $(\mathcal{Q}_{t+1}^{k-1}(  x_t^k ))_{k \in K}$ also has a limit when $k \rightarrow +\infty$ which is $\geq \mathcal{Q}_{t+1}( x_t^* ) -\sum_{\tau=t+1}^T {\tt{Err}}_{\tau}$.
Passing to the limit in \eqref{relationfinali} when $k \rightarrow +\infty, k \in K$, and using the continuity of $f_t$ we obtain 
$$
\mathcal{Q}_t ( x_{t-1}^* ) - \mathcal{Q}_{t+1} ( x_{t}^* ) + \sum_{\tau=t}^T {\tt{Err}}_{\tau} \geq f_t ( x_t^* , x_{t-1}^* )
$$
for $t=1,\ldots,T$ where $x_0^* = x_0$, and summing these inequalities we get 
$$
\mathcal{Q}_1( x_0 ) + \sum_{t=1}^T \sum_{\tau=t}^T {\tt{Err}}_{\tau} =  \mathcal{Q}_1( x_0 ) - \mathcal{Q}_{T+1}( x_T^* ) + \sum_{t=1}^T \sum_{\tau=t}^T {\tt{Err}}_{\tau} \geq \sum_{t=1}^T f_t( x_t^* , x_{t-1}^* ).
$$
If $X_t$ is of type $S1$ then since $\mathcal{X}_t$ is closed we have $x_t^* \in \mathcal{X}_t$ while
if $X_t$ is of type $S2$ then since $\mathcal{X}_t$ is closed and $g_t$ is differentiable (and therefore lower semicontinuous with closed level sets) 
we have that $x_t^* \in X_t( x_{t-1}^ *)$. This shows that $x^* :=(x_1^*, \ldots, x_T^* )$ is feasible for problem \eqref{defpb} and the relation above proves that the value
$\sum_{t=1}^T f_t( x_t^* , x_{t-1}^* )$ of the objective function at that point is at most the optimal value $\mathcal{Q}_1( x_0 )$ of the problem
plus $\sum_{\tau=1}^T \tau {\tt{Err}}_{\tau}$. This completes the proof of (i).
\par Let us now show (ii). First observe that \eqref{eqconv1bisfuture} still holds. We now show that 
\begin{equation}\label{proofnoisesvanish}
\lim_{k \rightarrow +\infty} \eta_t^k (  \varepsilon_t^k   ) =0.
\end{equation}
We consider two cases: $X_t$ is of type $S1$ (Case A) and $X_t$ is of type $S2$ (Case B).\\

\par {\textbf{Case A.}}
We have that $x_t^{B k}$ is an $\varepsilon_t^k$-optimal solution of the problem 
\begin{equation}\label{pbbackreformulationproof1}
\left\{
\begin{array}{l}
\min f_t( x_t, x_{t-1}^k ) + \mathcal{Q}_{t+1}^k ( x_t ) \\
x_t \in \mathcal{X}_t.
\end{array}
\right.
\end{equation}
Observe that $\eta_t^k (  \varepsilon_t^k   )$ can be written in the form (see \eqref{defetaiddp0} and the definition of $B_{t+1}^k$):
\begin{equation}\label{defetaiddp2}
\eta_t^k (  \varepsilon_t^k   )=
\left\{
\begin{array}{l}
\max_{x_t} \;\langle \nabla_{x_t} f_{t}({x}_t^{B  k}, x_{t-1}^k ), x_t^{B  k} - x_t  \rangle + \mathcal{Q}_{t+1}^{k}(  x_t^{B k}  )   -  \mathcal{Q}_{t+1}^{k}(  x_t  )  \\
x_t \in \mathcal{X}_t.
\end{array}
\right.
\end{equation}
We now apply Proposition \ref{propvanish1} to problems \eqref{pbbackreformulationproof1}, \eqref{defetaiddp2} setting:
\begin{itemize}
\item $Y=\mathcal{X}_t, X=\mathcal{X}_{t-1}$ which are nonempty, compact, and convex;
\item $f=f_t \in \mathcal{C}^1( Y \small{\times} X)$ convex on $Y \small{\times} X$;
\item $\mathcal{Q}^k=\mathcal{Q}_{t+1}^k$ which is convex Lipschitz continuous on $Y$ with Lipschitz constant
$L_{t+1}=M_{t+1 2}$ (see the proof of Proposition  \ref{propboundeddual}) and satisfy
$$
{\underline{Q}} := \mathcal{Q}_{t+1}^1 \leq \mathcal{Q}^k \leq {\bar{\mathcal{Q}}}:=\mathcal{Q}_{t+1}
$$
on $Y$ with ${\underline{Q}}, {\bar{\mathcal{Q}}}$ continuous on $Y$;
\item $(x^k) = (x_{t-1}^k)$ sequence in $X$ and $(y^k) = (x_t^{B k})$ sequence in $Y$.
\end{itemize}
Therefore we can apply  Proposition \ref{propvanish1} to obtain \eqref{proofnoisesvanish}.\\

\par {\textbf{Case B.}}
Now $x_t^{B k}$ is an $\varepsilon_t^k$-optimal solution of the problem 
\begin{equation}\label{pbbackreformulationproof2}
\left\{
\begin{array}{l}
\min f_t( x_t, x_{t-1}^k ) + \mathcal{Q}_{t+1}^k ( x_t ) \\
x_t \in X_t( x_{t-1}^k ).
\end{array}
\right.
\end{equation}
Observe that $\eta_t^k (  \varepsilon_t^k   )$ can be written as the optimal value of the following optimization problem (see \eqref{defetaiddp} and the definition of $B_{t+1}^k$):
\begin{equation}\label{defetaiddp2bis}
\begin{array}{l}
\max_{x_t} \;\langle \nabla_{x_t} f_t( x_t^{B k}, x_{t-1}^k ) + A_t^T \lambda_t^k + \sum_{i=1}^m \mu_t^k ( i) \nabla g_{t i}( x_t^{B k}, x_{t-1}^k ),x_t^{B k} - x_t   \rangle -\mathcal{Q}_{t+1}^k( x_t ) + \mathcal{Q}_{t+1}^k( x_t^{B k} )    \\
x_t \in \mathcal{X}_t.
\end{array}
\end{equation}
We now apply Proposition \ref{propvanish1dual} to primal  problem \eqref{pbbackreformulationproof2}, dual problem \eqref{dualproblemS2} and  problem \eqref{defetaiddp2bis} setting:
\begin{itemize}
\item $Y=\mathcal{X}_t, X=\mathcal{X}_{t-1}$ which are nonempty compact and convex;
\item $f=f_t \in \mathcal{C}^1( Y \small{\times} X)$ convex on $Y \small{\times} X$;
\item $g=g_t \in \mathcal{C}^1( Y \small{\times} X)$ with components $g_i,i=1,\ldots,p$, convex
on $Y \small{\times}X^{\varepsilon}$;
\item $\mathcal{Q}^k=\mathcal{Q}_{t+1}^k$ which is convex Lipschitz continuous on $Y$ with Lipschitz constant
$L_{t+1}$ obtained by replacing $t$ by $t+1$ in \eqref{formulaLt} (given in the proof of
Proposition  \ref{propboundeddual}) and satisfy
$$
{\underline{Q}} := \mathcal{Q}_{t+1}^1 \leq \mathcal{Q}^k \leq {\bar{\mathcal{Q}}}:=\mathcal{Q}_{t+1}
$$
on $Y$ with ${\underline{Q}}, {\bar{\mathcal{Q}}}$ continuous on $Y$;
\item $(x^k ) = (x_{t-1}^k)_k$ sequence in $X$, $\lambda^k = \lambda_t^k$, $\mu^k = \mu_t^k$, and $(y^k) = (x_t^{B k})$ sequence in $Y$. 
\end{itemize}
With this notation Assumption (H) is satisfied, since Assumption (H2) holds.
It follows that we can apply Proposition \ref{propvanish1dual} to obtain \eqref{propvanish1dual}.\\

Therefore \eqref{proofnoisesvanish} holds both when $X_t$ is of type $S1$ and of type $S2$.\\

Next, recall that $\mathcal{Q}_{t+1}$ is convex, functions $(\mathcal{Q}_{t+1}^k)_k$ are $L_{t+1}$-Lipschitz,
and for all $k \geq 1$ we have $\mathcal{Q}_{t+1}^k \leq \mathcal{Q}_{t+1}^{k+1} \leq\mathcal{Q}_{t+1}$ on compact set $\mathcal{X}_t$.
Therefore, the induction hypothesis
$$
\lim_{k \rightarrow +\infty} \mathcal{Q}_{t+1}( x_t^k ) - \mathcal{Q}_{t+1}^k ( x_t^k )=0 
$$
implies, using Lemma A.1 in \cite{lecphilgirar12}, that
\begin{equation}\label{indchypreformulated}
\lim_{k \rightarrow +\infty} \mathcal{Q}_{t+1}( x_t^k ) - \mathcal{Q}_{t+1}^{k-1} ( x_t^k )=0 . 
\end{equation}
Plugging \eqref{proofnoisesvanish} and \eqref{indchypreformulated} into
\eqref{eqconv1bisfuture}, we get 
$$
\lim_{k \rightarrow +\infty} \mathcal{Q}_{t}( x_{t-1}^k ) - \mathcal{Q}_{t}^k ( x_{t-1}^k )=0,
$$
which shows $H_2(t)$.

Next we write \eqref{eqconv1bis0} for $t=1$, implying for all $k \geq 1$:
\begin{equation}\label{finalF1ii}
\mathcal{Q}_1( x_0 ) \geq {\underline{\mathcal{Q}}}_1^k ( x_{0} )  
\geq F_1^{k-1}( x_1^k, x_0 ) -\varepsilon_1^k \geq \mathcal{Q}_1( x_0 ) + \mathcal{Q}_2^{k-1}( x_1^k )-  \mathcal{Q}_2 ( x_1^k ) -\varepsilon_1^k. 
\end{equation}
From $H_2(t)$ we have $\lim_{k \rightarrow +\infty} \mathcal{Q}_2^{k}( x_1^k ) - \mathcal{Q}_2( x_1^k )=0$.
Applying  once again Lemma A.1 in \cite{lecphilgirar12} (to the sequence of functions 
$(\mathcal{Q}_2^k)$ which are $L_2$-Lipschitz and satisfy $\mathcal{Q}_{2}^{k} \leq \mathcal{Q}_2^{k+1} \leq \mathcal{Q}_2$ on $\mathcal{X}_1$)
we deduce that $\lim_{k \rightarrow +\infty} \mathcal{Q}_2^{k-1}( x_1^k ) - \mathcal{Q}_2( x_1^k )=0$, which, plugged into
\eqref{finalF1ii}, gives $\lim_{k \rightarrow +\infty} F_1^{k-1}( x_1^k, x_0 )  = \mathcal{Q}_1( x_0 ) $.

Finally, consider an accumulation $(x_1^*,\ldots,x_T^*)$ of the sequence $(x_1^k, \ldots, x_T^k)_{k \geq 1}$.
Let $K$ be such that $\lim_{k \rightarrow +\infty, k\in K} (x_1^k, \ldots, x_T^k)=(x_1^*,\ldots,x_T^*)$.
Relation \eqref{eqconv1bis0} gives for all $k \geq 1$ and $t=1,\ldots,T$:
$$
\mathcal{Q}_t ( x_{t-1}^k ) \geq {\underline{\mathcal{Q}}}_t^k ( x_{t-1}^{k} )  
\geq F_t^{k-1}( x_t^k, x_{t-1}^k ) -\varepsilon_t^k = f_t(x_t^k , x_{t-1}^k) + \mathcal{Q}_{t+1}^{k-1}( x_t^k ) -\varepsilon_t^k . 
$$
Combining this relation with 
$$
\lim_{k \rightarrow +\infty, k\in K} \mathcal{Q}_{t+1}^{k-1}( x_t^k ) = \lim_{k \rightarrow +\infty, k \in K} \mathcal{Q}_{t+1}^{k}( x_t^k ) =
\lim_{k \rightarrow +\infty, k \in K} \mathcal{Q}_{t+1}( x_t^k ) =  \mathcal{Q}_{t+1}( x_t^* ),\,t=1,\ldots,T,
$$
(we have used the continuity of $\mathcal{Q}_{t+1}$), we get 
$$
\mathcal{Q}_t(x_{t-1}^*) \geq f_t ( x_t^* , x_{t-1}^* ) + \mathcal{Q}_{t+1}(x_{t}^*), \mbox{for all }t=1,\ldots,T.
$$
Summing these inequalities we obtain that the optimal value $\mathcal{Q}_1( x_0 )$ of \eqref{defpb} satisfies:
\begin{equation}\label{optvaluefinal}
\mathcal{Q}_1( x_0 ) \geq  \sum_{t=1}^T f_t ( x_t^* , x_{t-1}^* ).
\end{equation}
As in (i), if $X_t$ is of type $S1$ then since $\mathcal{X}_t$ is closed we have $x_t^* \in \mathcal{X}_t$ while
if $X_t$ is of type $S2$ then since $\mathcal{X}_t$ is closed and $g_t$ lower semicontinuous 
we have that $x_t^* \in X_t( x_{t-1}^ *)$. This shows that $(x_1^*, \ldots, x_T^* )$ is feasible
for \eqref{defpb}
and the value $\sum_{t=1}^T f_t ( x_t^* , x_{t-1}^* )$ of the objective function
at this point is at most the optimal value $\mathcal{Q}_1( x_0 )$ of the problem. 
Therefore, \eqref{optvaluefinal} is an equality and
$(x_1^*, \ldots, x_T^* )$ is an optimal solution to \eqref{defpb}.\hfill
\end{proof}
\begin{cor}[Approximate solution computed by IDDP for bounded noises] \label{corconviddp}
For $t=2,\ldots,T$, let us set $\bar \delta_t = {\bar \eta}_t$ if $X_t$
is of type S1 and $\bar \delta_t = {\bar \eta}_t + {\bar{\varepsilon}}_t$
if $X_t$ is of type S2. Therefore, setting $\bar \delta_1=0$, the error term ${\tt{Err}}_t$ given in Theorem \ref{convaniddp}
 can be written as ${\tt{Err}}_t= \bar \delta_t + {\bar{\varepsilon}}_t$ for $t=1,\ldots, T$, where 
 $\bar \varepsilon_t$ (resp. ${\bar{\delta}}_t$) is an error term coming from the fact that approximate optimal values (resp. approximate
 subgradients) for the value functions are computed. 
Recall that for $t=2,\ldots,T$,
if $X_t$ is of type S1 (resp. S2) then using
Proposition \ref{fixedprop1} (resp. Proposition \ref{varprop1})
the distance between the value ${\underline{\mathcal{Q}}}_t^k ( x_{t-1}^k )$
of ${\underline{\mathcal{Q}}}_t^k$ at $x_{t-1}^k$ and the value
$\mathcal{C}_t^k ( x_{t-1}^k )$
of cut $\mathcal{C}_t^k$ at $x_{t-1}^k$
is at most
$\eta_t^k (\varepsilon_t^k ) \leq \bar \delta_t$ (resp. $\eta_t^k (\varepsilon_t^k ) + \varepsilon_t^k  \leq \bar \delta_t$).

 We deduce a nice interpretation of (i) in Theorem \ref{convaniddp}: any accumulation point of the sequence
 $(x_1^k,\ldots,x_T^k)$ is a $\frac{T(T+1)}{2}({\bar \delta} + {\bar {\varepsilon}})$-optimal solution
 of \eqref{defpb} where $\displaystyle {\bar {\varepsilon}}=\max_{t=1,\ldots,T} {\bar \varepsilon}_t$ is an upper bound on noises $\varepsilon_t^k$ and
 ${\bar \delta}= \displaystyle \max_{t=2,\ldots,T} {\bar \delta}_t$ is an upper bound on the distance between the value of the
 (theoretical) exact cuts and the value of our inexact cuts at the trial points $x_{t-1}^k$. 
\end{cor}

\section{Inexact Stochastic Dual Dynamic Programming (ISDDP)}\label{sec:isddp}

In this section we introduce ISDDP, an inexact variant of SDDP 
which combines the tools developed in Sections \ref{sec:computeinexactcuts} and \ref{sec:boundingmulti} with SDDP.

\subsection{Problem formulation and assumptions}\label{sec:sddp1}

ISDDP applies to multistage stochastic nonlinear optimization problems of the form
\begin{equation}\label{pbtosolve}
\begin{array}{l} 
\displaystyle{\inf_{x_1,\ldots,x_T}} \; \mathbb{E}_{\xi_2,\ldots,\xi_T}[ \displaystyle{\sum_{t=1}^{T}}\;f_t(x_t(\xi_1,\xi_2,\ldots,\xi_t ), x_{t-1}(\xi_1,\xi_2,\ldots,\xi_{t-1}), \xi_t )]\\
x_t(\xi_1,\xi_2,\ldots,\xi_t) \in X_t( x_{t-1}(\xi_1,\xi_2,\ldots,\xi_{t-1}), \xi_t )\;\mbox{a.s.}, \;x_{t} \;\mathcal{F}_t\mbox{-measurable, }t=1,\ldots,T,
\end{array}
\end{equation}
where $x_0$ is given,  $(\xi_t)_{t=2}^T$ is a stochastic process, $\mathcal{F}_t$ is the sigma-algebra
$\mathcal{F}_t:=\sigma(\xi_j, j\leq t)$, and $X_t(x_{t-1}, \xi_t ),t=1,\ldots,T$, can be of two types:
\begin{itemize}
\item[(S1)] $X_t( x_{t-1}, \xi_t ) = \mathcal{X}_t \subset \mathbb{R}^n$ (in this case, for short, we say that $X_t$ is of type S1); 
\item[(S2)] $X_t( x_{t-1} , \xi_t)= \{x_t \in \mathbb{R}^n : x_t \in \mathcal{X}_t,\;g_t(x_t, x_{t-1}, \xi_t) \leq 0,\;\;\displaystyle A_{t} x_{t} + B_{t} x_{t-1} = b_t \}$.
In this case, for short, we say that $X_t$ is of type S2 and $\xi_t$ contains in particular the random elements in matrices $A_t, B_t$, and vector $b_t$.
\end{itemize}
Same as problem class \eqref{defpb}, a mix of these types of constraints is allowed: for instance we can have $X_1$ of type S1 and $X_2$ of type $S2$.\\

We make the following assumption on $(\xi_t)$:\\
\par (Sto-H0) $(\xi_t)$ 
is interstage independent and
for $t=2,\ldots,T$, $\xi_t$ is a random vector taking values in $\mathbb{R}^K$ with a discrete distribution and
a finite support $\Theta_t=\{\xi_{t 1}, \ldots, \xi_{t M}\}$ while $\xi_1$ is deterministic.\footnote{To simplify notation and without loss of generality, we have assumed that the number $M$ of possible realizations
of $\xi_t$, the size $K$ of $\xi_t$, and $n$ of $x_t$ do not depend on $t$.}\\

We will denote by $A_{t j}, B_{t j},$ and $b_{t j}$ the realizations of respectively $A_t, B_t,$ and $b_t$
in $\xi_{t j}$. For this problem, we can write Dynamic Programming equations: assuming that $\xi_1$ is deterministic,
the first stage problem is 
\begin{equation}\label{firststodp}
\mathcal{Q}_1( x_0 ) = \left\{
\begin{array}{l}
\inf_{x_1 \in \mathbb{R}^n} F_1(x_1, x_0, \xi_1) := f_1(x_1, x_0, \xi_1)  + \mathcal{Q}_2 ( x_1 )\\
x_1 \in X_1( x_{0}, \xi_1 )\\
\end{array}
\right.
\end{equation}
for $x_0$ given and for $t=2,\ldots,T$, $\mathcal{Q}_t( x_{t-1} )= \mathbb{E}_{\xi_t}[ \mathfrak{Q}_t ( x_{t-1},  \xi_{t}  )  ]$ with
\begin{equation}\label{secondstodp} 
\mathfrak{Q}_t ( x_{t-1}, \xi_{t}  ) = 
\left\{ 
\begin{array}{l}
\inf_{x_t \in \mathbb{R}^n}  F_t(x_t, x_{t-1}, \xi_t ) :=  f_t ( x_t , x_{t-1}, \xi_t ) + \mathcal{Q}_{t+1} ( x_t )\\
x_t \in X_t ( x_{t-1}, \xi_t ),
\end{array}
\right.
\end{equation}
with the convention that $\mathcal{Q}_{T+1}$ is null.

We set $\mathcal{X}_0=\{x_0\}$ and make the following assumptions (Sto-H1) on the problem data: there exists $\varepsilon>0$ such that for $t=1,\ldots,T$,\\
\par (Sto-H1)-(a) $\mathcal{X}_t$ is nonempty, convex, and compact.
\par (Sto-H1)-(b) For every $x_{t}, x_{t-1} \in \mathbb{R}^{n}$ the function
$f_t(x_{t}, x_{t-1}, \cdot)$ is measurable and for every $j=1,\ldots,M$, the function
$f_t(\cdot, \cdot,\xi_{t j})$ is convex on  $\mathcal{X}_t \small{\times} \mathcal{X}_{t-1}$
and belongs to  $\mathcal{C}^{1}(\mathcal{X}_t \small{\times} \mathcal{X}_{t-1})$.\\

For $t=1,\ldots,T$, if $X_t$ is of type $S2$ we additionally assume that there exists $\varepsilon_t>0$ such that
(without loss of generality, we will assume in the sequel that $\varepsilon_t=\varepsilon$):\\

\par (Sto-H1)-(c) for every $j=1,\ldots,M$, each component $g_{t i}(\cdot, \cdot, \xi_{t j}), i=1,\ldots,p$, of the function $g_t(\cdot, \cdot, \xi_{t j})$ is 
convex on $\mathcal{X}_t \small{\times} \mathcal{X}_{t-1}^{\varepsilon_t}$
and belongs to $\mathcal{C}^{1}( \mathcal{X}_t \small{\times} \mathcal{X}_{t-1} )$.
\par (Sto-H1)-(d) For every $j=1,\ldots,M$, for every
$x_{t-1} \in \mathcal{X}_{t-1}^{\varepsilon_t}$,
the set $X_t(x_{t-1}, \xi_{t j}) \cap \mbox{ri}( \mathcal{X}_t)$ is nonempty.
\par  (Sto-H1)-(e) If $t \geq 2$, for every $j=1,\ldots,M$, there exists
${\bar x}_{t j}=({\bar x}_{t j t}, {\bar x}_{t j t-1}  ) \in  \mbox{ri}(\mathcal{X}_t) \small{\times} \mathcal{X}_{t-1}$
such that $g_t(\bar x_{t j t}, \bar x_{t j t-1}, \xi_{t j}) < 0$ and  
$A_{t j} \bar x_{t j t} + B_{t j} \bar x_{t j t-1} = b_{t j}$.\\

These assumptions are natural extensions of Assumptions (H1) to the stochastic case.
Due to Assumption (Sto-H0), the $M^{T-1}$ realizations of $(\xi_t)_{t=1}^T$ form a scenario tree of depth $T+1$
where the root node $n_0$ associated to a stage $0$ (with decision $x_0$ taken at that
node) has one child node $n_1$
associated to the first stage (with $\xi_1$ deterministic).

We denote by $\mathcal{N}$ the set of nodes, by
{\tt{Nodes}}$(t)$ the set of nodes for stage $t$ and
for a node $n$ of the tree, we define: 
\begin{itemize}
\item $C(n)$: the set of children nodes (the empty set for the leaves);
\item $x_n$: a decision taken at that node;
\item $p_n$: the transition probability from the parent node of $n$ to $n$;
\item $\xi_n$: the realization of process $(\xi_t)$ at node $n$\footnote{The same notation $\xi_{\tt{Index}}$ is used to denote
the realization of the process at node {\tt{Index}} of the scenario tree and the value of the process $(\xi_t)$
for stage {\tt{Index}}. The context will allow us to know which concept is being referred to.
In particular, letters $n$ and $m$ will only be used to refer to nodes while $t$ will be used to refer to stages.}:
for a node $n$ of stage $t$, this realization $\xi_n$ contains in particular the realizations
$b_n$ of $b_t$, $A_{n}$ of $A_{t}$, and $B_{n}$ of $B_{t}$.
\item $\xi_{[n]}$: the history of the realizations of process $(\xi_t)$ from the first stage node $n_1$ to node $n$:
 for a node $n$ of stage $t$, the $i$-th component of $\xi_{[n]}$ is $\xi_{\mathcal{P}^{t-i}(n)}$ for $i=1,\ldots, t$,
 where $\mathcal{P}:\mathcal{N} \rightarrow \mathcal{N}$ is the function 
 associating to a node its parent node (the empty set for the root node).
\end{itemize}

\subsection{ISDDP algorithm} \label{sec:isddpalgo}

Similarly to SDDP, at iteration $k$ of the ISDDP algorithm,  trial points $x_n^k$ are computed in a forward pass
for all nodes $n$ of the scenario tree replacing recourse functions 
$\mathcal{Q}_{t+1}$ by the approximations $\mathcal{Q}_{t+1}^{k-1}$ available at the beginning of this iteration.

In a backward pass, we then select a set of nodes $(n_1^k, n_2^k, \ldots, n_T^k)$ 
(with $n_1^k=n_1$, and for $t \geq 2$, $n_t^k$ a node of stage $t$, child of node $n_{t-1}^k$) 
corresponding to a sample $({\tilde \xi}_1^k, {\tilde \xi}_2^k,\ldots, {\tilde \xi}_T^k)$
of $(\xi_1, \xi_2,\ldots, \xi_T)$. For $t=2,\ldots,T$, a cut  
\begin{equation}\label{eqcutctk}
\mathcal{C}_t^k( x_{t-1} ) = \theta_t^{k} - \eta_t^k ( \varepsilon_t^k ) +  \langle \beta_t^{k}, x_{t-1}-x_{n_{t-1}^{k}}^{k} \rangle
\end{equation}
is computed for $\mathcal{Q}_t$ at $x_{n_{t-1}^{k}}^{k}$ (see the ISDDP algorithm below for the computation of $\theta_t^{k}, \eta_t^k ( \varepsilon_t^k), \beta_t^{k}$).
%To alleviate notation, we will write $x_{t-1}^k := x_{n_{t-1}^{k}}^{k}$.
At the end of iteration $k$, we obtain the polyhedral lower approximations $\mathcal{Q}_{t}^{k}$ of $\mathcal{Q}_t,\;t=2,\ldots,T+1$, 
given by
$$
\begin{array}{lll}
\mathcal{Q}_{t}^{k}(x_{t-1}) & = &\displaystyle \max_{0 \leq \ell \leq k}\;\mathcal{C}_t^{\ell}( x_{t-1} ).
\end{array}
$$ 
The detailed ISDDP algorithm is given below.\\

\par {\bf ISDDP (Inexact Stochastic Dual Dynamic Programming for multistage stochastic nonlinear programs).}
\begin{itemize}
\item[Step 1)] {\textbf{Initialization.}} For $t=2,\ldots,T$,  take as initial approximations $\mathcal{Q}_t^0 \equiv -\infty$.  
Set $x_{n_0}^0 = x_0$, set the iteration count $k$ to 1, and $\mathcal{Q}_{T+1}^0 \equiv 0$. 
\item[Step 2)] {\textbf{Forward pass.}} \\
{\textbf{For }}$t=1,\ldots,T$,\\
\hspace*{0.8cm}{\textbf{For }}every node $n$ of stage $t-1$,\\
\hspace*{1.6cm}{\textbf{For }}every child node $m$ of node $n$, compute an $\varepsilon_t^k$-optimal solution $x_m^k$ of
\begin{equation} \label{defxtkj}
{\underline{\mathfrak{Q}}}_t^{k-1}( x_n^k , \xi_m ) = \left\{
\begin{array}{l}
\displaystyle \inf_{x_m} \; F_t^{k-1}(x_m , x_n^k, \xi_m):= f_t( x_m , x_n^k , \xi_m ) + \mathcal{Q}_{t+1}^{k-1}( x_m ) \\
x_m \in X_t( x_n^k, \xi_m ),
\end{array}
\right.
\end{equation}
\hspace*{2.4cm}where $x_{n_0}^k = x_0$.\\
\hspace*{1.6cm}{\textbf{End For}}\\
\hspace*{0.8cm}{\textbf{End For}}\\
{\textbf{End For}}
\item[Step 3)] {\textbf{Backward pass.}}\\
Select a set of nodes $(n_1^k, n_2^k, \ldots, n_T^k)$ 
with $n_t^k$ a node of stage $t$ ($n_1^k=n_1$ and for $t \geq 2$, $n_t^k$
a child node of $n_{t-1}^k$)
corresponding to a sample $({\tilde \xi}_1^k, {\tilde \xi}_2^k,\ldots, {\tilde \xi}_T^k)$
of $(\xi_1, \xi_2,\ldots, \xi_T)$.\\
Set $\theta_{T+1}^k=0, \eta_{T+1}^k = 0$, and $\beta_{T+1}^k=0$.\\
{\textbf{For }}$t=T,\ldots,2$,\\
\hspace*{0.8cm}{\textbf{For }}every child node $m$ of $n=n_{t-1}^k$ \\
\hspace*{1.6cm}{\textbf{If}} $X_t$ is of type $S1$ compute an $\varepsilon_t^{k}$-optimal  solution $x_m^{B k}$ of
$$
{\underline{\mathfrak{Q}}}_t^k( x_n^k , \xi_m ) = \left\{
\begin{array}{l}
\displaystyle \inf_{x_m} \;F_t^k(x_m, x_n^k , \xi_m):=f_t( x_m , x_n^k , \xi_m) + \mathcal{Q}_{t+1}^k ( x_m )\\ 
x_m \in \mathcal{X}_t.\\
\end{array}
\right.
$$
\hspace*{1.6cm}Compute 
\begin{equation}\label{l1isddp}
\ell_{1 t}^{k m}(x_m^{B k}, x_n^k ) =
\left\{ 
\begin{array}{l}
\displaystyle \max_{x_m} \displaystyle  \langle  \nabla_{x_t} f_t ( x_m^{B k}, x_n^k, \xi_m ), x_m^{B k} - x_m  \rangle   + \mathcal{Q}_{t+1}^{k}( x_m^{B k} ) - \mathcal{Q}_{t+1}^{k}( x_m ) \\  
x_m \in \mathcal{X}_t,
\end{array}
\right.
\end{equation}
\hspace*{1.6cm}and coefficients
$$
\begin{array}{lcl}
\theta_{t}^{k m} &=& f_t ( x_m^{B k}, x_n^k, \xi_m ) + \mathcal{Q}_{t+1}^{k}( x_m^{B k} ),\\
\eta_{t}^{k m}(\varepsilon_t^k ) &= &\ell_{1 t}^{k m}(x_m^{B k}, x_n^k ),\\
\beta^{k m}&=&\nabla_{x_{t-1}} f_t ( x_m^{B k}, x_n^k, \xi_m ).
\end{array}
$$
\hspace*{1.6cm}{\textbf{Else if}} $X_t$ is of type $S2$ compute an $\varepsilon_t^k$-optimal solution $x_{m}^{B k}$ of \\
\begin{equation}\label{primalpbisddp}
{\underline{\mathfrak{Q}}}_t^k(x_n^k , \xi_m ) = \left\{
\begin{array}{l}
\displaystyle \inf_{x_m} \;F_t^k(x_m , x_n^k ,  \xi_m):=f_t( x_m , x_n^k , \xi_m) + \mathcal{Q}_{t+1}^k ( x_m )\\ 
x_m \in X_t(x_n^k , \xi_m   ).\\
\end{array}
\right.
\end{equation}
\hspace*{1.6cm}Compute an $\varepsilon_t^k$-optimal solution $(\lambda_m^k, \mu_m^k)$ of the dual problem
\begin{equation}\label{dualpbisddp}
\begin{array}{l}
\displaystyle \max_{\lambda, \mu, x_m}  h_{t, x_n^k}^{k m}( \lambda , \mu )\\
\lambda = A_m x_m + B_m x_n^k  - b_m,\;x_m \in \mbox{Aff}( \mathcal{X}_t ), \; \mu \geq 0,
\end{array}
\end{equation}
\hspace*{1.6cm}where the dual function $h_{t, x_n^k}^{k m}$ is given by
$$
h_{t, x_n^k}^{k m}( \lambda , \mu )=
\left\{ 
\begin{array}{l}
\displaystyle \inf_{x_m} F_t^k ( x_m, x_n^k , \xi_m ) + \langle  \lambda , A_m x_m + B_m x_n^k -b_m \rangle + \langle \mu , g_t ( x_m , x_n^k , \xi_m ) \rangle \\
x_m \in \mathcal{X}_t.
\end{array}
\right.
$$
\hspace*{1.6cm}Compute the optimal value $\ell_{2 t}^{k m}(x_m^{B k}, x_n^k , \lambda_m^k , \mu_m^k, \xi_m )$ of the optimization problem\footnote{Observe that this is a linear program if $\mathcal{X}_t$ is polyhedral.}
\begin{equation}\label{l2isddp}
{\small{
\begin{array}{l}
\displaystyle \max_{x_m \in \mathcal{X}_t} \displaystyle  \langle \nabla_{x_t} f_t ( x_m^{B k}, x_n^k, \xi_m )  +  A_m^T \lambda_m^k + \sum_{i=1}^p \mu_m^{k}(i) \nabla_{x_t} g_{t i}(x_m^{B k}, x_n^k, \xi_m ) , x_m^{B k} - x_m  \rangle  + \mathcal{Q}_{t+1}^{k}( x_m^{B k} ) - \mathcal{Q}_{t+1}^{k}( x_m ),
\end{array}
}}
\end{equation}
\hspace*{1.6cm}and coefficients
$$
\begin{array}{lcl}
\theta_{t}^{k m}& =& f_t ( x_m^{B k}, x_n^k, \xi_m ) +  \mathcal{Q}_{t+1}^k ( x_m^{B k} )  + \langle \mu_m^k ,  g_{t} ( x_m^{B k} , x_n^k , \xi_m ) \rangle,\\
\eta_{t}^{k m}(\varepsilon_t^k )& =& \ell_{2 t}^{k m}(x_m^{B k}, x_n^k , \lambda_m^k , \mu_m^k, \xi_m ) ,\\
\beta^{k m}&=&\nabla_{x_{t-1}} f_t ( x_m^{B k}, x_n^k, \xi_m )+ B_m^T \lambda_m^k +  \sum_{i=1}^p \mu_m^k(i) \nabla_{x_{t-1}} g_{t i} ( x_m^{B k} , x_n^k , \xi_m ) .
\end{array}
$$
\hspace*{1.6cm}{\textbf{End If}}\\
\hspace*{0.8cm}{\textbf{End For}}\\
\hspace*{0.8cm}The new cut $\mathcal{C}_t^k$ is obtained  computing
\begin{equation}\label{formulathetak}
\theta_t^k=\sum_{m \in C(n)} p_m \theta_{t}^{k m},\;\;\eta_t^{k}(\varepsilon_t^k ) = \sum_{m \in C(n)} p_m \eta_{t}^{k m}(\varepsilon_t^k ) ,\;\;\beta_t^k=\sum_{m \in C(n)} p_m  \beta^{k m}.
\end{equation}
{\textbf{End For}}
\item[Step 4)] Do $k \leftarrow k+1$ and go to Step 2).\\
\end{itemize}

Observe that, as in IDDP, it is assumed that for ISDDP, 
nonlinear optimization problems  are solved approximately
whereas linear optimization problems are solved exactly. 
Since in ISDDP we compute the optimal value $\ell_{1 t}^{k m}(x_m^{B k}, x_n^k )$
of optimization problem \eqref{l1isddp} and the optimal value
$\ell_{2 t}^{k m}(x_m^{B k}, x_n^k , \lambda_m^k , \mu_m^k, \xi_m )$
of optimization problem \eqref{l2isddp}, it is assumed that these problems are linear. Since these optimization problems have a linear objective function, they are linear programs if and only if 
$\mathcal{X}_t$ is polyhedral. If this is not the case then
(a) either we add components  to $g$ 
pushing
the nonlinear constraints in the representation of $\mathcal{X}_t$ in $g$ or
(b) we also solve \eqref{l1isddp} and \eqref{l2isddp} approximately.  
In Case (b), we can still build an inexact cut $\mathcal{C}_t^k$ (see Proposition \ref{fixedprop2} and Remark \ref{remextincutvar}) and study the 
convergence of the corresponding variant of ISDDP along the lines of Section \ref{sec:convsddp}.

\subsection{Convergence analysis}\label{sec:convsddp}

Similarly to the deterministic case, we can easily check that functions $\mathcal{Q}_t$ are Lipschitz continuous
on $\mathcal{X}_{t-1}$: 
\begin{lemma} \label{convrecfuncQtS} Let Assumptions (Sto-H0) and (Sto-H1) hold. Then for $t=2,\ldots,T+1$, function $\mathcal{Q}_t$ is convex and Lipschitz continuous on 
$\mathcal{X}_{t-1}$.
\end{lemma}
\begin{proof} The proof is analogous to the proof of Lemma \ref{convrecfuncQt} (by backward induction on $t$, noting that
the fact that $\mathfrak{Q}_t(\cdot, \xi_{t j})$ is convex Lipschitz continuous can be justified using the arguments that
have shown this property for $\mathcal{Q}_t$ in Lemma \ref{convrecfuncQt} and since 
$\mathcal{Q}_t(\cdot)=\mathbb{E}_{\xi_t}[\mathfrak{Q}_t ( \cdot, \xi_t ) ]=\sum_{j=1}^M \mathbb{P}(\xi_t = \xi_{t j}) \mathfrak{Q}_t ( \cdot, \xi_{t j} )$,
convexity and Lipschitz continuity of $\mathcal{Q}_t$ on $\mathcal{X}_{t-1}$ follow). \hfill 
\end{proof}

In Proposition \ref{propboundeddualsto}, we show that the cut coefficients and approximate dual solutions computed in the backward passes are almost surely bounded
with the following additional assumption:\\

\par (Sto-H2) For $t=2, \ldots, T$, there exists $\kappa_t>0, r_t>0$ such that for every $x_{t-1} \in \mathcal{X}_{t-1}$,
for every $j=1,\ldots,M$,
there exists $x_t \in \mathcal{X}_t$ such that $\mathbb{B}(x_t, r_t) \cap \mbox{Aff}( \mathcal{X}_t ) \neq \emptyset$,
$A_{t j} x_t + B_{t j} x_{t-1}=b_{t j}$, and for every $i=1,\ldots,p$, $g_{t i}( x_t, x_{t-1}, \xi_{t j}) \leq -\kappa_t$. \\
\begin{prop}\label{propboundeddualsto} Assume that noises $(\varepsilon_t^k)_{k \geq 1}$ are bounded: for 
$t=1,\ldots,T$, we have $0 \leq \varepsilon_t^k \leq {\bar \varepsilon}_t< +\infty$.
If Assumptions (Sto-H0), (Sto-H1), and (Sto-H2) hold then the sequences 
$(\theta_{t}^k)_{t ,k}$, $(\eta_{t}^k (  \varepsilon_{t}^k ))_{t, k}$, $( \beta_t^k )_{t, k}$, 
$(\lambda_{m}^k )_{m, k}$, $( \mu_{m}^k )_{m, k}$
 generated by the ISDDP algorithm are almost surely bounded: for $t=2,\ldots,T+1$, there exists
 a compact set $C_t$ such that the sequence  
 $(\theta_{t}^k , \eta_{t}^k (  \varepsilon_{t}^k ), \beta_t^k )_{k \geq 1}$ almost surely belongs to 
 $C_t$ and for every $t=2,\ldots,T$, 
 if $X_t$ is of type $S2$ then for every $m \in {\tt{Nodes}}(t)$,
 there exists
 a compact set $\mathcal{D}_m$ such that the sequence  
 $( \lambda_{m}^k , \mu_m^k )_{k \geq 1}$ almost surely belongs to 
 $\mathcal{D}_m$.
\end{prop}
\begin{proof} The proof is analogous to the proof of Proposition \ref{propboundeddual}.\hfill
\end{proof}

We will assume that the sampling procedure in ISDDP satisfies the following property:\\

\par (Sto-H3) The samples in the backward passes are independent: $(\tilde \xi_2^k, \ldots, \tilde \xi_T^k)$ is a realization of
$\xi^k=(\xi_2^k, \ldots, \xi_T^k) \sim (\xi_2, \ldots,\xi_T)$ 
and $\xi^1, \xi^2,\ldots,$ are independent.\\

We can now study the convergence of ISDDP:
\begin{thm}[Convergence of ISDDP for multistage stochastic convex nonlinear programs]\label{convanisddp}
Consider the sequences of stochastic decisions $x_n^k$ and of recourse functions $\mathcal{Q}_ t^k$
generated by ISDDP.
Let Assumptions (Sto-H1), (Sto-H2), and (Sto-H3) hold and assume that for $t=1,\ldots,T$, we have
$\lim_{k \rightarrow +\infty} \varepsilon_{t}^k =0$. Then
\begin{itemize}
\item[(i)] almost surely, for $t=2,\ldots,T+1$, the following holds:
$$
\mathcal{H}(t): \;\;\;\forall n \in {\tt{Nodes}}(t-1), \;\; \displaystyle \lim_{k \rightarrow +\infty} \mathcal{Q}_{t}(x_{n}^{k})-\mathcal{Q}_{t}^{k}(x_{n}^{k} )=0.
$$
\item[(ii)]
Almost surely, the limit of the sequence
$( {F}_1^{k-1}(x_{n_1}^k , x_0 ,  \xi_1) )_k$ of the approximate first stage optimal values
and of the sequence
$({\underline{\mathfrak{Q}}}_1^{k}(x_{0}, \xi_1))_k$
is the optimal value 
$\mathcal{Q}_{1}(x_0)$ of \eqref{pbtosolve}.
Let $\Omega=(\Theta_2 \small{\times} \ldots \small{\times} \Theta_T)^{\infty}$ be the sample space
of all possible sequences of scenarios equipped with the product $\mathbb{P}$ of the corresponding 
probability measures. Define on $\Omega$ the random variable 
$x^* = (x_1^*, \ldots, x_T^*)$ as follows. For $\omega \in \Omega$, consider the 
corresponding sequence of decisions $( (x_n^k( \omega ))_{n \in \mathcal{N}} )_{k \geq 1}$
computed by ISDDP. Take any accumulation point
$(x_n^* (\omega) )_{n \in \mathcal{N}}$ of this sequence. 
If $\mathcal{Z}_t$ is the set of $\mathcal{F}_t$-measurable functions,
define $x_1^*(\omega),\ldots,x_T^*(\omega)$ taking $x_t^{*}(\omega): \mathcal{Z}_t \rightarrow \mathbb{R}^n$ given by
$x_t^{*}(\omega)( \xi_1, \ldots, \xi_t  )=x_{m}^{*}(\omega)$ where $m$ is given by $\xi_{[m]}=(\xi_1,\ldots,\xi_t)$ for $t=1,\ldots,T$.
Then $\mathbb{P}((x_1^*,\ldots,x_T^*) \mbox{ is an optimal solution to \eqref{pbtosolve}})  =1$.
\end{itemize}
\end{thm}
\begin{proof} Let $\Omega_1$ be the event on the sample space $\Omega$  of sequences
of scenarios such that every scenario is sampled an infinite number of times.
Due to (Sto-H3), this event has probability one.
Take an arbitrary  realization $\omega$ of ISDDP in $\Omega_1$. 
To simplify notation we will use $x_n^k, \mathcal{Q}_t^k, \theta_t^k, \eta_t^k(\varepsilon_t^k), \beta_t^k, \lambda_m^k, \mu_m^k$ instead 
of $x_n^k(\omega), \mathcal{Q}_t^k(\omega), \theta_t^k(\omega), \eta_t^k(\varepsilon_t^k)(\omega)$, $\beta_t^k( \omega ), \lambda_m^k( \omega ), \mu_m^k ( \omega )$.\\

\par Let us prove (i). 
We want to show that $\mathcal{H}(t), t=2,\ldots,T+1$, hold for that realization.
The proof is by backward induction on $t$. For $t=T+1$, $\mathcal{H}(t)$ holds
by definition of $\mathcal{Q}_{T+1}$, $\mathcal{Q}_{T+1}^k$. Now assume that $\mathcal{H}(t+1)$ holds
for some $t \in \{2,\ldots,T\}$. We want to show that $\mathcal{H}(t)$ holds.
Take an arbitrary node $n \in {\tt{Nodes}}(t-1)$. For this node we define 
$\mathcal{S}_n=\{k \geq 1 : n_{t-1}^k = n\}$ the set of iterations such that the sampled scenario passes through node $n$.
Observe that $\mathcal{S}_n$ is infinite because the realization of ISDDP is in $\Omega_1$.
We first show that 
$$
\displaystyle \lim_{k \rightarrow +\infty, k \in \mathcal{S}_n } \mathcal{Q}_{t}(x_{n}^{k})-\mathcal{Q}_{t}^{k}(x_{n}^{k} )=0. 
$$
For $k \in \mathcal{S}_n$, we have $n_{t-1}^k =n$, i.e., $x_n^k = x_{n_{t-1}^k}^k$, which implies
\begin{equation}\label{firsteqstosddp}
\mathcal{Q}_t ( x_n^k ) \geq \mathcal{Q}_t^k ( x_n^k ) \geq \mathcal{C}_t^k ( x_n^k ) = \theta_t^k - \eta_t^{k}(\varepsilon_t^k ) =  \sum_{m \in C(n)} p_m ( \theta_{t}^{k m} - \eta_{t}^{k m}(\varepsilon_t^k ) ). 
\end{equation}
Let us now bound $\theta_{t}^{k m}$ from below, considering two cases: $X_t$ is of type $S1$ (Case A) and $X_t$ is of type $S2$ (Case B).

In Case A we have $\theta_{t}^{k m} \geq {\underline{\mathfrak{Q}}}_t^k(x_n^k , \xi_m )$.
In Case B, observe that due to Assumption (Sto-H1)-(e), we can show (exactly as in the proof of Lemma \ref{boundcutcoeff}) that 
a Slater constraint qualification of form \eqref{slaterboundmulti} holds for primal problem
\eqref{primalpbisddp}  and therefore the optimal value  of dual problem   \eqref{dualpbisddp}
is the optimal value  ${\underline{\mathfrak{Q}}}_t^k(x_n^k , \xi_m )$ of primal problem \eqref{primalpbisddp}. Using the definition of $h_{t, x_n^k}^{k m}$
and the fact that $x_m^{B k} \in \mathcal{X}_t$ it follows that
$$
\theta_{t}^{k m} \geq h_{t, x_n^k}^{k m}( \lambda_m^k , \mu_m^k ) \geq 
{\underline{\mathfrak{Q}}}_t^k(x_n^k , \xi_m ) - \varepsilon_t^k.
$$
Next, we have the following lower bound on ${\underline{\mathfrak{Q}}}_t^k(x_n^k , \xi_m )$ for all $k \in \mathcal{S}_n$:
\begin{equation} \label{eqconv1bis0s}
\begin{array}{lcl}
{\underline{\mathfrak{Q}}}_t^k(x_n^k , \xi_m ) & \geq & {\underline{\mathfrak{Q}}}_t^{k-1}(x_n^k , \xi_m )  \mbox{ by monotonicity,}\\
& \geq  &  F_t^{k-1}(x_m^{k}, x_{n}^k, \xi_m ) -  \varepsilon_t^k    \mbox{ by definiton of } x_m^k,\\
& = &  f_t(x_m^{k}, x_{n}^k, \xi_m ) + \mathcal{Q}_{t+1}^{k-1}(x_m^{k}) -  \varepsilon_t^k  \mbox{ by definition of }F_t^{k-1},\\
& = &  F_t(x_m^{k}, x_{n}^k , \xi_m) + \mathcal{Q}_{t+1}^{k-1}(x_m^{k}) - \mathcal{Q}_{t+1}( x_m^{k})    -  \varepsilon_t^k   \mbox{ by definition of }F_t,\\
& \geq &  \mathfrak{Q}_t(x_{n}^k , \xi_m )+ \mathcal{Q}_{t+1}^{k-1}(x_m^{k}) - \mathcal{Q}_{t+1}( x_m^{k})  -  \varepsilon_t^k, 
\end{array}
\end{equation}
where for the last inequality we have used the definition of $\mathfrak{Q}_t$  and the fact that $x_m^k \in X_t ( x_n^k , \xi_m )$.

Combining \eqref{firsteqstosddp} with \eqref{eqconv1bis0s} and using our lower bound on $\theta_t^{k m}$, we obtain
\begin{equation} \label{eqconv1bisfutures}
\left\{
\begin{array}{l}
0 \leq \mathcal{Q}_{t}(x_{n}^k) - \mathcal{Q}_{t}^k(x_{n}^k)  \leq 
\varepsilon_t^k + \displaystyle \sum_{m \in C(n)} p_m \eta_t^{k m}( \varepsilon_t^ k ) + \displaystyle \sum_{m \in C(n)} p_m \Big(  \mathcal{Q}_{t+1}( x_m^k ) - \mathcal{Q}_{t+1}^{k-1}( x_m^k ) \Big)\\
\mbox{ if }X_t\mbox{ is of }\mbox{ type }S1 \mbox{ and }\\
0 \leq \mathcal{Q}_{t}(x_{n}^k) - \mathcal{Q}_{t}^k(x_{n}^k)  \leq  2 \varepsilon_t^k + \displaystyle \sum_{m \in C(n)} p_m \eta_t^{k m}( \varepsilon_t^ k ) + \displaystyle \sum_{m \in C(n)} p_m \Big(  \mathcal{Q}_{t+1}( x_m^k ) - \mathcal{Q}_{t+1}^{k-1}( x_m^k ) \Big)\\
\mbox{if }X_t\mbox{ is of }\mbox{ type }S2.
\end{array}
\right.
\end{equation}
We now show that for every $m \in C(n)$, we have 
\begin{equation}\label{noisesvanishisddp}
\lim_{k \rightarrow +\infty, k \in \mathcal{S}_n} \eta_t^{k m}( \varepsilon_t^ k ) = 0. 
\end{equation}
Let us fix $m \in C(n)$. We consider two cases: $X_t$ is of type $S1$ (Case A) and $X_t$ is of type $S2$ (Case B).\\

\par {\textbf{Case A.}} We have that $x_m^{B k}$ is an $\varepsilon_t^k$-optimal solution of
\begin{equation} \label{defxtkjisddp1}
\left\{
\begin{array}{l}
\displaystyle \inf_{x_m} \; F_t^{k}(x_m , x_n^k, \xi_m):= f_t( x_m , x_n^k , \xi_m ) + \mathcal{Q}_{t+1}^{k}( x_m ) \\
x_m \in \mathcal{X}_t,
\end{array}
\right.
\end{equation}
and $\eta_t^{k m}( \varepsilon_t^ k )$ is the optimal value of the following optimization problem:
\begin{equation} \label{probetakmsto1}
\left\{ 
\begin{array}{l}
\displaystyle \max_{x_m} \displaystyle  \langle  \nabla_{x_t} f_t ( x_m^{B k}, x_n^k, \xi_m ), x_m^{B k} - x_m  \rangle   + \mathcal{Q}_{t+1}^{k}( x_m^{B k} ) - \mathcal{Q}_{t+1}^{k}( x_m ) \\  
x_m \in \mathcal{X}_t.
\end{array}
\right.
\end{equation}
We now check that Proposition \ref{propvanish1} can be applied to problems \eqref{defxtkjisddp1}, \eqref{probetakmsto1} setting:
\begin{itemize}
\item $Y=\mathcal{X}_t, X=\mathcal{X}_{t-1}$ which are nonempty, compact, and convex;
\item $f(y, x)=f_t(y,x,\xi_m)$ which is convex and continuously differentiable on $Y \small{\times} X$;
\item $\mathcal{Q}^k=\mathcal{Q}_{t+1}^k$ which is convex Lipschitz continuous on $Y$ with Lipschitz constant $L_{t+1}$
($L_{t+1}$ is an upper bound on 
$(\|\beta_{t+1}^k \|)_{k \in \mathcal{S}_n}$, see Proposition  \ref{propboundeddualsto}) and satisfies
$$
{\underline{Q}} := \mathcal{Q}_{t+1}^1 \leq \mathcal{Q}^k \leq {\bar{\mathcal{Q}}}:=\mathcal{Q}_{t+1}
$$
on $Y$ with ${\underline{Q}}, {\bar{\mathcal{Q}}}$ continuous on $Y$;
\item $(x^k)_{k \in \mathcal{S}_n} = (x_{n}^k)_{k \in \mathcal{S}_n}$ sequence in $X$ and 
$(y^k)_{k \in \mathcal{S}_n} = (x_m^{B k})_{k \in \mathcal{S}_n}$ sequence in $Y$.
\end{itemize}
Therefore we can apply  Proposition \ref{propvanish1} to obtain \eqref{noisesvanishisddp}.\\

\par {\textbf{Case B.}} Here $x_m^{B k}$ is an $\varepsilon_t^k$-optimal solution of
\begin{equation} \label{defxtkjisddp2}
\begin{array}{l}
\left\{
\begin{array}{l}
\displaystyle \inf_{x_m} \;f_t( x_m , x_n^k , \xi_m) + \mathcal{Q}_{t+1}^k ( x_m )\\ 
x_m \in X_t(x_n^k , \xi_m   ),\\
\end{array}
\right.
\end{array}
\end{equation}
and $\eta_t^{k m}( \varepsilon_t^ k )$ is the optimal value of the following optimization problem:
\begin{equation} \label{probetakmsto2}
\begin{array}{l}
{\small{
\begin{array}{l}
\displaystyle \max_{x_m \in \mathcal{X}_t} \displaystyle  \langle \nabla_{x_t} f_t ( x_m^{B k}, x_n^k, \xi_m )  +  A_m^T \lambda_m^k + \sum_{i=1}^p \mu_m^{k}(i) \nabla_{x_t} g_{t i}(x_m^{B k}, x_n^k, \xi_m ) , x_m^{B k} - x_m  \rangle  + \mathcal{Q}_{t+1}^{k}( x_m^{B k} ) - \mathcal{Q}_{t+1}^{k}( x_m ).
\end{array}
}}
\end{array}
\end{equation}
We now check that Proposition \ref{propvanish1dual} can be applied to problems \eqref{defxtkjisddp2}, \eqref{probetakmsto2} setting:
\begin{itemize}
\item $Y=\mathcal{X}_t, X=\mathcal{X}_{t-1}$ which are nonempty compact, and convex;
\item $f(y, x)=f_t(y,x,\xi_m)$ which is convex and continuously differentiable on $Y \small{\times} X$;
\item $g(y,x)=g_t(y,x,\xi_m) \in \mathcal{C}^1( Y \small{\times} X)$ with components $g_i,i=1,\ldots,p$, convex
on $Y \small{\times}X^{\varepsilon}$;
\item $\mathcal{Q}^k=\mathcal{Q}_{t+1}^k$ which is convex Lipschitz continuous on $Y$ with Lipschitz constant $L_{t+1}$
($L_{t+1}$ is an upper bound on 
$(\|\beta_{t+1}^k \|)_{k \in \mathcal{S}_n}$, see Proposition  \ref{propboundeddualsto}) and satisfies
$$
{\underline{Q}} := \mathcal{Q}_{t+1}^1 \leq \mathcal{Q}^k \leq {\bar{\mathcal{Q}}}:=\mathcal{Q}_{t+1}
$$
on $Y$ with ${\underline{Q}}, {\bar{\mathcal{Q}}}$ continuous on $Y$;
\item $(x^k ) = (x_{n}^k)_{k \in \mathcal{S}_n}$ sequence in $X$, $(\lambda^k, \mu^k )_{k \in \mathcal{S}_n} = (\lambda_m^k , \mu_m^k )_{k \in \mathcal{S}_n}$, and 
$(y^k)_{k \in \mathcal{S}_n} = (x_m^{B k})_{k \in \mathcal{S}_n}$ sequence in $Y$. 
\end{itemize}
With this notation Assumption (H) is satisfied with $\kappa=\kappa_t$, since Assumption (H2) holds.
Therefore we can apply Proposition \ref{propvanish1dual} to obtain \eqref{noisesvanishisddp}.\\

It follows that \eqref{noisesvanishisddp} holds for every $m \in C(n)$ both when $X_t$ is of type $S1$ and of type $S2$.\\

Next, recall that $\mathcal{Q}_{t+1}$ is convex; functions $(\mathcal{Q}_{t+1}^k)_k$ are $L_{t+1}$-Lipschitz;
and for all $k \geq 1$ we have $\mathcal{Q}_{t+1}^k \leq \mathcal{Q}_{t+1}^{k+1} \leq\mathcal{Q}_{t+1}$ on compact set $\mathcal{X}_t$.
Therefore, the induction hypothesis
$$
\lim_{k \rightarrow +\infty} \mathcal{Q}_{t+1}( x_m^k ) - \mathcal{Q}_{t+1}^k ( x_m^k )=0 
$$
implies, using Lemma A.1 in \cite{lecphilgirar12}, that
\begin{equation}\label{indchypreformulatedsto}
\lim_{k \rightarrow +\infty} \mathcal{Q}_{t+1}( x_m^k ) - \mathcal{Q}_{t+1}^{k-1} ( x_m^k )=0 . 
\end{equation}

Plugging \eqref{noisesvanishisddp} and \eqref{indchypreformulatedsto} into
\eqref{eqconv1bisfutures} we obtain 
\begin{equation}\label{mainisddp}
\displaystyle \lim_{k \rightarrow +\infty, k \in \mathcal{S}_n } \mathcal{Q}_{t}(x_{n}^{k})-\mathcal{Q}_{t}^{k}(x_{n}^{k} )=0. 
\end{equation}
It remains to show that
\begin{equation}\label{alsoconvnotinsnisddp}
\displaystyle \lim_{k \rightarrow +\infty, k \notin \mathcal{S}_n } \mathcal{Q}_{t}(x_{n}^{k})-\mathcal{Q}_{t}^{k}(x_{n}^{k} )=0. 
\end{equation}
The relation above can be proved using Lemma 5.4 in \cite{guilejtekregsddp} which can be applied since 
(A) relation \eqref{mainisddp} holds (convergence was shown for the iterations in $\mathcal{S}_n$),
(B) the sequence $(\mathcal{Q}_t^k)_k$ is monotone, i.e., 
$\mathcal{Q}_t^k \geq \mathcal{Q}_t^{k-1}$ for all $k \geq 1$, (C) Assumption (Sto-H3) holds, and
(D) $\xi_{t-1}^k$ is independent on $( (x_{n}^j,j=1,\ldots,k), (\mathcal{Q}_{t}^j,j=1,\ldots,k-1))$.\footnote{Lemma 5.4 in \cite{guilejtekregsddp} is similar to the end of the proof of Theorem 4.1 in \cite{guiguessiopt2016} and uses
the Strong Law of Large Numbers. This lemma itself applies the ideas of the end of the convergence proof of SDDP given in \cite{lecphilgirar12}, which
was given with a different (more general) sampling scheme in the backward pass.} Therefore, we have shown (i).

\par (ii) Recalling that the root node $n_0$ with decision 
$x_0$ taken at that node has a single child node $n_1$ with corresponding decision 
$x_{n_1}^k$ computed at iteration $k$, we have for every $k \geq 1$:
\begin{equation}\label{convoptvalfirst}
\begin{array}{lll}
0 \leq  \mathcal{Q}_1( x_0 ) - {\underline{\mathfrak{Q}}}_1^k(x_0, \xi_1 )  & \leq &  \mathcal{Q}_1( x_0 ) - {\underline{\mathfrak{Q}}}_1^{k-1}(x_0, \xi_1 ),\\
& \leq &  \mathcal{Q}_1( x_0 ) - F_1^{k-1}( x_{n_1}^k, x_0, \xi_1 ) + \varepsilon_1^k,\\
& = &  \mathcal{Q}_1( x_0 ) - f_1( x_{n_1}^k, x_0, \xi_1 )  -  \mathcal{Q}_2^{k-1}(  x_{n_1}^k ) + \varepsilon_1^k,\\
& = &  \mathcal{Q}_1( x_0 ) - F_1( x_{n_1}^k, x_0, \xi_1 ) + \mathcal{Q}_2 ( x_{n_1}^k ) -  \mathcal{Q}_2^{k-1}(  x_{n_1}^k ) + \varepsilon_1^k,\\
& \leq & \mathcal{Q}_2 ( x_{n_1}^k ) -  \mathcal{Q}_2^{k-1}(  x_{n_1}^k ) + \varepsilon_1^k.
\end{array}
\end{equation}
We have shown in (i) that 
\begin{equation}\label{deduce1}
\lim_{k \rightarrow +\infty} \mathcal{Q}_2 ( x_{n_1}^k ) -  \mathcal{Q}_2^{k}(  x_{n_1}^k )=0. 
\end{equation}
Since  $\mathcal{Q}_{2}$ is convex, functions $(\mathcal{Q}_{2}^k)_k$ are $L_{2}$-Lipschitz,
and for all $k \geq 1$ we have $\mathcal{Q}_{2}^k \leq \mathcal{Q}_{2}^{k+1} \leq\mathcal{Q}_{2}$ on compact set $\mathcal{X}_1$,
we can once again apply Lemma A.1 in \cite{lecphilgirar12}, to deduce from \eqref{deduce1} that
$\lim_{k \rightarrow +\infty} \mathcal{Q}_2 ( x_{n_1}^k ) -  \mathcal{Q}_2^{k-1}(  x_{n_1}^k )=0$, which, combined
with \eqref{convoptvalfirst}, gives 
$$
\lim_{k \rightarrow +\infty} {\underline{\mathfrak{Q}}}_1^k(x_0, \xi_1 ) = \lim_{k \rightarrow +\infty} F_1^{k-1}( x_{n_1}^k, x_0, \xi_1 )  = \mathcal{Q}_1( x_0 ).
$$

Now take an accumulation point $(x_n^*)_{n \in \mathcal{N}}$ 
of the sequence $((x_n^k)_{n \in \mathcal{N}})_{k \geq 1}$ and let
$K$ be an infinite set of iterations such that for every $n \in \mathcal{N}$, $\lim_{k \rightarrow +\infty, k \in K} x_n^k= x_n^*$.\footnote{The existence of an accumulation point comes from the fact that 
the decisions belong to a compact set.}
Combining inequalities \eqref{eqconv1bis0s} which hold for every $k \geq 1, t=2,\ldots,T$, with  \eqref{convoptvalfirst}, we get 
for every $t=1,\ldots,T$, for every $n \in {\tt{Nodes}}(t-1)$, for every $m \in C(n)$,
\begin{equation}
-\varepsilon_t^k \leq \mathfrak{Q}_t( x_n^k , \xi_m ) - F_t^{k-1}( x_m^k , x_n^k , \xi_m ) \leq \mathcal{Q}_{t+1}( x_m^k )- \mathcal{Q}_{t+1}^{k-1}( x_m^k ). 
\end{equation}
From (i) we have $\lim_{k \rightarrow +\infty} \mathcal{Q}_{t+1}( x_m^k )- \mathcal{Q}_{t+1}^{k-1}( x_m^k )=0$ which implies that 
for every $t=1,\ldots,T$, for every $n \in {\tt{Nodes}}(t-1)$, for every $m \in C(n)$,
\begin{equation}\label{convoptval}
\lim_{k \rightarrow +\infty} \mathfrak{Q}_t( x_n^k , \xi_m ) - F_t^{k-1}( x_m^k , x_{n}^k, \xi_m )=0.
\end{equation}
We will now use the continuity of $\mathfrak{Q}_t(\cdot, \xi_m)$ which follows from (Sto-H1) (see Lemma 3.2 in \cite{guiguessiopt2016} for a proof).
We have
\begin{equation}\label{optsolconv}
\begin{array}{lll}
\mathfrak{Q}_{t} ( x_{n}^* , \xi_m )  & = & \displaystyle  \lim_{k \rightarrow +\infty, k\in K} \mathfrak{Q}_{t} ( x_{n}^k , \xi_m ) \mbox{ using the continuity of }\mathfrak{Q}_t(\cdot, \xi_m),\\
& =& \displaystyle \lim_{k \rightarrow +\infty, k\in K} F_t^{k-1}( x_m^k , x_{n}^k, \xi_m ) \mbox{ using }\eqref{convoptval},\\
& = & \displaystyle \lim_{k \rightarrow +\infty, k\in K} f_t(x_m^k , x_{n}^k, \xi_m ) + \mathcal{Q}_{t+1}^{k-1}(x_m^k ),\\
& = & f_t(x_m^*, x_n^*, \xi_m) + \displaystyle \lim_{k \rightarrow +\infty, k\in K} \mathcal{Q}_{t+1}(x_m^k ) \mbox{ using (i) and continuity of }f_t,\\
& = &f_t(x_m^*, x_n^*, \xi_m) + \mathcal{Q}_{t+1}(x_m^* )=F_t(x_m^*, x_n^*, \xi_m)
\end{array}
\end{equation}
where for the last equality we have used the continuity of $\mathcal{Q}_{t+1}$.
To achieve the proof of (ii) it suffices to observe that the sequence $(x_m^k, x_n^k)_{k \in K}$ belongs
to the set 
$$
{\bar X}_{t, m}=\{(x_{t},x_{t-1}) \in \mathcal{X}_{t}  \small{\times} \mathcal{X}_{t-1} : g_t(x_{t},x_{t-1},\xi_m ) \leq 0, \,A_m x_t + B_m x_{t-1} = b_m\}
$$
and this set is closed since $g_t$ is lower semicontinuous and $\mathcal{X}_t$ is closed.
Therefore $x_m^* \in X_t(x_n^*, \xi_m)$, which, together with \eqref{optsolconv}, shows that
$x_m^*$ is an optimal solution of $\mathfrak{Q}_{t} ( x_{n}^* , \xi_m )=\inf\{F_t(x_m , x_n^* , \xi_m) \;:\;x_m \in X_t(x_n^*, \xi_m)\}$
and
completes the proof of (ii).
\hfill
\end{proof}

\begin{rem}\label{importantremark}
In ISDDP algorithm presented in Section \ref{sec:isddpalgo}, decisions are computed at every iteration for all the nodes of the scenario tree
in the forward pass.
However, in practice, at iteration $k$  decisions will only be computed for the nodes $(n_1^k,\ldots,n_T^k)$
and their children nodes. For this variant of ISDDP, the backward pass is exactly as the backward of ISDDP presented in Section \ref{sec:isddpalgo}
while the forward pass reads as follows:\\

\par {\textbf{Forward pass with sampling for ISDDP.}}\\

Select a set of nodes $(n_1^k, n_2^k, \ldots, n_T^k)$ 
with $n_t^k$ a node of stage $t$ ($n_1^k=n_1$ and for $t \geq 2$, $n_t^k$
a child node of $n_{t-1}^k$)
corresponding to a sample $({\tilde \xi}_1^k, {\tilde \xi}_2^k,\ldots, {\tilde \xi}_T^k)$
of $(\xi_1, \xi_2,\ldots, \xi_T)$.\\

{\textbf{For }}$t=1,\ldots,T$,\\
\hspace*{0.8cm}Setting $m=n_{t}^k$ and $n=n_{t-1}^k$, compute an $\varepsilon_t^k$-optimal solution $x_m^k$ of
\begin{equation} \label{defxtkj}
{\underline{\mathfrak{Q}}}_t^{k-1}( x_n^k , \xi_m ) = \left\{
\begin{array}{l}
\displaystyle \inf_{y} \; F_t^{k-1}(y , x_n^k, \xi_m):= f_t( y , x_n^k , \xi_m ) + \mathcal{Q}_{t+1}^{k-1}( y ) \\
y \in X_t( x_n^k, \xi_m ),
\end{array}
\right.
\end{equation}
\hspace*{0.8cm}where $x_{n_0}^k = x_0$.
\par {\textbf{End For}}\\

This variant of ISDDP will build the same cuts and compute the same decisions for the nodes of the
sampled scenarios as ISDDP described in Section \ref{sec:isddpalgo}. For this variant, for a node $n$, the decision variables $(x_n^k)_k$ are defined for
an infinite subset ${\tilde{\mathcal{S}}}_{n}$ of iterations where the sampled scenario passes through the parent node of node $n$, i.e., 
${\tilde{\mathcal{S}}}_{n}=\mathcal{S}_{\mathcal{P}(n)}$.
With this notation, for this variant, applying Theorem \ref{convanisddp}-(i), we get for $t=2,\ldots,T+1$,
\begin{equation}\label{variantalg1}
\mbox{for all }n \in {\tt{Nodes}}(t-1), \lim_{k \rightarrow +\infty, k \in \mathcal{S}_{\mathcal{P}(n)}} \mathcal{Q}_{t}(x_{n}^{k})-\mathcal{Q}_{t}^{k}(x_{n}^{k} )=0
\end{equation}
almost surely. Also almost surely, the limit of the sequence
$({F}_1^{k-1}(x_{n_1}^k , x_0 , \xi_1) )_k$ of the approximate first stage optimal values
is the optimal value 
$\mathcal{Q}_{1}(x_0)$ of \eqref{pbtosolve}. The variant of ISDDP without sampling in the forward pass was presented  first to allow for the application
of Lemma 5.4 from \cite{guilejtekregsddp}. More specifically, item (D): $\xi_{t-1}^k$ is independent on $( (x_{n}^j,j=1,\ldots,k), (\mathcal{Q}_{t}^j,j=1,\ldots,k-1))$, 
given in the end of the proof of (i) of Theorem \ref{convanisddp} does not apply for ISDDP with sampling in the forward pass.
\end{rem}

\section{Conclusion}

We have introduced the first inexact variants of DDP and SDDP to solve respectively nonlinear deterministic and stochastic dynamic programming
equations. We have shown that these methods solve the dynamic programming equations for vanishing noises.

This study opens the way to a series of interesting issues:
\begin{itemize}
\item[a)] For linear dynamic programming equations, inexact variants of DDP and SDDP can still be derived.
For these problems, inexact cuts can easily be obtained for the cost-to-go functions $\mathcal{Q}_t$ on the basis
of approximate dual solutions. Indeed, since the dual of a linear program is also a linear program, feasible dual solutions
provide valid cuts. It would be worth writing and testing on real-life applications modelled by multistage stochastic linear
programs the corresponding inexact variant of SDDP. Note that we have assumed in our analysis that linear programs can be solved
exactly. For this variant of ISDDP, inexactness would be "forced", by solving inexactly the subproblems in
the first iterations and stages and increasing the precision of the computed solutions as the algorithm progresses. 
This inexact variant of SDDP applied to MSLPs could well converge more quickly than exact SDDP on some instances for well chosen noises $\varepsilon_t^k$. 
\item[b)] For constraints of type $S1$, we can obtain simpler formulas for inexact cuts when the objective function $f_t$
is strongly convex jointly in $(x_t, x_{t-1})$. It would be interesting to compare the quality of these cuts with the inexact
cuts from Section \ref{fixedcsetcut}.
\item[c)] To derive inexact cuts for value function $\mathcal{Q}$ given by \eqref{vfunctionq}, we could rely on the strong convexity 
of the objective function and on the strong concavity of the dual function, when these assumptions are satisfied. Unfortunately, for the decomposition methods under consideration in this paper,
such tool cannot be used 
since the objectives of the problems solved in the backward passes involve a piecewise affine function $\mathcal{Q}_{t+1}^k$ and therefore 
the corresponding dual functions are not strongly concave. 
However, this technique can well be applied for two-stage stochastic nonlinear problems, coupled with, for instance, level methods.
We intend to pursue this idea in a forthcoming paper.
\item[d)] Finally, it would be interesting to implement IDDP and ISDDP on various instances of deterministic and stochastic nonlinear dynamic programming
equations using various strategies for noises $\varepsilon_t^k$.
\end{itemize}

\section*{Acknowledgments} The author's research was 
partially supported by an FGV grant, CNPq grant 307287/2013-0,
and FAPERJ grant E-26/201.599/2014. The author would like to thank Ren\'e Henrion and Arkadi Nemirovski
for useful discussions.

\section*{Appendix}

\begin{lemma} \label{optcondonediffonenotdiff}
Consider the optimization problem
\begin{equation}\label{pboptappendix}
\left\{
\begin{array}{l}
\min f_0(x) + f_1(x)\\
x \in X
\end{array}
\right.
\end{equation}
with $X \subset \mathbb{R}^n$ nonempty, closed, and convex,
$f_0:X  \rightarrow \mathbb{R}$ differentiable and convex and $f_1: X \rightarrow \mathbb{R}$ convex. Then 
$x_*$ is an optimal solution to \eqref{pboptappendix} if and only if for every $x \in X$ we have
$$
\langle \nabla_x f_0( x_* ) , x - x_* \rangle + f_1 ( x ) - f_1 ( x_* ) \geq 0.
$$
\end{lemma}

\addcontentsline{toc}{section}{References}
\bibliographystyle{plain}
\bibliography{Bibliography}

\end{document}